\documentclass{amsart}

\calclayout

\usepackage{amsxtra,amssymb,mathrsfs}
\usepackage{enumitem}
\setlist[enumerate]{font=\upshape}

\usepackage{asymptote}

\usepackage{tikz}
\usetikzlibrary{arrows}
\usepackage{tikz-cd}
\tikzcdset{diagrams={nodes={inner sep=3pt}}}

\usepackage[pagebackref]{hyperref}
\usepackage[dvipsnames]{xcolor}

\definecolor{ceruleanblue}{rgb}{0.16, 0.32, 0.75}
\hypersetup{colorlinks=true,linkcolor=ceruleanblue,citecolor=ceruleanblue,urlcolor=Plum}

\newtheorem{theorem}{Theorem}[section]
\newtheorem{lemma}[theorem]{Lemma}
\newtheorem{proposition}[theorem]{Proposition}
\newtheorem{prop-def}[theorem]{Proposition-Defintion}
\newtheorem{corollary}[theorem]{Corollary}

\theoremstyle{definition}
\newtheorem{definition}[theorem]{Definition}
\newtheorem{remark}[theorem]{Remark}
\newtheorem{remarks}[theorem]{Remarks}
\newtheorem{example}[theorem]{Example}

\usepackage{manfnt}

\makeatletter
 \newenvironment{caution}[1][]{
    \begin{trivlist} \item[] \noindent%
      \begingroup\hangindent=2pc\hangafter=-2
      \clubpenalty=10000%
      \hbox to0pt{\hskip-\hangindent\dbend\hfill}\ignorespaces%
      \refstepcounter{theorem}\textbf{Remark~\thetheorem}%
      \@ifnotempty{#1}{\the\thm@notefont \ (#1)}\textbf{.}
      \let\p@@r=\par \def\p@r{\p@@r \hangindent=0pc} \let\par=\p@r}%
    {\hspace*{\fill}
     \endgraf\endgroup\end{trivlist}}
\makeatother

\newcommand\abs[1]{\lvert#1\rvert}

\newcommand\biggabs[1]{\biggl\lvert#1\biggr\rvert}

\newcommand\inner[1]{\langle#1\rangle}
\newcommand\biginner[1]{\bigl\langle#1\bigr\rangle}
\newcommand\Biginner[1]{\Bigl\langle#1\Bigr\rangle}

\newcommand\longto{\longrightarrow} 

\newcommand\pardif[2]{\frac{\partial#1}{\partial#2}}

\newcommand{\Lie}{\mathscr{L}}
\newcommand\orbit{\mathscr{O}}

\DeclareMathOperator\ad{ad}
\DeclareMathOperator\Ad{Ad}

\newcommand\pr{\mathrm{pr}}
\DeclareMathOperator{\tot}{\mathbb{T}ot}

\DeclareMathOperator\Sym{Sym}
\DeclareMathOperator\Exp{Exp}
\DeclareMathOperator\Stab{\mathrm{Stab}}
\DeclareMathOperator\Alg{\mathrm{Alg}}

\newcommand{\bT}{\mathbb{T}}

\newcommand\X{\mathfrak{X}}

\newcommand\qu[1][\kern.3ex]{/\kern-.7ex/_{\kern-.4ex#1}}
\newcommand\bigqu[1][\,\,]{\big/\kern-.85ex\big/_{\!\!#1}}

\let\oldtocsection=\tocsection

\let\oldtocsubsection=\tocsubsection

\let\oldtocsubsubsection=\tocsubsubsection

\renewcommand{\tocsection}[2]{\bf\hspace{0em}\oldtocsection{#1}{#2}}
\renewcommand{\tocsubsection}[2]{\hspace{1em}\oldtocsubsection{#1}{#2}}
\renewcommand{\tocsubsubsection}[2]{\hspace{2em}\oldtocsubsubsection{#1}{#2}}

\newcommand\ca{\mathcal}
\newcommand\bu{{\scriptscriptstyle\bullet}}
\newcommand\ann{ann}
\newcommand{\R}{\mathbb{R}} 
\newcommand{\N}{\mathbb{N}} 

\DeclareMathOperator{\Hom}{Hom}

\DeclareMathOperator{\Kan}{Kan}     

\DeclareMathOperator*{\colim}{\mathrm{colim}}

\newcommand{\huaB}{\mathcal{B}}
\newcommand{\huaS}{\mathcal{S}}
\newcommand{\huaA}{\mathcal{A}}
\newcommand{\huaL}{\mathcal{L}}

\newcommand{\huaF}{\mathcal{F}}
\newcommand{\huaG}{\mathcal{G}}

\newcommand{\huaU}{\mathcal{U}}
\newcommand{\huaV}{\mathcal{V}}
\newcommand{\huaW}{\mathcal{W}}
\newcommand{\huaX}{\mathcal{X}}

\newcommand{\huaQ}{\mathcal{Q}}
\newcommand{\huaP}{\mathcal{P}}
\newcommand{\huaC}{\mathcal{C}}

\newcommand{\huaH}{\mathcal{H}}
\newcommand{\huaK}{\mathcal{K}}

\newcommand{\huaT}{\mathcal{T}}
\newcommand{\huaZ}{\mathcal{Z}}
\newcommand{\huaM}{\mathcal{M}}
\newcommand{\huaN}{\mathcal{N}}

\makeatletter
\def\namedlabel#1#2{\begingroup
    #2%
    \def\@currentlabel{#2}%
    \phantomsection\label{#1}\endgroup
}

\newcommand{\pt}{\mathrm{pt}}
\newcommand{\red}{\mathrm{red}}

\newcommand{\pts}{\mathsf{P}}
\newcommand{\ppt}{\mathsf{p}}

\newcommand\diag{\mathrm{diag}}    

\newcommand*{\Simpc}{\triangle}
\newcommand*{\Simp}[1]{{\Delta[#1]}}
\newcommand*{\Horn}[2]{{\Lambda^{#1}_{#2}}}
\newcommand*{\face}{\textsf{\textup d}}
\newcommand*{\de}{\textsf{\textup s}}
\newcommand\id{\mathrm{id}}

\DeclareMathOperator\sgn{sgn}
\DeclareMathOperator\Ksgn{Ksgn}

\newcommand{\Z}{\ensuremath{\mathbb Z}}

\newcommand{\Cat}{\mathsf{C}}
\newcommand{\Catd}{\mathsf{D}}
\newcommand{\Mfd}{\mathsf{Mfd}}
\newcommand{\DMfd}{\mathsf{DMfd}}
\newcommand{\covers}{\mathsf{T}}
\newcommand{\etale}{\textup{\'et}}

\newcommand{\Sh}{\mathsf{Sh}}
\newcommand{\PSh}{\mathsf{PSh}}

\newcommand{\arrow}[1]{\mathrm{\Theta(#1)}} 
\newcommand{\g}{\ensuremath{\mathfrak{g}}}
\newcommand{\h}{\ensuremath{\mathfrak{h}}}
\newcommand{\m}{\ensuremath{\mathfrak{m}}}

\newcommand{\maps}{\colon}
\newcommand{\xto}{\xrightarrow}
\newcommand{\yon}{\mathbf{y}}
\newcommand{\gpd}[1]{{\mathsf{Gpd}_\infty[#1]}}

\newcommand{\po}{\ar@{}[dr]|{\ulcorner}}
\newcommand{\pb}{\ar@{}[dr]|{\lrcorner}}

\makeatletter
\let\@wraptoccontribs\wraptoccontribs
\makeatother

\newcommand{\Set}{\mathsf{Set}}

\newcommand{\Cosk}{\mathsf{Cosk}}

\newcommand{\dCalg}{\mathsf{dC^\infty Alg}}

\newcommand\lsf{{\mathrm{lsf}}}
\newcommand\sfib{{\mathrm{sf}}}
\newcommand\open{{\mathrm{open}}}
\newcommand\sursub{\mathrm{ss}}
\newcommand\et{\mathrm{\acute{e}t}}
\newcommand\new{{\mathrm{new}}}

\newcommand\can{\mathrm{can}}

\newcommand\lin{\mathrm{lin}}

\newcommand\op{{\mathrm{op}}}

\newcommand{\dgngpds}{derived Lie $n$-groupoids}

\newcommand{\dnGpd}{\mathsf{Gpd}_n[\DMfd, \covers_\lsf]} 

\newcommand{\dgnGpdCscho}{\mathsf{hGpd}_n[\dCalg^\op_0, \covers_\new]}

\makeatletter
\@namedef{subjclassname@1991}{2020 Mathematics Subject Classification}
\makeatother

\title{The shifted symplectic geometry of derived higher groupoids}

\author{Miquel Cueca Ten}

\address{Departement of Mathematics, KU Leuven. Celestijnenlaan 200B, Leuven (Heverlee), B-3001,
Belgium}

\email{miquel.cuecaten@kuleuven.be}

\author{Florian Dorsch}

\address{Georg-August-Universit\"at G\"ottingen}

\email{florian.dorsch@mathematik.uni-goettingen.de}

\author{Reyer Sjamaar}

\address{Cornell University, Ithaca, New York 14853, USA}

\email{sjamaar@math.cornell.edu}

\author{Chenchang Zhu}

\address{Georg-August-Universit\"at G\"ottingen}

\email{chenchang.zhu@mathematik.uni-goettingen.de}

\date{\today}

\begin{document}

\begin{abstract} 
The main goal of this work is to introduce \emph{derived Lie $n$-groupoids} and their shifted symplectic structures. We further define shifted lagrangian structures and prove that their composition is well defined under suitable conditions. As an application, we show that our framework incorporates several reduction procedures at critical values, including: classical Hamiltonian reduction, group valued moment maps, Poisson Lie group valued moment maps and Mikami-Weinstein for proper symplectic groupoids. 
\end{abstract}

\keywords{Derived Lie \(n\)-groupoids, shifted symplectic structures, shifted Lagrangian structures, Lagrangian correspondences, derived manifolds, symplectic reduction, Morita equivalence.} 
\subjclass[2020]{53D05, 58H05, 53D12, 53D20, 58A50, 18N50.}    

\maketitle

\tableofcontents

\section{Introduction}

A central theme in symplectic geometry is Weinstein’s proposal \cite{wei:sympcat} that symplectic manifolds should form a category whose morphisms are given by lagrangian correspondences. A major obstacle to realizing this idea is that composition is  not well defined in general \cite{blandW:rel, wei:sympcat}. In derived algebraic geometry, see e.g. \cite{Calaque:14, cal:lag, ptvv}, this issue is resolved via a two-step construction. First, one introduces the notion of shifted symplectic derived stacks and their lagrangian morphisms, which generalizes symplectic varieties and classical lagrangian submanifolds. Secondly, ordinary fibre products are replaced by their homotopy-theoretic counterparts. As a result, one obtains a well-defined  higher category of iterated spans \cite{cal:AKSZ}. From the differential geometric point of view, this last step is analogous in spirit to the use of perturbations in symplectic geometry, where one replaces a non-transversal fiber product by a better one in which composition becomes well defined \cite{ww:func}.

Another fundamental aspect of symplectic geometry is the reduction procedure introduced in \cite{mw:red}.  For a Hamiltonian $G$-space $(M,\omega,\mu)$ it is shown that, when $0$ is a regular value of the moment map $\mu:M\to\g^*$ and the action is free and proper, the reduced space
\[M_\red=M\qu[0] G=\mu^{-1}(0)/G\]
is again a symplectic manifold. In order to perform the aforementioned reduction, first one forms a fiber product by passing to $\mu^{-1}(0)$ and then performs a quotient. The conditions of zero being a regular value and the action being free and proper guarantee that the quotient is a manifold. Nevertheless, one desires to perform this construction without further assumptions. 
It is well known that Lie $n$-groupoids allow us to consider the singularities coming from quotients \cite{milnor:uniI, mil:geo}. In order to additionally consider singularities arising from non-transversal fiber products, we need derived direction \cite{kos:homo, tate:homo}.

Derived smooth manifolds were introduced in \cite{spi:der} in order to obtain a category that is closed under taking arbitrary intersections. As observed in \cite{blx, carchedi:23},  derived manifolds admit a concrete presentation in terms of non-positively graded manifolds with a cohomological vector field, see \S \ref{sec:derman}. It was shown in \cite{blx} that the category of derived manifolds $\DMfd$ forms a Category of Fibrant Objects (CFO) in the sense of \cite{brown:1973}. It is shown in \cite{carchedi:23} that, after localization at weak equivalences,
this approach is equivalent to the $\infty$-categorical approach to derived
manifolds of Carchedi--Steffens \cite{car-stef:19}.  More broadly, the
foundation of derived higher stacks in the $\infty$-categorical
$C^\infty$-ring setting was developed in \cite{Steffens:cinfty}, while a
$2$-categorical truncation of derived differential geometry, in terms of
$d$-manifolds and $d$-orbifolds, was developed by Joyce
\cite{Joyce:dmanifolds}. Nevertheless, while the stack approach provides the natural conceptual framework, groupoid presentations are often more tractable and geometrically transparent for concrete applications in differential geometry and dynamical systems. We therefore develop in this article a theory of $n$-groupoid objects in $\DMfd$.

As we recall in \S \ref{sec:points}, given a category with a pretopology and a jointly conservative family of points $(\Cat, \covers, \pts)$ it was shown in  \cite{Rogers-Zhu:2016} 
that $\mathsf{Gpd}_n[\Cat, \covers]$, the category of n-groupoid objects on it,  forms an incomplete category of fibrant objects (iCFO). Our first main result shows that one can endow $\DMfd$ with a pretopology $\covers_\lsf$ whose covers are locally split fibrations, see \S \ref{subsection;topology}. 

A presheaf on derived manifolds is a $\covers_\lsf$-sheaf if and only if it is a $\covers_\open$-sheaf, that is, a sheaf for the usual open-cover topology. This makes $\covers_\lsf$ a natural choice.
Moreover, $\covers_\lsf$ is restrictive enough for the stalkwise construction of the iCFO in \cite{Rogers-Zhu:2016}, while still flexible enough to contain relevant geometric examples.   

\begin{theorem}[Thm.~\ref{thm:iCFO-n-details}]\label{thm:iCFO-n-intro}
Derived Lie $n$-groupoids  form an iCFO $\mathsf{Gpd}_n[\DMfd, \covers_\lsf]$ in which:
\begin{itemize}
\item the weak equivalences are the stalkwise weak equivalences; 

\item the fibrations are the Kan fibrations; 

\item the acyclic fibrations are hypercovers; and
\item the path object of a derived Lie $n$-groupoid $\huaG_\bu$ is $\huaG_\bu^I$ with 
\begin{equation*}
    \huaG^I_k = \Hom(\Simp{k}\times \Simp{1}, \huaG_\bu). 
\end{equation*}
\end{itemize}
\end{theorem}

This theorem gives the homotopical foundation for the rest of the paper. It provides a workable
differential-geometric model for higher derived smooth geometry. Notice that the above iCFO does not know of the CFO structure on derived manifolds. Hence, we generalize the notion of a \emph{Morita morphism between derived Lie $n$-groupoids}, see \S \ref{subsection;ME}, as being either a simplicial morphism which at each level is a weak equivalence in the CFO of $\DMfd$  or a weak equivalence
in $\mathsf{Gpd}_n[\DMfd, \covers_\lsf]$. In fact, to present derived higher stacks, $\mathsf{Gpd}_n[\DMfd, \covers_\lsf]$ is possibly not enough. Homotopy derived Lie $n$-hypergroupoids as in \cite{Pridham:higher-stack} might be necessary. Thus Morita equivalence in $\dnGpd$ should be zigzags of homotopy hypercovers travelling through homotopy derived Lie $n$-hypergroupoids. Nevertheless,  the theory might be significantly simplified if one uses Reedy fibrant objects \cite{getzler;private} (See Remark \ref{rmk:ME}).

Our second main achievement is to extend the shifted symplectic structures on Lie \(n\)-groupoids introduced in \cite{Lesdiablerets} (see also \cite{cueca-zhu}) to derived Lie \(n\)-groupoids, see \S \ref{sec:symp}. Our construction builds on adding a third direction to the simplicial de Rham complex. Similarly as in \cite{Lesdiablerets}, one uses an adjusted version of the Eilenberg-Zilber map to define an IM-pairing for an $m$-shifted $2$-form on a derived Lie $n$-groupoid, see \S\,\ref{subsection;IM}. The nondegeneracy is expressed by requiring the associated map from the tangent complex to the shifted cotangent complex to be a quasi-isomorphism. In \S\,\ref{shifted-lagrangian} we extend the notion of $m$-shifted lagrangian structures on Lie groupoid morphisms introduced in \cite{ABC:lag} to derived Lie $n$-groupoids and prove a composition theorem for shifted lagrangian correspondences, c.f.~\cite{cal:lag, ptvv}.

\begin{theorem} [See Thm.~\ref{thm:complag}]\label{thm:main}
Consider two $m$-shifted lagrangian correspondences
\[\begin{tikzcd}
	{\mathcal{G}^1_\bu}\ar[r,dashed,"{\mathcal{L}_\bu}"]& {\mathcal{G}^2_\bu}\ar[r,dashed,"{\mathcal{L}_\bu'}"]&{\mathcal{G}^3_\bu}
\end{tikzcd}\]
between $m$-shifted symplectic derived Lie $n$-groupoids. Under a levelwise transversality hypothesis and a mild  assumption, the composition
$$\huaL_\bu\times_{\huaG^2_\bu}\huaL'_\bu\colon\huaG^1_\bu\dashrightarrow\huaG^3_\bu$$
is an $m$-shifted lagrangian correspondence.
\end{theorem}

 The above result provides a smooth transversal version of \cite{wei:sympcat} in the shifted setting. The point is not to remove the obstruction to composition abstractly, but to control it geometrically. Rem.\ \ref{rmk:ptvv-thm5} explains how our result together with a fibrant replacement recovers the homotopy-pullback picture in \cite{cal:lag, ptvv}.

The third and last main accomplishment of this work is to provide explicit examples coming from various sorts of symplectic reductions. Recall that a quasi-symplectic groupoid $\Gamma_\bu$, as introduced in \cite{BCWZ, moxu}, is the same as a $1$-shifted symplectic Lie $1$-groupoid, see \cite{cueca-zhu, Lesdiablerets}, while a hamiltonian $\Gamma_\bu$-space $(M,\omega,\mu)$ is equivalent to a $1$-shifted lagrangian structure on the morphism $$\mu: \Gamma_\bu\ltimes M\to \Gamma_\bu,$$ where $\Gamma_\bu\ltimes M$ stands for the action groupoid, see \cite{ABC:lag, cal:lag}. Each pre-symplectic orbit $\orbit\subseteq \Gamma_0$ together with the inclusion provides an example of a hamiltonian $\Gamma_\bu$-space. Hence, for a hamiltonian $\Gamma_\bu$-space $(M,\omega,\mu)$ and an orbit $\orbit$ one should compute the reduced space as the homotopy fiber product of the two induced $1$-shifted lagrangian structures \cite{cal:lag}, i.e.\
\begin{equation}\label{eq:introred}
  (\Gamma_\bu\ltimes M)\times_{\Gamma_\bu}^h\Gamma_\bu|_\orbit.  
\end{equation}
Observe that transversality for the $1$-shifted lagrangian structures $\mu:\Gamma_\bu\ltimes M\to \Gamma_\bu$ and $i:\Gamma_\bu|_\orbit\to \Gamma_\bu$ is equivalent to the points in the orbit being regular values for the moment map $\mu$. Our main idea is to fibrantly replace the $1$-shifted lagrangian structure $$i:\Gamma_\bu|_\orbit\to \Gamma_\bu$$ by a derived Lie $1$-groupoid ensuring transversality. We achieve this by using linearization results of presymplectic orbits in quasi-symplectic groupoids \cite{amm, am:linpl, pmct1, pmct3, del:riemgpd}.

\begin{theorem}[Thm.~\ref{theorem;reduction}]
Let $(\Gamma_\bu,\Omega_\bu)$ be a $1$-shifted symplectic Lie $1$-groupoid and let $(M,\mu,\omega)$ be a hamiltonian $\Gamma_\bu$-space. 
Let $\orbit$ be a $\Gamma_\bu$-orbit in $\Gamma_0$ and suppose that $\Gamma_\bu$ is weakly symplectically linearizable at $\orbit$. Then the following  holds: 
\begin{enumerate}
\item The symplectic quotient $(M\qu[\orbit]\Gamma_\bu,\omega^\red_\bu)$ is a  $0$-shifted symplectic quasi-smooth Lie groupoid and  it gives the homotopy pullback \eqref{eq:introred}.
\item If $\mu$ is transverse to $\orbit$, then $M\qu[\orbit]\Gamma$ is symplectically Morita equivalent to the $0$-shifted symplectic Lie groupoid $\Gamma|_\orbit\ltimes\mu^{-1}(\orbit)$ of Thm.~\textup{\ref{theorem;regular-reduction}}.
\end{enumerate}
\end{theorem}

Therefore, when the moment map has critical values, Theorem \ref{theorem;reduction} gives a unified framework to deal with ordinary hamiltonian reduction \cite{mw:red} at the origin \S \ref{sec:hamred}, quasi-hamiltonian reduction \cite{amm} at the unit \S \ref{sec:quamat1}, Lu reduction for Poisson-Lie group actions \cite{lu:mom} at the unit \S \ref{subsection;lu} and Mikami–Weinstein reduction \cite{MiWe88} for proper symplectic groupoids \S \ref{sec:proper}.

Our work is an attempt to produce a concrete model for derived higher geometry that is friendly for differential geometers. As such, there are many questions that remain open and we plan to further investigate in the future. We introduced a triple complex encoding the de Rham cohomology of a derived Lie $n$-groupoid. We expect this complex to be invariant under Morita maps as this happens for Lie $n$-groupoids \cite{weiershausen2025}, see also \cite{taroyan;de-rham}. Another point to be clarified is the relation between derived Lie $n$-groupoids as introduced here and derived differentiable stacks as introduced in \cite{porta1, Pridham:higher-stack} using homotopy derived Lie $n$-hypergroupoids, see Rem.\ \ref{rmk:ME}. 

Theorem \ref{thm:main} still requires transversality, one can remove that hypothesis at the price of having no control on the fibrant replacement,  which is how this problem is dealt with in algebraic geometry \cite{cal:lag, ptvv}. In differential geometry perturbation theory is the usual way of dealing with the lack of transversality \cite{wehrheim-woodward:10}. It would be interesting to find a way of unifying these two approaches.

There are other  instances of reduction to which our work could be applied, including the reduction of Poisson manifolds by $2$-groups as in \cite{cz:supergeo,car:homo}, and the reduction of Courant algebroids and generalized complex structures as in \cite{bcg, Burz-Carvarlho:26}.

A natural next step is to place the Weinstein category produced by Theorem \ref{thm:main} in the context of classical field theory. The guiding picture is that one should obtain a symmetric monoidal functor from a suitable cobordism category to a shifted symplectic correspondence category, in the spirit of AKSZ constructions and their derived extensions
\cite{aksz,cal:AKSZ,saf:qua}. Closely related developments on the differential-geometric side include constructions of 2D-field theories where gluing is related to reduction and fusion of quasi-hamiltonian spaces \cite{amm, crooks-mayrand:MT,maiza-mayrand:tqft,maiza:multiplicative}. 

Finally, one would like to produce a quantization of shifted symplectic structures. Parts of this picture have already been developed in the algebraic and derived settings \cite{cal:shi,prid:shift-nq,saf:shifted-poisson,saf:geoqua}. For example, the prequantization of shifted symplectic structures on stacks such as $BG$ and its inertia stack leads to the study of string $2$-groups \cite{waldorf:string-conn,cueca-zhu} and equivariant Dixmier--Douady gerbes \cite{krep:equipre, LGX:pre, meinrenken} and hence to quantized invariants such as twisted $K$-theory \cite{FHT, meinrenken, LGTX}. More conceptually, one expects our category with shifted lagrangian correspondences to provide the classical geometric input for an extended TFT, while quantization should assign representation or $K$-theoretic invariants. It would be particularly interesting to understand whether the fusion and reduction procedures admit compatible quantum counterparts, and whether other examples naturally lead to string-theoretic or even TMF-related structures.

\subsection*{Acknowledgements}

We thank Daniel \'Alvarez, David Carchedi, Ezra Getzler, Eugene Lerman, Yiannis Loizides, Tony Pantev, Jonathan Pridham, Stefano Ronchi, Pelle Steffens, Bertrand To\"en, and Hao Xu for helpful discussions.

M.C. was partially supported by FWO grants 1249325N and  G014726N. F.D. and C.Z. are supported by DFG grant ZH 274/5-1 and RTG 2491.  M.C., F.D. thank the Centre de Recerca Matem\`atica (CRM) Barcelona for the research pair program, and C.Z.\ thanks CRM for the Inaugural CRM-María de Maeztu Chair of Excellence program. 

\section{Preliminaries} 

This section is a review of pretopologies and related notions, including that of  higher groupoids in a category equipped with a pretopology.  This material is based on \cite{Duskin,Glenn:Realization,henriques}.  We also review the notion of an incomplete category of fibrant objects introduced in \cite{Rogers-Zhu:2016}.

\subsection{Pretopologies on categories}\label{sec:pre-points}

A (\emph{singleton}) \emph{pretopology}  on a category $\Cat$ is a collection $\covers$ of arrows, called \emph{covers}, with the following properties:
\begin{itemize}
\item[\namedlabel{P1}{(P1)}] isomorphisms are covers;
\item[\namedlabel{P2}{(P2)}] the composition of two covers is a cover;
\item[\namedlabel{P3}{(P3)}] pullbacks of covers are covers; more precisely, for every cover \(U\to X\) and every morphism \(Y\to X\), the pullback (fibred product) \(Y\times_X U\) exists in $\Cat$ and the canonical morphism \(Y\times_X U \to Y\) is a cover.
\end{itemize}
In addition we will assume that a pretopology satisfies the following axiom:
\begin{itemize}
\item[\namedlabel{P4}{(P4)}] products of covers are covers; that is, for all covers $f\colon U\to X$ and $g\colon V\to Y $, the product $f\times g\colon U\times V\to X\times Y$ exists and is a cover. 
\end{itemize}

\begin{remark}\label{rmk:p4}
This is the version of the pretopology axioms given in \cite[Def.\ 2.1]{henriques} except that, instead of axiom \ref{P4}, \cite[Assumptions 2.2]{henriques} imposes the following axiom:
\begin{itemize}
\item[\namedlabel{P4'}{(P4)$'$}] $\Cat$ has a terminal object $*$ and for every object \(X\) of \(\Cat\) the unique morphism \(X\to *\) is a cover.
\end{itemize}
This axiom is stronger than \ref{P4} (see \cite[Lemma~2.9]{meyer-zhu}), but it fails in many contexts.  In particular it fails for derived manifolds, whereas the weaker axiom \ref{P4} holds (see Rem.~\ref{remark;terminal}). We shall frequently rely on \cite{Rogers-Zhu:2016}, which also imposes axiom \ref{P4'}.  None of the results of \cite{Rogers-Zhu:2016} appeal to \ref{P4'}, except for two: 
\begin{enumerate}
\item It follows from axiom \ref{P4'} that group objects of $\Cat$ are $1$-group objects in $(\Cat,\covers)$ (\cite[Def.~3.4]{Rogers-Zhu:2016}). However, when $\Cat$ is the category of derived manifolds, this is automatically true even though \ref{P4'} fails (see Rem.~\ref{rmk:gp}). 
\item The proof of \cite[Prop.~7.13]{Rogers-Zhu:2016} uses that the product of covers is a cover, which is guaranteed by our weaker axiom \ref{P4}.
\end{enumerate} 
Thus we may safely replace \ref{P4'} by \ref{P4} and use all the results of \cite{Rogers-Zhu:2016} in this article.
\end{remark}
A \emph{refinement} of a cover $c\colon U\to X$ in a pretopology is a cover $c'\colon U'\to X$ together with a morphism $f\colon U'\to U$ such that $c'=c\circ f$.

We denote the category of presheaves of sets on a category $\Cat$ (i.e.\ contravariant functors $\Cat\to\Set$) by $\PSh(\Cat)$.  We denote by 
\[\yon\colon\Cat\longto\PSh(\Cat)\]
the Yoneda embedding given by $\yon(X)=\Hom_\Cat(\cdot, X)$.  A presheaf is \emph{representable} if it is isomorphic to a presheaf of the form $\yon(X)$ for some $X\in\Cat$.  If $\Cat$ is equipped with a pretopology $\covers$, we call a presheaf $F$ a \emph{sheaf (with respect to $\covers$)} if the diagram 
\[
\begin{tikzcd}
F(X)\ar[r,"F(c)"]&F(U)\ar[r,shift left=0.8ex,"F(\pr_1)"]\ar[r,shift right=0.8ex,"F(\pr_2)"']&F(U\times_X U)
\end{tikzcd}
\]
is an equalizer for all covers $c\colon U\to X$ in $\covers$.  We denote the category of sheaves by $\Sh(\Cat,\covers)$, or $\Sh(\Cat)$ when the pretopology $\covers$ is clear from the context.  

\begin{caution}\label{caution;sheaf}
Let $S$ be any set.  The constant presheaf $\underline{S}$ defined by $\underline{S}(X)=S$ for all $X\in\Cat$ is a sheaf with respect to any pretopology on $\Cat$.  But if $\Cat$ is the category of topological spaces, then $\underline{S}$ restricts to the constant presheaf on every object of $\Cat$, which is not a sheaf in the classical sense unless $S$ is empty.  This discrepancy is caused by our insistence that a cover should consist of a single morphism, whereas in Grothendieck's original definition of a pretopology a cover is a (possibly empty) collection of morphisms.  See \cite{waldorf;internal} for a comparison between various notions of pretopologies and sheaves.
\end{caution}

The pretopology $\covers$ is called \emph{subcanonical} if for every object $X\in\Cat$ the presheaf $\yon(X)$ is a sheaf with respect to $\covers$.   An \emph{effective epimorphism} is a morphism $f\colon X\to Y$ such that the fibred product $X\times_YX$ exists and $f$ is the colimit of the diagram $X\times_YX\rightrightarrows X$.  A pretopology is subcanonical if and only if every cover is an effective epimorphism; see e.g.~\cite[Def.-Lemma~2.2]{meyer-zhu}.

\subsection{Higher groupoids in a category with pretopology}\label{sec:higgrou}

Let $\Cat$ be a category. A \emph{simplicial object} in~\(\Cat\) is a
functor \(X_\bu\colon \Simpc^\op \to \Cat\), where~\(\Simpc\) is the
category of finite ordinals
\[
[0]=\{0\}, \qquad [1]=\{0, 1\},\quad \dotsc,\quad
[l]=\{0, 1,\dotsc, l\},
\]
with order-preserving maps as morphisms.  More explicitly, \(X_\bu\)~consists of a collection of objects~\(X_l\) in~\(\Cat\) and arrows \(\face^l_i\colon X_l \to X_{l-1}\) and \(\de^l_i\colon X_l \to X_{l+1}\) for \(i\in\{0, 1, 2,\dots, l\}\), respectively, called
\emph{face and degeneracy maps}, that satisfy
the \emph{simplicial identities}
\begin{equation}
  \label{eq:face-degen}
  \begin{alignedat}{2}
    \face^{l-1}_i \face^{l}_j &= \face^{l-1}_{j-1} \face^l_i &\quad
    &\text{if } i<j,\\
    \de^{l}_i \de^{l-1}_j &= \de^{l}_{j+1} \de^{l-1}_i &\quad
    &\text{if } i\le j,
  \end{alignedat}\qquad
  \face^l_i \de^{l-1}_j =
  \begin{cases}
    \de^{l-2}_{j-1} \face^{l-1}_i &\text{if } i<j,\\
    \id&\text{if } i=j\text{ or } j+1,\\
    \de^{l-2}_j \face^{l-1}_{i-1} &\text{if } i>j+1.
  \end{cases}
\end{equation}

A \emph{simplicial morphism} \(f_\bu\colon X_\bu\to Y_\bu\) is a
family of morphisms \(f_l\colon X_l\to Y_l\) in $\Cat$ that intertwine the face and
degeneracy maps of \(X_\bu\) and~\(Y_\bu\).

When~\(\Cat\) is the category of sets we refer to $X_\bu$ as a \emph{simplicial set}. Basic examples of simplicial sets are 
\begin{align*}
  (\Simp{k})_l &= \{f\colon \{0,1,\dotsc,l\} \to \{0,1,\dotsc, k\}
  \mid f(i)\le f(j) \text{ for all }i \le j\},\\
  (\partial\Simp{k})_l &=
  \bigl\{f\in (\Simp{k})_l \bigm| \{0,\dotsc,k\}
  \nsubseteq \{f(0),\dotsc, f(l)\} \bigr\},\\
  (\Horn{k}{j})_l &=
  \bigl\{f\in (\Simp{k})_l\bigm| \{0,\dotsc,j-1,j+1,\dotsc,k\}
  \nsubseteq \{f(0),\dotsc, f(l)\} \bigr\}.
\end{align*}
known respectively as the \emph{simplicial \(k\)-simplex}, its
\emph{boundary} and the $(k,j)$-\emph{horn}. Note that the
 $(k,j)$-horn is obtained from the
\(k\)-simplex by taking away its interior and
its \(j\)-th face.

  For $(\Cat, \covers)$ a category with pretopology and $X_\bu$ a simplicial object, we use the following notation for the sheaves $$\partial_j(X) := \Hom(\partial\Simp{j}, X)\quad \text{and}\quad  \Horn{j}{i}(X): = \Hom(\Horn{j}{i}, X).$$ 
  
  A simplicial object~\(X_\bu\) in~\((\Cat, \covers)\) satisfies the
  \emph{Kan condition} \(\Kan(k,j)\) if the restriction map
  \[
  p_{k,j}\colon\Hom(\Simp{k},X)\longto\Hom(\Horn{k}{j},X)
  \]
  is a cover and it satisfies the \emph{unique
    Kan condition} \(\Kan!(m,j)\) if $ p_{k,j}$ is an
  isomorphism.  In both cases, the existence of the limit
  \(\Hom(\Horn{m}{j},X)\) in~\(\Cat\) is assumed.

\begin{definition}[\cite{henriques}]\label{def:ngpd} 
Let $n\in\N\cup\{\infty\}$.  An \emph{\(n\)-groupoid} in a category with pretopology \((\Cat, \covers)\) is a simplicial object~\(X_\bu\) in~\(\Cat\) that satisfies 
\begin{itemize}
    \item  \(\Kan(k,j)\) for all \(k\ge1\)
  and \(0\le j\le k\); and
    \item \(\Kan!(k,j)\) for all \(k>n\) and
  \(0\le j\le k\).
\end{itemize}
We denote by $ \mathsf{Gpd}_n[\Cat, \covers]$  the full subcategory of simplicial objects in $\Cat$ whose objects are \(n\)-groupoids in \((\Cat, \covers)\). 
\end{definition}

We say that a simplicial morphism $f_\bu\colon X_\bu\to Y_\bu$ satisfies the  \emph{Kan condition} $\Kan(k,j)$ if the sheaf $\Hom(\Horn{k}{j}\to\Simp{k}, X\to Y)$ is representable and the canonical map (i.e., the horn projection)
\begin{equation}
    \label{eq:Kan_arrow}
    X_k = \Hom(\Simp{k},X)\xrightarrow{(\iota^\ast_{k,j},f_\ast)} 
\Hom(\Horn{k}{j}\xto{\iota_{k,j}} \Simp{k}, X\xto{f} Y)
\end{equation}
is a cover.  We say that $f_\bu$ satisfies the \emph{unique Kan condition} $\Kan!(k, j)$ if it satisfies $\Kan(k,j)$ and~\eqref{eq:Kan_arrow} is an isomorphism. 

Here for simplicial sets $S_\bu, T_\bu$ and simplicial objects $X_\bu, Y_\bu$, $\Hom(S\to T, X\to Y) $ is defined as the following fiber product (in the world of sheaves) with natural morphisms among the spaces (see \cite[\S\,2.1]{z:tgpd-2} for an explanation):
\begin{equation} \label{eq:2-arrow-prod}
    \Hom(S\to T, X\to Y) := \Hom(S, X) \times_{\Hom(S, Y)} \Hom(T, Y).
\end{equation} 
\begin{definition}[\cite{henriques}]
  \label{def:Kan_arrow}
A simplicial morphism $f_\bu \maps X_\bu \to Y_\bu$ of simplicial objects in $(\Cat, \covers)$ is a \emph{Kan fibration} if it satisfies $\Kan(k,j)$
  for all $k\ge 1$ and $0\le j\le k$.
\end{definition}

After introducing Kan fibrations we can also state that a simplicial object $X_\bu$ in $\Cat$ is  an $n$-groupoid in $(\Cat, \covers)$ if and only if the terminal morphism
\[
X_\bu \to \ast
\]
satisfies the Kan condition $\Kan(k,j)$ for $1\leq k \leq n$, $0 \leq
j \leq k$, and the unique Kan condition $\Kan!(k,j)$ for all $k >n$,
$0 \leq j \leq k$. In this case,  $\Hom(\Horn{k}{j}, X)$ is automatically representable, see \cite[Cor.2.5]{henriques} for more details.

\begin{definition} \label{def:equivalence}
A morphism \(f_\bu\colon X_\bu\to Y_\bu\) of simplicial objects in $(\Cat,\covers)$ is a  \emph{hypercover} if for all $j\ge0$ the sheaf \(\Hom(\partial\Simp{j}\to \Simp{j},X\to Y)\) is representable and the canonical boundary projection
  \begin{equation}
    \label{eq:hypercover-general}
   X_j= \Hom(\Simp{j}, X) \xrightarrow{q(f)_j:=(\jmath^\ast_j,f_\ast)}
    \Hom(\partial\Simp{j}\xto{\jmath_j} \Simp{j},X \xto{f} Y),
  \end{equation}
  is a cover  in \(\covers\).
  (For \(j=0\) we define \(\Hom(\partial\Simp{0}\to \Simp{0},X\to Y)\) to be \(Y_0\).)
\end{definition}
\begin{remark}\label{rm:h-cover-rep}
   The fiber product $\partial_j(X)\times_{\partial_j(Y)} Y_j$, which is by \eqref{eq:2-arrow-prod} the right hand side of \eqref{eq:hypercover-general}, is not necessarily representable even if $X_\bu, Y_\bu\in \mathsf{Gpd}_n[\Cat, \covers]$. 
However, one can show inductively that the above  fiber product is always representable if $X_\bu$ and $Y_\bu$ are $n$-groupoid objects and \(f_\bu\colon X_\bu\to Y_\bu\) is a hypercover, see \cite[\S\,2.1]{z:tgpd-2} for details.
\end{remark}

\subsubsection{An explicit combinatoric formula for hypercovers}

In this subsection let $(\Cat,\covers)$ be a category with a pretopology.
We have the following lemma which clarifies the condition for hypercover in the  situation when $n$ is finite. It is analogous to \cite[Lemma~3.13]{Behrend-Getzler:2015},  however since it is an important technical lemma and it is in a different setting, we still give a detailed proof here.

\begin{lemma} \label{lem:hypercover-n}
If $f_\bu\colon X_\bu\to Y_\bu$ is a hypercover of $n$-groupoid objects in $(\Cat, \covers)$, then $q(f)_j$ in \eqref{eq:hypercover-general}  is automatically an isomorphism  for $j\geq n$. 
\end{lemma}

\begin{proof}
We have a commutative diagram, in which the square is a pullback:
\begin{equation}\label{diag:hypercover-induction}
\begin{tikzcd}
X_j\ar[r,"q(f)_j"]\ar[dr,"p(f)^j_i"']&
\partial_j(X)\times_{\partial_j(Y)}Y_j\ar[r]\ar[d,"\tilde{q}"]&
X_{j-1}\ar[d,"q(f)_{j-1}"]\\
&\Horn{j}{i}(X)\times_{\Horn{j}{i}(Y)}Y_{j}\ar[r,"\delta_j"]&
\partial_{j-1}(X)\times_{\partial_{j-1}(Y)} Y_{j-1}. 
\end{tikzcd}
\end{equation}

Here $\delta_j$ is the natural map induced by face map $d^j_i\colon Y_j \to Y_{j-1}$. 

Recall that $n$-groupoids in $(\Cat, \covers)$ are $(n+1)$-coskeletal \cite[\S\,2.3]{z:tgpd-2},  i.e.\ $\Cosk^{n+1}(X)_\bu \cong X_\bu$. As
\[
\Cosk^{n+1}(X)_j= 
\begin{cases}
X_j&\text{for $j\leq n+1$}\\
\partial_{j} (X)&\text{for $j \geq n+2$}, 
\end{cases}
\]
we conclude that for an $n$-groupoid object $X_\bu$ we have
\[X_j \cong \partial_j(X)\quad \text{for $j\geq n+2$}.\]
Thus $q(f)_j$ is an isomorphism for $j\geq n+2$ in diagram \eqref{diag:hypercover-induction}. Moreover, $p(f)^j_i$ is an isomorphism for $j\geq n+1$ by the strict Kan condition. Thus $\tilde{q}$ is an isomorphism for $j\geq n+2$. 
When $j\ge n+1$, the map $\delta_j\colon X_j \to \partial_{j-1} (X)\times_{\partial_{j-1}(Y)} Y_{j-1}$ is simply the composition $q(f)_{j-1} \circ d^j_i$, thus a cover. 
Since our pretopology is subcanonical, the property of being an isomorphism is local; see \cite[Prop.~2.5]{meyer-zhu}.  This means that if the pullback $X\times_Y U \to U$  of a morphism $f\colon X\to Y$ by a cover $U\to Y$ is an isomorphism, then $f$ itself is an isomorphism.  It follows that $q(f)_{j-1}$ is an isomorphism for $j\ge n+2$.

Therefore, if we input the fact that $q(f)_{n+1}$ is an isomorphism and run the above argument again, we obtain that $q(f)_n$ is an isomorphism.
\end{proof}

The above Lemma~\ref{lem:hypercover-n} allows us to simplify the condition for hypercover in the case of $n$-groupoid objects.

\begin{corollary}\label{cor:hypercover-n}
    A morphism $f_\bu\colon X_\bu\to Y_\bu$ between $n$-groupoid objects in $(\Cat, \covers)$ is a hypercover if and only if the map
    \begin{equation}\label{eq:hyper-n}
    q(f)_j\colon X_j \to \partial_j(X)\times_{\partial_j(Y)} Y_j, 
\end{equation}
is a cover for $0\le j\le n-1$ and an isomorphism for $j= n$. In this case, $\partial_j(X)\times_{\partial_j(Y)} Y_j$ is automatically representable for all $j$ and $q(f)_j$ is automatically an isomorphism for all $j\ge n$. 
\end{corollary}
\begin{proof}
 The representability of $\partial_j(X)\times_{\partial_j(Y)} Y_j$ is proven in \cite[Lemma~2.4]{z:tgpd-2}.  If $q(f)_n$ is an isomorphism, then $\tilde{q}$ in diagram \eqref{diag:hypercover-induction} is an isomorphism.  Since $p(f)^{n+1}_i$ is an isomorphism, $q(f)_{n+1}$ is an isomorphism. Thus inductively, all $q(f)_{\geq n}$ are isomorphisms. 
\end{proof}

\begin{example}
 Let $X_\bu$ and $Y_\bu$ be $1$-groupoid objects in a category with pretopology $(\Cat,\covers)$. An immediate consequence of Cor.~\ref{cor:hypercover-n} is that a morphism $f_\bu\colon X_\bu\to Y_\bu$ is an hypercover if and only if the following two conditions hold:
\begin{itemize}
    \item the map $f_0\colon X_0\to Y_0$ is a cover; and
    \item the map $q(f)_1=((d_0,d_1),f_1)\colon X_1\to (X_0\times X_0)\times_{(Y_0\times Y_0)}Y_1$ is an isomorphism.
\end{itemize}
\end{example}

\subsection{Incomplete categories of fibrant objects}\label{def:cfo}

Let $\Catd$ be a category with finite products and terminal object $\ast
\in \Catd$ equipped with two distinguished classes of morphisms called
\emph{weak equivalences} and \emph{fibrations}. A morphism which is both a weak equivalence and a fibration is called an
\emph{acyclic fibration}.  Following \cite{Rogers-Zhu:2016}, we say $\Catd$ is an
\emph{ incomplete category of fibrant objects} (iCFO) if the following conditions holds.
\begin{enumerate}
\setlength{\itemindent}{1em}
\item[\namedlabel{iCFO1}{(iCFO1)}]{Every isomorphism in $\Catd$ is an acyclic fibration.}

\item[\namedlabel{iCFO2}{(iCFO2)}]{The class of weak equivalences has the two-out-of-three property.  That is, if $f$ and $g$ are composable morphisms in $\Catd$ and any two of the morphisms $f$, $g$, $g\circ f$ are weak equivalences, then so is the third.}

\item[\namedlabel{iCFO3}{(iCFO3)}]{The composition of two fibrations is a fibration.}

\item[\namedlabel{iCFO4}{(iCFO4)}]{If the pullback along a fibration exists, then it is a fibration.
That is, if $f\colon X\to Z$ and $g\colon Y\to Z$ are morphisms in $\Catd$, $f$ is a fibration, and $X \times_{Z} Y$ exists, then the induced morphism $X\times_ZY\to Y$ is a fibration.}

\item[\namedlabel{iCFO5}{(iCFO5)}]{The pullback along an acyclic fibration exists and is an acyclic fibration again.
That is, if $f\colon X\to Z$ and $g\colon Y\to Z$ are morphisms in $\Catd$ and $f$ is an acyclic fibration, the pullback $X\times_ZY$ exists, and the induced morphism $X\times_ZY \to Y$ is an acyclic fibration.}

\item[\namedlabel{iCFO6}{(iCFO6)}]{For every object $X \in \Catd$ there exists a \emph{path object}, that is, an object $X^{I}$ equipped with morphisms
\[
\begin{tikzcd}
X\ar[r,"s"]&X^{I}\ar[r,"{(d_0,d_1)}"]&X\times X,    
\end{tikzcd}
\]
such that $s$ is a weak equivalence, $(d_0,d_1)$ is a fibration, and their composite is the diagonal.}

\item[\namedlabel{iCFO7}{(iCFO7)}]{All objects of $\Catd$ are \emph{fibrant}, that is for any $X \in \Catd$ the unique map 
$ X \to \ast$ is a fibration.}
\end{enumerate}
This is a generalization of the notion of a category of fibrant objects introduced in \cite{brown:1973}.  An iCFO is a \emph{category of fibrant objects} (CFO) if it satisfies the following stronger version of \ref{iCFO4}:
\begin{enumerate}
\setlength{\itemindent}{1em}
\item[\namedlabel{CFO4}{(CFO4)}]The pullback along every fibration exists and is a fibration.
\end{enumerate}

\subsection{Points and locally stalkwise pretopologies}\label{sec:points}

The main result of \S \ref{section;derived-higher} states that the category of higher groupoids of derived manifolds with respect to a suitable pretopology forms an iCFO.  Our proof will rely on the following notions introduced in \cite{Rogers-Zhu:2016}.

\begin{definition}\label{definition;point}
Let $\Cat$ be a category equipped with a pretopology $\covers$.  A \emph{point} of $(\Cat,\covers)$ is a functor
\[
\ppt\colon\Sh(\Cat, \covers) \longto \Set
\]
which preserves finite limits and small colimits.  A collection of points $\pts$ of $(\Cat,\covers)$ is \emph{jointly conservative} if a morphism $\phi\colon F
\to G$ in $\Sh(\Cat, \covers)$ is an isomorphism if and only if for all $\ppt \in \pts$ the map
\[
\ppt(\phi)\colon \ppt(F) \longto \ppt(G) 
\]
is a bijection.
\end{definition}

\begin{definition}\label{definition;stalkwise-surjection}
Let $\Cat$ be a category equipped with a subcanonical pretopology $\covers$ and a jointly conservative collection of points $\pts$.
\begin{enumerate}
\item A morphism of sheaves $\phi\colon F\to G$ is a \emph{stalkwise surjection (with respect to $\pts$)} if for all points $\ppt \in \pts$ the map
\[
\ppt(\phi)\colon \ppt (F) \longto  \ppt (G) 
\]
is surjective.
\item Let $X\in\Cat$.  A morphism $\phi\colon F\to\yon(X)$ of sheaves is a \emph{locally stalkwise cover (with respect to $\pts$)} if and only if there exists a stalkwise surjection $\psi\colon\yon(U)\to F$ such that the composition $\phi\circ\psi\colon\yon(U)\to\yon(X)$ is represented by a cover $c\colon  U\to X$, that is $\yon(c)=\phi\circ\psi$. 
\end{enumerate}
\end{definition}

\begin{definition}\label{definition;stalkwise}
Let $\Cat$ be a category equipped with a subcanonical pretopology $\covers$ and a jointly conservative collection of points $\pts$.  The pretopology $\covers$ is \emph{locally stalkwise (with respect to $\pts$)} if the following conditions hold:
\begin{enumerate}\setlength{\itemindent}{1em}
\item[\namedlabel{LSP1}{(LSP1)}] for all morphisms $f\colon U\to V$ and $g\colon V\to W$,
if $g\circ f$ is a cover and $\yon(f)\colon\yon(U)\to\yon(V)$ is a stalkwise surjection of sheaves, then $g$ is a cover;
\item [\namedlabel{LSP2}{(LSP2)}] (locality of covers) for all morphisms $p\colon U\to W$ and $q\colon V\to W$, if $\yon(q)\colon\yon(V)\to\yon(W)$ is a stalkwise surjection and the map $\yon(U)\times_{\yon(W)}\yon(V)\to\yon(V)$ induced by $p$ is a locally stalkwise cover, then $p$ is a cover.  (It follows that $U\times_WV$ exists and that $\yon(U)\times_{\yon(W)}\yon(V)=\yon(U\times_WV)$.)
\end{enumerate}
\end{definition}

\begin{remarks}\label{rk:2-out-of-3}
\begin{enumerate}[wide,labelwidth=0pt,labelindent=0pt]
\item\label{item;2-out-of-3} If $\phi\colon F\to G$ and $\psi\colon G\to H$ are morphisms of sheaves and $\psi\circ\phi$ is stalkwise surjective, then $\psi$ is stalkwise surjective. 
\item\label{item;stalkwise-cover} It follows from \ref{LSP1} that if a morphism $c\colon U\to X$ has the property that $\yon(c)\colon\yon(U)\to\yon(X)$ is a locally stalkwise cover, then $c$ is a cover.  But the converse is false: there may exist covers $c\colon U\to X$ such that $\yon(c)\colon\yon(U)\to\yon(X)$ is not a locally stalkwise cover.
\end{enumerate}
\end{remarks}
%

\subsubsection{Stalkwise weak equivalences}

Here we give a notion of ``weak equivalence'' which will turn out to be equivalent to a more combinatorical one (see Prop.~\ref{prop:w-eq-comb}), but provides us with an easy connection with the well-established theory of simplicial sets and will help to simplify proofs.  In this subsection $(\Cat,\covers, \pts)$ denotes a category equipped with a collection of jointly conservative points $\pts$ and a locally stalkwise pretopology $\covers$.

\begin{definition} \label{def:stalk_weq}
 A morphism $f_\bu \maps X_\bu
\to Y_\bu$ of higher groupoids in $(\Cat,\covers, \pts)$ is a \emph{stalkwise
  weak equivalence} if and only if $\ppt(f_\bu)\maps\ppt(X_\bu)\to\ppt(Y_\bu)$ is a weak
homotopy equivalence of simplicial sets for all points $\ppt \in \pts$.
\end{definition}

By \cite[Prop.~6.7]{Rogers-Zhu:2016}, a Kan fibration which is at the same time a stalkwise weak equivalence is precisely a hypercover as introduced in Definition \ref{def:equivalence}. In \cite{Behrend-Getzler:2015}, there is a combinatoric version of weak equivalence. Moreover, for Lie groupoids, there is also a notion of essential equivalence or Morita equivalence map \cite{moerdijk-mrcun}. All these turn out to coincide with  our stalkwise weak equivalence. We have the following results.

\begin{proposition}\label{prop:w-eq-comb}
   Let $\gpd{\Cat, \covers}$ denote $\infty$-groupoid objects in a category $\Cat$ equipped with a locally
   stalkwise pretopology $\covers$ with respect to a jointly conservative collection of points $\pts$.  A map $f_\bu \maps X_\bu\to Y_\bu$  is a stalkwise weak equivalence if and only if the morphism
\begin{equation}\label{eq:w-eq-infty}
    r(f)_j=(\partial_j, p_{j+1}^{j+1})\colon X_j\times_{Y_j, d_{j+1}} Y_{j+1} \to \partial_j(X)\times_{\partial_j(Y)} \Lambda^{j+1}_{j+1}(Y)
\end{equation}
is a cover for all $j\ge 0$.  Moreover, the right hand side is automatically representable when $f_\bu$ is a stalkwise weak equivalence.
\end{proposition}

This also yields us a simplified version of stalkwise weak equivalence in the truncated case for the  $n$-groupoid objects in $(\Cat, \covers, \pts)$. 

\begin{corollary}\label{cor:w-eq-comb}
Let $X_\bu,Y_\bu$ be $n$-groupoid objects in $(\Cat, \covers, \pts)$.
\begin{enumerate}
    \item A map $f_\bu\colon X_\bu\to Y_\bu$ is a stalkwise weak equivalence if and only if the maps $r(f)_j$ in \eqref{eq:w-eq-infty} are  covers for $j\le n-1$ and an isomorphism for $j=n$.
    \item  The map $r(f)_n$ is an isomorphism if and only if the map $q(f)_n$ given in \eqref{eq:hyper-n} is an isomorphism.
\end{enumerate}
\end{corollary}

The proof of Prop.~\ref{prop:w-eq-comb} is analogous to that of \cite[Thm.~5.1]{Behrend-Getzler:2015} or \cite[Cor.~7.16]{Rogers-Zhu:2016}. But since we are in a different setting from \cite{Behrend-Getzler:2015}, e.g.\ the representability of the right-hand side is an issue, and \cite{Rogers-Zhu:2016} only gives a brief summary, we give a detailed proof in App.~\ref{app:comb} together with a proof of Cor.~\ref{cor:w-eq-comb}. See also \cite[Lemma 1.2.50]{Ronchi:thesis} for a similar statement. Taking $n=1$ in Cor.~\ref{cor:w-eq-comb} gives the following result.

\begin{corollary}\label{cor:w-eq-1gpd}
Let $X_\bu$ and $Y_\bu$ be $1$-groupoid objects in $(\Cat,\covers,\pts)$. A morphism $f_\bu\colon X_\bu\to Y_\bu$ is a stalkwise weak equivalence if and only if the following two conditions hold:
\begin{itemize}
    \item the map $d_0\circ\pr_2:X_0\times_{Y_0} Y_1 \to Y_0 $ is a cover; 
    \item the map $((d_0,d_1),f_1): X_1 \to (X_0\times X_0) \times_{(Y_0\times Y_0)} Y_1$ is an isomorphism. 
\end{itemize}
\end{corollary}

If $\Cat$ is the category of Lie groupoids, this means that a morphism is a stalkwise weak equivalence if and only if it is a weak equivalence in the sense of \cite[\S\,5.4]{moerdijk-mrcun}.

\subsubsection{The iCFO of a locally stalkwise pretopology}

The following is one of the main results of \cite{Rogers-Zhu:2016} and will play an important role in \S\,\ref{section;derived-higher}.  (See Rem.~\ref{rmk:p4} concerning the validity of this result under our slightly different axioms for pretopologies.)

\begin{theorem}[{\cite[Thm.~7.1]{Rogers-Zhu:2016}}]\label{thm:RZ}
Let $\Cat$ be a category equipped with a pretopology $\covers$ that is  locally stalkwise with respect to a jointly conservative collection of points $\pts$.  The category $\mathsf{Gpd}_n[\Cat,\covers]$ is an iCFO in which
    \begin{itemize}
        \item the weak equivalences are the stalkwise weak equivalences;
        \item the fibrations are the Kan fibrations;
        \item the acyclic fibrations are hypercovers.
    \end{itemize}
\end{theorem}

\section{Derived manifolds}\label{sec:derman}

This section is a quick tutorial on derived manifolds based mostly on 
\cite{blx}, \cite{bursztyn-cueca-mehta;n-manifolds-frobenius}, \cite{carchedi:23}, and \cite{cattaneo-schatz;supergeometry}.  Derived manifolds can be described in
many different but equivalent ways.  
For brevity we concentrate on the point of view that a derived manifold
is a differential graded manifold of nonnegative amplitude. 
Broadly this means an ordinary smooth manifold to which we adjoin a number of linear coordinates 
of \emph{negative} degree, and which we equip with a vector field of degree $+1$ that generates a ``flow''.  
The simplest derived manifolds, and the most important for the 
applications covered in this paper, are the quasi-smooth ones (see 
\S\,\ref{subsection;derived-zero}), where the adjoined variables are all
of degree $-1$.  
In plainest terms, a quasi-smooth derived manifold can be described as a manifold $M$ furnished 
with a vector bundle $E$ (whose fibres provide the extra coordinates) and a section $s$ of $E$ (the ``vector field'').  
Derived manifolds form a category of fibrant objects (see \S\,\ref{subsection;fibrant}) and a quasi-smooth derived manifold is to be regarded 
as a fibrant resolution of the zero locus of the section $s$, which is usually not a smooth submanifold of~$M$.  
 
\subsubsection*{Cohomological grading conventions}

In agreement with our main references, our grading conventions are chosen so as to make all complexes \emph{cochain} complexes.

\subsection{Non-positively graded manifolds}\label{subsection;graded}

Let $V_0$, $V_1,\dots$ be a sequence of finite-dimensional real vector
spaces with $V_n=0$ for all but finitely many $n$.  Consider the
positively graded vector space $V=\bigoplus_{n=1}^\infty V_n$ and its
graded dual
\[V^*=\bigoplus_{n=-\infty}^{-1}(V^*)_n,\]
where $(V^*)_n=(V_{-n})^*$.  A \emph{graded chart modelled on}
$(V_0,V_1,\dots)$ is a ringed space $(O,\ca{S}^\bu_O)$, where $O$ is
an open subset of $V_0$ and the structure sheaf $\ca{S}^\bu_O$ is the
sheaf of graded commutative $\R$-algebras over $O$ associated with the presheaf 
\[
U\longmapsto C^\infty(U)\otimes_\R\Sym^\bu(V^*).
\]
Here $U$ is an open subset of $O$, $C^\infty(U)$ denotes the algebra of smooth functions on $U$, and
$\Sym^\bu(V^*)$ denotes the \emph{graded symmetric algebra} of $V^*$,
which is the universal graded commutative $\R$-algebra generated by
$V^*$.  Given a choice of bases
\begin{equation}\label{notation;chart}
\begin{split}
x_1,x_2,\dots,x_{k_0}&\quad\text{of\/ $(V_0)^*$}\quad\text{(degree
  $0$)}\\
\xi_1,\xi_2,\dots,\xi_{k_1}&\quad\text{of $(V_1)^*$}\quad\text{(degree
  $-1$)}\\
\xi_{k_1+1},\xi_{k_1+2},\dots,\xi_{k_1+k_2}&\quad\text{of
  $(V_2)^*$}\quad\text{(degree $-2$)},\quad\dots,
\end{split}
\end{equation}
each section of $\ca{S}^\bu_O$ over $U$ can be written, in a sufficiently small neighbourhood of any point of $U$, as a function
\[
f(x,\xi)
\]
which is a polynomial in the $\xi$-variables whose coefficients are smooth functions in the $x$-variables.  The variables $x_i$ are
central and the variables $\xi_j$ anticommute according to the Koszul
sign rule,
\begin{equation}\label{equation;koszul}
ba=(-1)^{\abs{a}\abs{b}}ab
\end{equation}
for homogeneous elements $a$, $b$ of $V^*$, where $\abs{a}$ denotes
the degree of $a$.

A \emph{graded manifold} is a ringed space $(M,\ca{S}^\bu_M)$, where
$M$ is a paracompact space and $\ca{S}^\bu_M$ is a
sheaf of graded commutative $\R$-algebras, which locally near every
point in $M$ is isomorphic to a graded chart.  The \emph{body} of a
graded manifold $(M,\ca{S}^\bu_M)$ is the ringed space
$(M,\ca{S}^0_M)$, where $\ca{S}^0_M$ is the degree $0$ part of the
sheaf $\ca{S}^\bu_M$.  The body is an ordinary smooth manifold with
structure sheaf $\ca{C}^\infty_M=\ca{S}^0_M$.  A \emph{morphism} of
graded manifolds $(M,\ca{S}^\bu_M)\to(N,\ca{S}^\bu_N)$ is a pair
$(f,f^\sharp)$ consisting of a smooth map $f\colon M\to N$ and a
morphism of sheaves of graded commutative $\R$-algebras
$f^\sharp\colon f^{-1}\ca{S}^\bu_N\to\ca{S}^\bu_M$ which in degree $0$
agrees with the pullback map $f^{-1}\ca{C}^\infty_N\to\ca{C}^\infty_M$.
By abuse of notation we will often write a morphism as
$f\colon(M,\ca{S}^\bu_M)\to(N,\ca{S}^\bu_N)$.

\subsection{Vector fields}\label{sec:Dman-vf}

Let $(M,\ca{S}^\bu_M)$ be a graded manifold.  A \emph{vector field of
degree $n$} on $M$ is an $\R$-linear left derivation of degree $n$ of
the structure sheaf
\[X\colon\ca{S}^\bu_M\to\ca{S}^\bu_M[n].\]
Vector fields of degree $n$ form a sheaf denoted by $\X_M^n$.  We
call the direct sum
\[\X_M^\bu=\bigoplus_{n\in\Z}\X_M^n\]
the \emph{tangent sheaf} of $M$ and its sections \emph{vector fields}.
Calculating in a graded chart as in~\eqref{notation;chart}, we can
express a vector field as
\begin{equation}\label{notation;vector-field}
X=\sum_if_i(x,\xi)\pardif{}{x_i}+ \sum_jg_j(x,\xi)\pardif{}{\xi_j}
\end{equation}
where $\pardif{}{\xi_j}$ denotes the \emph{left} partial derivative
with respect to $\xi_j$.
(E.g.\ $\pardif{}{\xi_j}\xi_j\xi_k=\xi_k$, but
$\pardif{}{\xi_k}\xi_j\xi_k=(-1)^{\abs{\xi_j}\abs{\xi_k}}\xi_j$.)
Equipped with the graded commutator
\begin{equation}\label{equation;commutator}
[X_1,X_2]=X_1\circ X_2-(-1)^{\abs{X_1}\abs{X_2}}X_2\circ X_1
\end{equation}
the tangent sheaf $\X^\bu_M$ is a sheaf of $\Z$-graded real Lie
algebras.

The tangent sheaf is also a left $\ca{S}^\bu_M$-module, which because
of~\eqref{notation;vector-field} is locally freely generated by
elements $\pardif{}{x_i}$ and $\pardif{}{\xi_j}$ of degrees
\[
\biggabs{\pardif{}{x_i}}=0,\qquad\biggabs{\pardif{}{\xi_j}}=-\abs{\xi_j}.
\]
It follows that $\X_M^\bu$ is the sheaf of smooth sections of an
$\N$-graded vector bundle over the graded manifold $(M,\ca{S}^\bu_M)$,
which is called the \emph{graded tangent bundle} and denoted by
\[T^\bu M=\bigoplus_{n=0}^\infty T^nM.\]
Its
fibre at $x\in M$ is an $\N$-graded vector space, called the
\emph{graded tangent space} at $x$ and denoted by
\begin{equation}\label{notation;tangent}
T^\bu_xM=\bigoplus_{n=0}^\infty T^n_xM.
\end{equation}
The piece $T^n_xM$ is equal to the space of pointwise derivations
$\ca{S}^\bu_{M,x}\to\R[n]$, where $\ca{S}^\bu_{M,x}$ is the stalk of
the structure sheaf at $x$.  In particular $T^0_xM=T_xM$ is the
ordinary tangent space to the body of $M$ at $x$.  In a graded chart
we find the graded tangent space by evaluating the
expressions~\eqref{notation;vector-field} at $x$ and setting $\xi=0$,
which yields
\begin{equation}\label{equation;tangent}
T^n_xM\cong V_n.
\end{equation}
A morphism of graded manifolds $f\colon M\to N$ induces a morphism of
graded tangent bundles
$T^\bu f\colon T^\bu M\to T^\bu N$, which we refer to as the
\emph{tangent map} of~$f$.

\subsection{Derived manifolds}\label{sec:derdef}

Let $(M,\ca{S}^\bu_M)$ be a graded manifold.  A \emph{cohomological
vector field} on $M$ is a vector field $Q$ of degree $1$ which
satisfies the integrability condition $[Q,Q]=2Q^2=0$.  A \emph{derived manifold} is a graded
manifold equipped with a cohomological vector field.  We will usually
denote a derived manifold by a script letter, often, but not always,
the script letter $\ca{M}$ corresponding to the body $M$.  If $x$ is a point of the body of $\huaM$, we will by abuse of language say that $x$ is a point of $\huaM$ and write $x\in\huaM$.

Let $\ca{M}=(M,\ca{S}^\bu_M,Q)$ be a derived manifold.  The
cohomological vector field $Q$ makes the structure sheaf a
(nonpositively graded) cochain complex, called the \emph{structure
complex},
\begin{equation}\label{equation;structure-complex}
\begin{tikzcd}
\cdots\ar[r,"Q"]&\ca{S}^{-2}_M\ar[r,"Q"]&\ca{S}^{-1}_M\ar[r,"Q"]&
\ca{S}^0_M.
\end{tikzcd}
\end{equation}
Equipped with the operator $\Lie_Q=\ad_Q=[Q,\bu]$, the tangent sheaf
becomes a cochain complex, called the \emph{tangent sheaf complex},
\begin{equation}\label{equation;tangent-sheaf-complex}
\begin{tikzcd}
\cdots\ar[r,"\Lie_Q"]&\X^{-1}_M\ar[r,"\Lie_Q"]&\X^0_M\ar[r,"\Lie_Q"]&
\X^1_M\ar[r,"\Lie_Q"]&\X^2_M\ar[r,"\Lie_Q"]&\cdots
\end{tikzcd}
\end{equation}
A \emph{classical point} of $\ca{M}$ is a point $x$ where $Q$
vanishes, i.e.\ where the pointwise derivation
$Q_x\colon\ca{S}^\bu_{M,x}\to\R[1]$ induced by $Q$ is equal to $0$.
The \emph{classical locus} or \emph{Maurer-Cartan locus}
$\pi_0(\ca{M})$ of $\ca{M}$ is the vanishing locus of $Q$,
\begin{equation}\label{notation;classical}
\pi_0(\ca{M})=\{\,x\in\ca{M}\mid Q_x=0\,\}.
\end{equation}
At a classical point $x$ the operation $\Lie_Q$ linearizes to an operator
denoted by $\ell_{Q,x}$ or simply $\ell_Q$, making the graded tangent space
a cochain complex, called the \emph{tangent complex} at $x$,
\begin{equation}\label{notation;tangent-complex}
\begin{tikzcd}
T^\bu_x\ca{M}\colon\quad T^0_xM\ar[r,"\ell_Q"]&T^1_xM\ar[r,"\ell_Q"]&T^2_xM
\ar[r,"\ell_Q"]&\cdots
\end{tikzcd}
\end{equation}
The Euler characteristic of this complex is called the \emph{virtual
dimension} $\dim_x\ca{M}$ of $\ca{M}$ at $x$, which is independent of
$x$ if $M$ is connected.

Let $\ca{M}=(M,\ca{S}^\bu_M,Q_M)$ and $\ca{N}=(N,\ca{S}^\bu_N,Q_N)$ be
derived manifolds and let $f\colon(M,\ca{S}^\bu_M)\to(N,\ca{S}^\bu_N)$ be a morphism of the underlying graded manifolds.  We say that the
cohomological vector fields $Q_M$ and $Q_N$ are \emph{$f$-related}, and we write $Q_M\sim_fQ_N$, if
\begin{equation}\label{equation;related}
Q_M(f^\sharp(g))=f^\sharp(Q_N(g))
\end{equation}
for all sections $g$ of $\ca{S}^\bu_N$. If $Q_M$ and $Q_N$ are $f$-related, we call $f\colon\ca{M}\to\ca{N}$ a \emph{morphism of derived manifolds} or a \emph{derived morphism}.  We denote the category of derived manifolds by $\DMfd$.  We will view an ordinary manifold $M$ as a derived manifold with
trivial grading ($\ca{S}^\bu_M=\ca{S}^0_M=\ca{C}^\infty_M)$ and zero
cohomological vector field.  Accordingly we will regard the category of manifolds as a full subcategory of the category of derived manifolds,
\begin{equation}\label{equation;mfd-dmfd}
\begin{tikzcd}
\Mfd\ar[r,hook]&\DMfd.
\end{tikzcd}
\end{equation}
A morphism of derived manifolds $f\colon\ca{M}\to\ca{N}$ induces a map
of classical loci
\begin{equation}\label{equation;morphism-loci}
f\colon\pi_0(\ca{M})\longto\pi_0(\ca{N})
\end{equation}
and for each $x\in\pi_0(\ca{M})$ a morphism of tangent complexes
\begin{equation}\label{equation;morphism-tangent}
T^\bu_xf\colon T^\bu_x\ca{M}\longto T^\bu_{f(x)}\ca{N}.
\end{equation}

\subsubsection{Amplitude and quasi-smoothness}

Following \cite[\S\,1.1]{wei:hom} we say that a cochain complex $V$ in an additive
category has \emph{amplitude} $I$ (where $I$ is an interval in the extended
integers $\Z\cup\{\pm\infty\}$) if $V^n=0$ for $n\not\in I$.  By abuse of language,
if $V$ has amplitude $[0,k]$ we also say it has amplitude $k$.  We say that a derived manifold has amplitude $k$ if its graded tangent bundle has amplitude $k$. 
We call a derived manifold \emph{quasi-smooth} if it has amplitude $1$, i.e.\ its graded tangent bundle is zero in degrees $\ge2$.

\subsection{Derived zero loci}
\label{subsection;derived-zero}

Let $E\to M$ be a smooth real vector bundle and $s$ a smooth section
of $E$.  Let $\ca{E}$ be the sheaf of smooth sections of $E$, let $\ca{E}^*[1]$ be a copy of the dual sheaf $\ca{E}^*$ placed in
degree $-1$, and let
\[\ca{S}^\bu_M=\Sym^\bu(\ca{E}^*[1])\]
be its sheaf of graded symmetric $\R$-algebras.  A vector bundle chart
$U\times\R^r$ on $E$ amounts to the same thing as a graded chart on
the ringed space $(M,\ca{S}^\bu_M)$ with coordinates $x_1$,
$x_2,\dots$, $x_n$ (of degree~$0$) on the base $U$ and coordinates
$\xi_1$, $\xi_2,\dots$, $\xi_r$ (of degree $-1$) along the fibre
$\R^r$.  Thus $(M,\ca{S}^\bu_M)$ is a graded manifold modelled on
$(\R^n,\R^r,0,\dots)$, where $n$ is the dimension of $M$ and $r$ is
the rank of $E$.  The degree $-k$ part of the structure sheaf is
\[\ca{S}^{-k}_M=\Sym^{-k}(\ca{E}^*[1])=\Lambda^k(\ca{E}^*),\]
the $k$-th exterior power of $\ca{E}^*$.  A section of $\ca{S}^{-k}_M$
over an open subset $U\subseteq M$ is nothing but an alternating
$C^\infty(U)$-multilinear form
\[
\underbrace{\ca{E}(U)\times\ca{E}(U)\times\cdots\times\ca{E}(U)}_k
\longto C^\infty(U).
\]
It follows that contraction with the section $s$ defines a
cohomological vector field
\[\iota_s\colon\ca{S}^\bu_M\longto\ca{S}^\bu_M[1].\]
In a graded chart $U$ we can express $s$ as a linear combination
$s=\sum_js_j(x)\xi^*_j$ with $s_j\in C^\infty(U)$.  For a section $f(x,\xi)$ of $\ca{S}^\bu_M$
we then have $\iota_sf=\sum_js_j(x)\pardif{f}{\xi_j}$, which we will write symbolically as
\[\iota_sf=\pardif{f}{s}.\]
Thus the triple $(M,\ca{S}^\bu_M,\iota_s)$ is a derived manifold and,
since the structure sheaf is generated by elements of degrees $0$ and
$-1$, it is quasi-smooth.  A converse of this
statement is as follows; see e.g.\ \cite[Ex.~3.11]{bursztyn-cueca-mehta;n-manifolds-frobenius}.

\begin{proposition}\label{prop:quasi-smooth}
Let $\mathsf{C}$ be the category whose objects are triples $(M,E,s)$ consisting of a manifold $M$, a vector bundle $E$, and a smooth section $s$, and whose morphisms $(M,E,s)\to(M',E',s')$ are pairs consisting of a smooth map $f\colon M\to M'$ and a vector bundle map $f^\sharp\colon E\to E'$ over $f$ such that $s'\circ f=f^\sharp\circ s$.  The functor $\ca{Z}\colon\mathsf{C}\to\DMfd$ defined by
\[\ca{Z}(M,E,s)=(M,\Sym^\bu(\ca{E}^*[1]),\iota_s)\]
is fully faithful and its essential image is the category of quasi-smooth derived manifolds.
\end{proposition}

We will call the quasi-smooth derived manifold $\ca{Z}(M,E,s)$ the \emph{derived
zero locus} of the section $s$.  Depending on the context we will adopt various 
notations such as $\ca{Z}(M,E,s)$, or $\ca{Z}(s)$, or $(E[-1],s)$.  
The classical locus of $\ca{Z}(s)$ is
\[\pi_0(\ca{Z}(s))=Z(s)=\{\,x\in M\mid s(x)=0\,\},\]
the usual zero locus of $s$.  The structure complex of $\ca{Z}(s)$ is
the Koszul complex of the section $s$,
\[
\begin{tikzcd}
\cdots\ar[r,"\iota_s"]&\Lambda^3(\ca{E}^*)\ar[r,"\iota_s"]&
\Lambda^2(\ca{E}^*)\ar[r,"\iota_s"]&\ca{E}^*\ar[r,"\iota_s"]&
\ca{C}^\infty_M,
\end{tikzcd}
\]
which is exact off the zero locus.  The graded tangent bundle is given
by
\[T^0\ca{Z}(s)=TM,\qquad T^1\ca{Z}(s)=E.\]
To find the tangent complex at a classical point $x\in Z(s)$ we let
$X=\sum_if_i(x)\pardif{}{x_i}$ and we compute
\[
[\iota_s,X]=\Bigl[\pardif{}{s},X\Bigr]=
\biggl[\sum_js_j(x)\pardif{}{\xi_j},\sum_if_i(x)\pardif{}{x_i}\biggr]=
-\sum_jX(s_j)\pardif{}{\xi_j}.
\]
This yields $\ell_{\iota_s}=-d_xs$, minus the derivative at $x$ of the
section $s$, which is well-defined because $s(x)=0$.  Thus the tangent
complex at $x$ is 
\[
\begin{tikzcd}
T^\bu_x\ca{Z}(s)\colon\quad T_xM\ar[r,"-d_xs"]&E_x.
\end{tikzcd}
\]
The virtual dimension of $\ca{Z}(s)$ is $n-r$.  If the section $s$ is
transverse to the zero section at $x$, then $Z(s)$ is an
$n-r$-dimensional submanifold of $M$ near $x$ and the cohomology of the tangent
complex is
\[H^0(T^\bu_x\ca{Z}(s))=T_xZ(s),\qquad H^1(T^\bu_x\ca{Z}(s))=0.\]
If $s$ is not transverse to the zero section, then $Z(s)$ is typically
not a manifold, and we regard the derived zero locus $\ca{Z}(s)$ as a
``smooth resolution'' of $Z(s)$ in the category of derived manifolds.

\subsection{Submanifolds and products}

\subsubsection{Submanifolds}\label{subsubsection;submanifolds}

Let $(M,\ca{S}^\bu_M)$ be a graded manifold.  A \emph{graded submanifold}
of $(M,\ca{S}^\bu_M)$ is a graded manifold $(P,\ca{S}^\bu_P)$ together 
with a morphism $i\colon(P,\ca{S}^\bu_P)\to(M,\ca{S}^\bu_M)$ such that $P$ is an embedded submanifold of $M$ with inclusion map $i\colon P\to M$ and the tangent map $T^\bu i\colon T^\bu P\to T^\bu M$ is injective.  If $(P,\ca{S}^\bu_P)$ is a graded submanifold of $(M,\ca{S}^\bu_M)$, then the sheaf map $i^\sharp\colon\ca{S}^\bu_M\to i_*\ca{S}^\bu_P$ is surjective, and its kernel is called the \emph{ideal sheaf} of $(P,\ca{S}^\bu_P)$. 

As a special case, let $P$ be a submanifold of (the body of) a graded manifold
$(M,\ca{S}^\bu_M)$.  Let $\ca{I}^\bu_P$  be the ideal of $\ca{S}^\bu_M$ generated by the
sections of $\ca{S}^0_M=\ca{C}^\infty_M$ that vanish on $P$ and let $\ca{S}^\bu_P$ be the restriction to $P$ of the quotient sheaf $\ca{S}^\bu_M\big/\ca{I}^\bu_P$.  The ringed space $(P,\ca{S}^\bu_P)$ is a graded submanifold of $(M,\ca{S}^\bu_M)$; we call
$\ca{S}^\bu_P$ the \emph{induced graded structure}.  If $M$ is
modelled on the tuple $(V_0,V_1,V_2,\dots)$, then $P$ is modelled on
the tuple $(W_0,V_1,V_2,\dots)$ for a suitable subspace $W_0$
of~$V_0$.  The graded tangent bundle of $(P,\ca{S}^\bu_P)$ is given by 
\begin{equation}\label{equation;tangent-submanifold}
T^0P=TP,\qquad T^nP=T^nM|_P\quad\text{for $n\ge1$}.
\end{equation}

Let $Q_M$ be a cohomological vector field on $M$, making $\ca{M}=(M,\ca{S}^\bu_M,Q_M)$ a derived manifold.  We say that $Q_M$ is \emph{tangent} to a graded submanifold $(P,\ca{S}^\bu_P)$ if $Q_M$ preserves the ideal sheaf of $(P,\ca{S}^\bu_P)$.  If this is the case, then $Q_M$ restricts to a cohomological vector field $Q_P$ on $P$, making the inclusion $i\colon\ca{P}=(P,\ca{S}^\bu_P,Q_P)\to\ca{M}$ a morphism of derived manifolds.  The classical locus of $\ca{P}$ is then
\begin{equation}\label{equation;classical-submanifold}
\pi_0(\ca{P})=\pi_0(\ca{M})\cap P,    
\end{equation}
and for each $x\in\pi_0(\ca{P})$ the tangent complex $T^\bu_x\ca{P}$ is a subcomplex of $T^\bu_x\ca{M}$.  If $P$ is open in $M$, then the induced graded structure is $\ca{S}_P^\bu=\ca{S}_M^\bu|_P$ and $Q_M$ is automatically tangent to $P$, making $\ca{P}$ an \emph{open derived submanifold} of $\ca{M}$.

Now suppose the submanifold $P$ of $M$ is equipped with the induced graded structure.  
Then $Q_M$ is automatically tangent to $P$ and we call $(\ca{S}^\bu_P,Q_P)$ the \emph{induced derived structure} on $P$.  In this case the tangent complex at $x\in\pi_0(\ca{P})$ is
\begin{equation}\label{notation;tangent-complex-submanifold}
\begin{tikzcd}
T^\bu_x\ca{P}\colon\quad T^0_xP\ar[r,"\ell_{Q_P}"]&T^1_xP=T^1_xM\ar[r,"\ell_{Q_M}"]&T^2_xP=T^2_xM
\ar[r,"\ell_{Q_M}"]&\cdots
\end{tikzcd}
\end{equation}
where $\ell_{Q_P}=\ell_{Q_M}\big|_{T^0_xP}$. The map $H^k(T^\bu_x\ca{P})\to H^k(T^\bu_x\ca{M})$ induced by the inclusion is injective in degree $k=0$, surjective in degree $k=1$, and an isomorphism in degree $k\ge2$.
In particular, if $\ca{M}$ has amplitude $k\ge1$, then $\ca{P}$ has amplitude (at most) $k$, but if $\ca{M}$ has amplitude $0$, then $\ca{P}$ may have amplitude $1$.

\subsubsection{Products}\label{subsubsection;products}

The product of a graded chart $(O,\ca{S}^\bu_O)$ modelled on $(V_0,V_1,\dots)$ and a graded chart $(O',\ca{S}^\bu_{O'})$ modelled on $(V'_0,V'_1,\dots)$ is a graded chart $(O\times O',\ca{S}^\bu_{O\times O'})$ modelled on $(V_0\times V_0',V_1\times V_1',\dots)$.  It follows that the product of graded manifolds $(M,\ca{S}^\bu_M)$ and $(N,\ca{S}^\bu_N)$ is a graded manifold $(M\times N,\ca{S}^\bu_{M\times N})$ in a natural way.
Its graded tangent bundle is
\begin{equation}\label{equation;tangent-product}
T^\bu(M\times N)=\pr_M^*T^\bu M\oplus\pr_N^*T^\bu N.
\end{equation}
Let $Q_M$, resp.\ $Q_N$, be a cohomological vector field on $M$, resp.\ $N$.  Then $M\times N$ is equipped with a cohomological vector field $Q_M\oplus Q_N$, making $\ca{M}\times\ca{N}=(M\times N,\ca{S}^\bu_{M\times N},Q_M\oplus Q_N)$ a product of $\ca{M}=(M,\ca{S}^\bu_M,Q_M)$ and $\ca{N}=(N,\ca{S}^\bu_N,Q_N)$ in the category of derived manifolds.  The classical locus of $\ca{M}\times\ca{N}$ is 
\begin{equation}\label{equation;classical-product}
\pi_0(\ca{M}\times\ca{N})=\pi_0(\ca{M})\times\pi_0(\ca{N}),    
\end{equation}
and for each $(x,y)\in\pi_0(\ca{M}\times\ca{N})$ the tangent complex at $x$ is the direct sum
\begin{equation}\label{notation;tangent-complex-product}
T^\bu_{(x,y)}(\ca{M}\times\ca{N})=T^\bu_x\ca{M}\oplus T^\bu_y\ca{N}.
\end{equation}
This shows that the amplitude of $\ca{M}\times\ca{N}$ is equal to the maximum of the amplitudes of $\ca{M}$ and $\ca{N}$.

\subsection{Transversality and fibred products}\label{subsection;transverse}

Two morphisms of graded manifolds
\[
f\colon(M,\ca{S}^\bu_M)\to(P,\ca{S}^\bu_P),\qquad
g\colon(N,\ca{S}^\bu_N)\to(P,\ca{S}^\bu_P)
\]
are said to be \emph{transverse} if for all $(x,y,z)\in M\times
N\times P$ with $f(x)=g(y)=z$ we have
\begin{equation}\label{equation;transverse}
T^\bu_xf(T^\bu_xM)+T^\bu_yg(T^\bu_yN)=T^\bu_zP.
\end{equation}
Transversality allows us to form the fibred product $M\times_PN$,
which is a submanifold of $M\times N$ equal to the inverse image $(f\times g)^{-1}(\Delta)$ of the diagonal $\Delta$ of $P\times P$.  
Let $\ca{I}^\bu_\Delta\subset\ca{S}^\bu_{P\times P}$ be the ideal sheaf
of $\Delta$ and let $\ca{I}^\bu=\ca{S}^\bu_{M\times N}\cdot(f\times g)^\sharp(\ca{I}^\bu_\Delta)$ be the sheaf of ideals on $M\times N$ generated by the image of $\ca{I}^\bu_\Delta$.  
The support of the quotient sheaf $\ca{S}^\bu_{M\times_PN}=\ca{S}^\bu_{M\times N}/\ca{I}^\bu$ is the submanifold $M\times_PN$ and the pair $(M\times_PN,\ca{S}^\bu_{M\times_PN})$ is a graded submanifold of $M\times N$.  
A proof of the following statement can be found e.g.\ in~\cite[\S\,7.6]{vysoky;global-graded}.

\begin{proposition}\label{proposition;graded-fibred-product}
Let $f\colon(M,\ca{S}^\bu_M)\to(P,\ca{S}^\bu_P)$ and $g\colon(N,\ca{S}^\bu_N)\to(P,\ca{S}^\bu_P)$ be transverse morphisms of graded manifolds.  Then $(M\times_PN,\ca{S}^\bu_{M\times_PN})$ is a fibred product in the category of graded manifolds.  The graded tangent space to $M\times_PN$ at $(x,y)$
is
\begin{equation}\label{equation;fibred-tangent}
T^\bu_{(x,y)}(M\times_PN)=T^\bu_xM\times_{T^\bu_zP}T^\bu_yN.
\end{equation}
\end{proposition}

Now suppose that $M$, $N$, and $P$ are equipped with cohomological vector fields $Q_M$, $Q_N$, and $Q_P$ such that $f\colon\ca{M}=(M,\ca{S}^\bu_M,Q_M)\to\ca{P}=(P,\ca{S}^\bu_P,Q_P)$ and $g\colon\ca{N}=(N,\ca{S}^\bu_N,Q_N)\to\ca{P}$ are morphisms of derived manifolds.  
We call $f$ and $g$ \emph{transverse} if the underlying morphisms of graded manifolds are transverse.  
If this is the case, the cohomological vector field $Q_M\oplus Q_N$ of $M\times N$ is tangent to the graded submanifold $M\times_PN$, 
thus turning the latter into a derived manifold, which we denote by $\ca{M}\times_{\ca{P}}\ca{N}$. 

\begin{proposition}\label{proposition;derived-fibred-product}
Let $f\colon\ca{M}\to\ca{P}$ and $g\colon\ca{N}\to\ca{P}$ be transverse morphisms of derived manifolds.  Then $\ca{M}\times_{\ca{P}}\ca{N}$ is a fibred product in the category of derived manifolds $\DMfd$.  Its classical locus is
the fibred product of sets,
\begin{equation}\label{equation;fibred-classical}
\pi_0(\ca{M}\times_{\ca{P}}\ca{N})=\pi_0(\ca{M})\times_{\pi_0(\ca{P})}\pi_0(\ca{N})
,
\end{equation}
and its tangent complex at $(x,y)\in\pi_0(\ca{M}\times_{\ca{P}}\ca{N})$ is the fibred product of complexes,
\[T^\bu_{(x,y)}(\ca{M}\times_{\ca{P}}\ca{N})=T^\bu_x\ca{M}\times_{T^\bu_z\ca{P}}T^\bu_y\ca{N},\]
where $z=f(x)=g(y)$.  The amplitude of $\ca{M}\times_{\ca{P}}\ca{N}$ is less than or equal to the maximum of the amplitudes of $\ca{M}$, $\ca{N}$, and $\ca{P}$.
\end{proposition}

\begin{proof}
Let $f'\colon\ca{W}\to\ca{M}$ and $g'\colon\ca{W}\to\ca{N}$ be morphisms of derived manifolds such that $f\circ f'=g\circ g'$.  
By Prop.~\ref{proposition;graded-fibred-product} there is a unique morphism of graded manifolds $h\colon W\to M\times_PN$ with $\pi_M\circ h=f'$ and $\pi_N\circ h=g'$, where $\pi_M,\,\pi_N$ denote the natural morphisms from $M\times_PN$ to $M,N$, respectively. 
Composing $h$ with the inclusion $M\times_PN\to M\times N$ we get a map $k\colon W\to M\times N$, which is just the product of $f'$ and $g'$.   
Since the cohomological vector fields $Q_W$ and $Q_M\oplus Q_N$ are $k$-related, $k$ is a morphism of derived manifolds.  It follows that $h$ is a morphism of derived manifolds.  This establishes the universal property of $\ca{M}\times_{\ca{P}}\ca{N}$.  A point $(x,y)$ of $\ca{M}\times_{\ca{P}}\ca{N}$ is classical if and only if $Q_M(x)=Q_N(y)=0$.  Hence also $Q_P(z)=0$, where $z=f(x)=g(y)$, which implies~\eqref{equation;fibred-classical}.  The remaining assertions now follow from~\eqref{equation;fibred-tangent}.   
\end{proof}

\subsection{Differential forms and Cartan calculus}
\label{subsection;cartan}

Let $(M,\ca{S}_M^\bu)$ be a graded manifold.  A \emph{differential form of bidegree $(k,l)$}, or a \emph{$(k,l)$-form}, or a
\emph{$k$-form of internal degree $l$}, on $M$ is a collection of maps
\[
\alpha\colon\X^{i_1}(M)\times\X^{i_2}(M)\times\cdots\times
\X^{i_k}(M)\longto\ca{S}^{i_1+i_2+\cdots+i_k+l}(M),
\]
defined for all $(i_1,i_2,\dots,i_k)\in\Z^k$, satisfying
\begin{gather*}
\alpha(\dots,X_{i+1},X_i,\dots)=
-(-1)^{\abs{X_i}\abs{X_{i+1}}}\alpha(\dots,X_i,X_{i+1},\dots),\\
\alpha(fX_1,X_2,\dots,X_k)=(-1)^{\abs{f}l}f\alpha(X_1,X_2,\dots,X_k)
\end{gather*}
for all $X_i\in\X^\bu(M)$ and $f\in\ca{S}^\bu(M)$.  We refer to
$k\in\N$ as the \emph{de Rham degree} and to $l\in\Z$ as the
\emph{internal degree}.  We denote the set of forms of bidegree
$(k,l)$ on $M$ by $\Omega^{k,l}(M)$.  We have
$\Omega^{0,l}(M)=\ca{S}^l(M)$.  For each $k\in\N$ we make
$\Omega^{k,\bu}(M)=\bigoplus_{l\in\Z}\Omega^{k,l}(M)$ a graded
\emph{left} $\ca{S}^\bu$-module by defining
\[(f\alpha)(X_1,X_2,\dots,X_k)=f\alpha(X_1,X_2,\dots,X_k)\]
for $X_i\in\X^\bu(M)$ and $f\in\ca{S}^\bu(M)$.  For
$\alpha\in\Omega^{k,l}(M)$ and $\beta\in\Omega^{p,q}(M)$ we define the
\emph{wedge product} $\alpha\wedge\beta\in\Omega^{k+p,l+q}(M)$ by
\begin{multline*}
(\alpha\wedge\beta)(X_1,\dots,X_k,X_{k+1},\dots,X_{k+p})=\\
\sum_{\sigma\in\Sh(k,p)}\Ksgn(\sigma)
\alpha(X_{\sigma(1)},\dots,X_{\sigma(k)})
\beta(X_{\sigma(k+1)},\dots,X_{\sigma(k+p)}).
\end{multline*}
Here $\Sh(k,p)\subseteq S_{k+p}$ is the set of $(k,p)$-shuffles and
$\Ksgn(\sigma)=\pm1$ is the product of $\sgn(\sigma)$ and the Koszul
sign of the permutation
\[
(\alpha,\beta,X_1,\dots,X_k,X_{k+1},\dots,X_{k+p})\longmapsto
(\alpha,X_{\sigma(1)},\dots,X_{\sigma(k)},\beta,X_{\sigma(k+1)},\dots,
X_{\sigma(k+p)}),
\]
as defined in~\cite{berger;koszul}.  With these sign conventions the
wedge product is \emph{bigraded commutative} in the sense that
\[\beta\wedge\alpha=(-1)^{kp+lq}\alpha\wedge\beta.\]
(This is analogous to the sign convention
of~\cite{deligne-morgan;supersymmetry} for graded supercommutativity.)
For every open subset $U$ of $M$ (endowed with the induced graded
structure) we have a bigraded algebra of forms $\Omega^{\bu,\bu}(U)$,
and the presheaf $U\mapsto\Omega^{\bu,\bu}(U)$ is a sheaf on $M$
denoted by $\Omega_M^{\bu,\bu}$.  A function $f\in\ca{S}^l(U)$ defines
a $(1,l)$-form $df$ by the formula
\begin{equation}\label{equation;df}
df(X)=(-1)^{kl}X(f)
\end{equation}
for $X\in\X^k(U)$.  If $U$ is a graded chart, we can take $f$ to be
any of the coordinate functions $x_i$ or $\xi_j$, and we see
from~\eqref{notation;vector-field} that the forms
\[dx_i\in\Omega^{1,0}(U),\qquad d\xi_j\in\Omega^{1,\abs{\xi_j}}(U)\]
are a basis of the $\ca{S}^\bu$-module $\Omega^{1,\bu}(U)$ and a set
of generators of the $\ca{S}^\bu$-algebra $\Omega^{\bu,\bu}(U)$.

For a vector field $X\in\X^j(M)$ and a $(k,l)$-form $\alpha$ we define
the \emph{contraction} $\iota_X\alpha\in\Omega^{k-1,l+j}(M)$ and the
\emph{Lie derivative} $\Lie_X\alpha\in\Omega^{k,l+j}(M)$ by
\begin{align*}
(\iota_X\alpha)(X_2,\dots,X_k)&=(-1)^{jl}\alpha(X,X_2,\dots,X_k),\\
X(\iota_{X_1}\cdots\iota_{X_k}\alpha)&=
(-1)^{j(\abs{X_1}+\cdots+\abs{X_k})}\iota_{X_1}\cdots\iota_{X_k}\Lie_X\alpha\\
&\qquad\qquad+\sum_{l=1}^k(-1)^{j(\abs{X_1}+\cdots+\abs{X_{l-1}})}
\iota_{X_1}\cdots\iota_{[X,X_l]}\cdots\iota_{X_k}\alpha.
\end{align*}
In particular, $\Lie_Xf=X(f)$ for $f\in\Omega^{0,l}(M)$.  The
\emph{exterior derivative} $d\alpha\in\Omega^{k+1,l}(M)$ is defined by
\begin{multline*}
\iota_{X_0}\cdots\iota_{X_k}d\alpha=
\sum_{l=1}^k(-1)^{l+\abs{X_l}(\abs{X_0}+\cdots+\abs{X_{l-1}})}
\Lie_{X_l}\iota_{X_0}\cdots\widehat{\iota_{X_l}}\cdots\iota_{X_k}\alpha\\
+\sum_{1\le j<l\le
  k}(-1)^{j+1+\abs{X_j}(\abs{X_{j+1}}+\cdots+\abs{X_{l-1}})}
\iota_{X_0}\cdots\widehat{\iota_{X_j}}\cdots\iota_{[X_j,X_l]}\cdots
\iota_{X_k}\alpha,
\end{multline*}
which for $k=0$ coincides with~\eqref{equation;df}.  The next
proposition is a summary of the Cartan calculus for graded manifolds.
See~\cite[\S\,2.6]{cueca-maglio-valencia;symplectic-graded}
or~\cite[\S\,6]{vysoky;global-graded} for a proof and further
discussion.  By a \emph{bigraded left derivation of bidegree $(p,q)$}
of the algebra of forms we mean an $\R$-linear map
$D\colon\Omega^{\bu,\bu}(M)\to\Omega^{\bu+p,\bu+q}(M)$ satisfying
\[
D(\alpha\wedge\beta)=D\alpha\wedge\beta+(-1)^{kp+lq}\alpha\wedge
D\beta
\]
for $\alpha\in\Omega^{k,l}(M)$.  The \emph{bigraded commutator} of
operators $F$ of bidegree $(k,l)$ and $G$ of bidegree $(p,q)$ is
\[[F,G]=F\circ G-(-1)^{kp+lq}G\circ F.\]

\begin{proposition}\label{proposition;cartan}
The operators 
\begin{gather*}
\iota_X\colon\Omega^{\bu,\bu}(M)\longto\Omega^{\bu-1,\bu+\abs{X}}(M),
\qquad
\Lie_X\colon\Omega^{\bu,\bu}(M)\longto\Omega^{\bu,\bu+\abs{X}}(M),\\
d\colon\Omega^{\bu,\bu}(M)\longto\Omega^{\bu+1,\bu}(M)
\end{gather*}
are bigraded left derivations of the algebra $\Omega^{\bu,\bu}(M)$.
For all vector fields $X$, $Y\in\X^\bu(M)$ we have
\begin{gather*}
[\iota_X,\iota_Y]=0,\quad\iota_{[X,Y]}=[\Lie_X,\iota_Y],
\quad[\Lie_X,\Lie_Y]=\Lie_{[X,Y]},\\
[\Lie_X,d]=0,\quad\Lie_X=[\iota_X,d],\quad[d,d]=2d^2=0,
\end{gather*}
where the brackets denote bigraded commutators.
\end{proposition}

\subsubsection*{The de Rham double complex of a derived manifold}

If $\ca{M}=(M,\ca{S}_M^\bu,Q)$ is a derived manifold, it follows from Prop.~\ref{proposition;cartan} that
$d\Lie_Q=\Lie_Qd$ and $\Lie_Q^2=\Lie_{Q^2}=0$.  Thus $\Lie_Q$ is a differential, which we call the \emph{internal differential}, 
and $(\Omega^{\bu,\bu}(M),d,\Lie_Q)$ is a double complex equipped with the exterior derivative $d$ of bidegree $(1,0)$ and the internal differential 
$\Lie_Q$ of bidegree $(0,1)$.  For brevity we will denote this double complex by
\[\Omega^{\bu,\bu}(\ca{M})\]
and the corresponding double complex of sheaves by
$\Omega^{\bu,\bu}_{\ca{M}}$.

\begin{example}\label{example;derived-lie}
Let $\ca{Z}(M,E,s)$ be a derived zero locus (notation as in \S\,\ref{subsection;derived-zero}).  The cohomological vector field is $Q=\pardif{}{s}$.  For any section $f$ of the structure sheaf we have $\Lie_Qf=\pardif{f}{s}$.  In a graded chart $U$ with coordinates $x_i$, $\xi_j$ we express $s$ as a linear combination $s(x)=\sum_js_j(x)\xi_j^*$.  Then $\Lie_Qx_i=0$ and $\Lie_Q\xi_j=s_j=s^*\xi_j$, so $\Lie_Qdx_i=d\Lie_Qx_i=0$ and
\[
\Lie_Qd\xi_j=d\Lie_Q\xi_j=s^*d\xi_j=\sum_idx_i\,\pardif{s_j}{x_i}.
\]
\end{example}

\subsection{The category of fibrant objects \texorpdfstring{$\DMfd$}{DMfd}}
\label{subsection;fibrant}

In the following definitions $\ca{M}=(M,\ca{S}^\bu_M,Q_M)$ and
$\ca{N}=(N,\ca{S}^\bu_N,Q_N)$ denote derived manifolds.

\begin{definition}\label{def:fibDman}
A \emph{fibration} is a morphism $f\colon\ca{M}\to\ca{N}$ such that
the tangent map $$T^k_xf\colon T^k_x M\to T^k_{f(x)}N$$ is surjective for
all $x\in M$ and in every degree $k$.
\end{definition}

\begin{remarks}\label{remark;fibration}
\begin{enumerate}[wide,labelwidth=0pt,labelindent=0pt]
\item\label{item;fibration}
In particular ($k=0$) every fibration is a submersion.  
\item\label{item;compose}
If the composition $g\circ f$ of two morphisms is a fibration and $f$ is surjective, then $g$ is a fibration.
\item\label{item;mfd-dmfd}
A fibration between two manifolds is the same thing as a submersion.  (Here, as usual, we regard manifolds as a full subcategory of derived manifolds; see~\eqref{equation;mfd-dmfd}.)
\end{enumerate}
\end{remarks}

\begin{definition}
A \emph{weak equivalence} is a morphism $f\colon\ca{M}\to\ca{N}$ such
that the induced map on classical loci
$f\colon\pi_0(\ca{M})\to\pi_0(\ca{N})$ is a bijection and for each
$x\in\pi_0(\ca{M})$ the morphism of tangent complexes $$T^\bu_xf\colon
T^\bu_x\ca{M}\to T^\bu_{f(x)}\ca{N}$$ is a quasi-isomorphism,
i.e.\ induces isomorphisms $H^k(T^\bu_x\ca{M})\cong
H^k(T^\bu_{f(x)}\ca{N})$ for all $k$.
\end{definition}

We can now state the main result of~\cite{blx}.  See \S\,\ref{def:cfo} for the definition of a category of fibrant objects.

\begin{theorem}[\cite{blx}]\label{theorem;blx-cfo}
The category of derived manifolds $\DMfd$ with the fibrations and weak
equivalences just defined is a category of fibrant objects.
\end{theorem}

The most difficult part of the proof is the construction of path objects.  It turns out that if a derived manifold $\ca{M}$ has amplitude $[0,n]$, then its path object has amplitude $[0,n+1]$.  In particular, if $\ca{M}=M$
is an ordinary manifold, then its path object is quasi-smooth.  The following simple construction of a path object of an ordinary manifold $M$ will be frequently useful in computations.

\begin{lemma}\label{lemma;path}
\begin{enumerate}
\item\label{item;vector-bundle}
Let $M$ be a manifold, let $\pi_E\colon E\to M$ be a vector bundle, and let
$\epsilon_E$ be the tautological section of the bundle $\pi_E^*E$. The projection $\pi_E$ defines an acyclic fibration from the quasi-smooth derived manifold $\ca{Z}(\epsilon_E)=\ca{Z}(E,\pi_E^*E,\epsilon_E)$ to $M$ and the zero section $i\colon M\to E$ defines a weak equivalence from $M$ to $\ca{Z}(\epsilon_E)$.
\item\label{item;submanifold}
Let $M$ be a submanifold of a manifold $N$ and let $E=(TN|_{TM})/TM$ be the normal bundle of $M$.  The morphism $p\colon\ca{Z}(\epsilon_E)\to N$ defined by a tubular neighbourhood embedding of $E$ into $N$ is a fibration, and the composition
\[
\begin{tikzcd}
M\ar[r,"i"]&\ca{Z}(\epsilon_E)\ar[r,"p"]&N,
\end{tikzcd}
\]
where $i$ is the weak equivalence defined in~\eqref{item;vector-bundle}, is equal to the inclusion of $M$ into $N$.
\item\label{item;path}
Let $\pi\colon TM\to M$ be the tangent bundle of a manifold $M$ and let
$\epsilon_M$ be the Euler vector field of $TM$, regarded as a section
of the vertical tangent bundle $\pi^*TM=\ker(T\pi)$.  The quasi-smooth derived manifold
\[\ca{P}M=\ca{Z}(\epsilon_M)=\ca{Z}(TM,\pi^*TM,\epsilon_M)\]
is a path object for $M$ in the category of derived manifolds.  The
weak equivalence $i\colon M\to\ca{P}M$ is the zero section of $TM$ and
the fibration $p\colon\ca{P}M\to M\times M$ is a tubular neighbourhood
embedding of the diagonal.
\end{enumerate}
\end{lemma}

\begin{proof}
\eqref{item;vector-bundle}~%
The classical locus of $\ca{Z}(\epsilon_E)$ is $\pi_0(\ca{Z}(\epsilon_E))=Z(\epsilon_E)=i(M)$, so the projection $\pi_E$ and the zero section $i$ define morphisms $\pi_E\colon\ca{Z}(\epsilon_E)\to M$ and $i\colon M\to\ca{Z}(\epsilon_E)$.  The morphism $\pi_E$ induces surjections $T^0\ca{Z}(\epsilon_E)=TM\to T^0M=TM$ and $T^1\ca{Z}(\epsilon_E)=E\to T^1M=0$, so $\pi_E$ is a fibration.  Also $\pi_E$ induces
a bijection of classical loci
$\pi_E\colon\pi_0(\ca{Z}(\epsilon_E))=i(M)\to\pi_0(M)=M$.  Let $x\in M$.  The tangent
space to $E$ at $i(x)$ is the direct sum $T_{i(x)}E=T_xM\oplus E_x$.
The tangent map of $\pi_E$ at $i(x)\in i(M)$,
\[
\begin{tikzcd}[ampersand replacement=\&,row sep=large,column sep=large]
T^\bu_{i(x)}\ca{Z}(\epsilon_E)\ar[d,"T^\bu_{i(x)}\pi_E"']\colon\&T_xM\oplus
E_x\ar[r,"-d_x\epsilon_E"]
\ar[d,"{(\begin{smallmatrix}\id&0\end{smallmatrix})}"']\&
E_x\ar[d]\\
T^\bu_xM\colon\&T_xM\ar[r]\&0,
\end{tikzcd}
\]
is a quasi-isomorphism because $d_x\epsilon_E=(0,\id)$ is the projection onto $E_x$.  Thus $\pi_E$ is an acyclic fibration.  Since $i$ is a right inverse of $\pi_E$, it follows that $i$ is a weak equivalence.  

\eqref{item;submanifold}~The tangent map of $p$ at $i(x)\in TM$,
\[
\begin{tikzcd}[row sep=large,column sep=large]
T^\bu_{i(x)}\ca{Z}(\epsilon_E)\ar[d,"T^\bu_{i(x)}p"']\colon&T_xM\oplus E_x
\ar[r,"-d_x\epsilon_E"]\ar[d]&E_x\ar[d]\\
T^\bu_xN\colon&T_xN\ar[r]&0,
\end{tikzcd}
\]
is a surjection in every degree, so $p$ is a fibration, and
$p\circ i\colon M\to N$ is the inclusion map.

\eqref{item;path}~Since the tangent bundle of $M$ is isomorphic to the normal bundle of the diagonal in $M\times M$, this is a special case of~\eqref{item;submanifold}.
\end{proof}

\subsubsection{Linear differential forms}\label{subsection;linear}

Let $\pi\colon E\to M$ be a vector bundle.  A differential form $\phi$
on (the total space of) $E$ is \emph{linear} if it satisfies $\Lie_v\phi=\phi$, where $v$ is the Euler vector field on $E$, or equivalently if
$m_t^*\phi=t\phi$ for all $t\in\R$, where $m_t\colon E\to E$ is multiplication by $t$; see e.g.\ \cite[\S\,2]{BuCa12}. The \emph{linear part} of a form $\phi\in\Omega^k(E)$ is the linear $k$-form
\begin{equation}\label{equation;linear}
\phi_\lin=\frac{d}{dt}m_t^*\phi\Big|_{t=0}.
\end{equation}
A linear differential form $\phi$ on $E$ can be expressed in a vector bundle
chart as
\[\phi=\sum_j\chi_jy_j+\sum_k\psi_k\wedge dy_k,\]
where $y_j$ are linear coordinates along the fibres of $E$ and
$\chi_j$ and $\psi_k$ are forms on $M$.  Let $\epsilon=\epsilon_E$ be the tautological section of $\pi^*E$.  To the linear
differential form $\phi\in\Omega^k(E)$ corresponds a $(k,-1)$-form
$\epsilon_*\phi$ on the derived manifold $\huaZ(E,\pi^*E,\epsilon)$
defined by
\[
\epsilon_*\phi=\sum_j\chi_j\eta_j+\sum_k\psi_k\wedge d\eta_k,
\]
where $\eta_j$ are the variables of degree $-1$ corresponding to
$y_j$.  It follows from Ex.~\ref{example;derived-lie} that
$\Lie_Q\eta_j=\epsilon^*\eta_j=y_j$, so
\begin{equation}\label{equation;primitive}
\Lie_Q\epsilon_*\phi=\phi,
\end{equation}
where $Q=\iota_\epsilon$.  Thus, for linear differential forms
$\epsilon_*$ is a right inverse of the derivation $\Lie_Q$.  For
instance, let $E=T^*M$ be the cotangent bundle of $M$ and let
$\phi=\omega_\can$ be the canonical symplectic form.  Then
$\phi=\sum_idx_i\wedge dy_i$, $\epsilon_*\phi=\sum_idx_i\wedge
d\eta_i$, and $\Lie_Q\epsilon_*\phi=\phi$.

\section{Derived higher Lie groupoids}\label{section;derived-higher}

In this section we introduce higher derived Lie groupoids and prove that they form an incomplete category of fibrant objects (iCFO).  The first step is to introduce a pretopology on the category of derived manifolds.  There are many different pretopologies on this category, but we are aware of only a few that satisfy our requirement that they should be stalkwise with respect to a suitable jointly conservative collection of points.  The most flexible such pretopology is the pretopology of locally split fibrations introduced in \S\,\ref{subsection;topology}.  We present a jointly conservative collection of points and show that this pretopology is stalkwise (Thm.~\ref{pro:LSW-pret-dmfd}). In \S\,\ref{subsection;derived-higher} we introduce higher derived Lie groupoids.  It then follows immediately from Thm.~\ref{thm:RZ} and Thm.~\ref{pro:LSW-pret-dmfd} that they form an iCFO.  In \S\,\ref{subsection;ME} and \S\,\ref{subsection;pull} we define Morita equivalences and homotopy pullbacks of higher derived Lie groupoids.

\subsection{A pretopology and points for derived manifolds}\label{subsection;topology}

\subsubsection{Locally split fibrations}\label{subsubsection;lsf}

Let $f\colon\ca{M}\to\ca{N}$ be a morphism of derived manifolds $\huaM=(M,\huaS_M^\bu,Q_M)$ and $\huaN=(N,\huaS_N^\bu,Q_N)$.  A \emph{local splitting of $f$ at $y\in\huaN$} is a pair $(\ca{U},s)$, where $\ca{U}$ is an open neighbourhood of $y$ in $\ca{N}$ (equipped with the induced derived structure) and $s\colon\ca{U}\to\ca{M}$ is a morphism of derived manifolds such that $f\circ s=\id_{\ca{U}}$.  Following \cite[Definition 9.25]{meyer-zhu} we say that $f$ is \emph{locally split} if at every $y\in\huaN$ there exists a local splitting of $f$.

We say that $f$ is \emph{surjective} if $f$ is surjective as a map of sets $M\to N$.  We say that $f$ is \emph{essentially surjective} if the map of classical loci $\pi_0(\ca{M})\to\pi_0(\ca{N})$ induced by $f$ is surjective. If $f$ is locally split, then $f$ is both surjective and essentially surjective.  

A surjective fibration $f\colon M\to N$ between manifolds is the same thing as a surjective submersion (see Rem.~\ref{remark;fibration}\eqref{item;mfd-dmfd}), and therefore is locally split.  Indeed, for every $x\in M$ it has a local splitting $s$ at $y=f(x)$ satisfying $s(y)=x$.   Both these facts are false for derived manifolds for an obvious reason: if $f\colon\ca{M}\to\ca{N}$ is a surjective fibration of derived manifolds, $y\in\pi_0(\ca{N})$ is a classical point, and $s$ is a local splitting at $y$, then $s(y)$ must be a classical point of $\ca{M}$.  Thus there cannot exist a local splitting at $y$ if the fibre $f^{-1}(y)$ has no classical points, and if $x\in f^{-1}(y)$ is not a classical point, there cannot exist a local splitting which maps $y$ to $x$.  It is easy to produce examples of this.

\begin{example}\label{ep:counter-ep}
Let $P$ be a manifold and $g\colon P\to\R$ a smooth function, viewed as a section of the trivial line bundle $\underline{\R}$ over $P$.  The quasi-smooth derived manifold $\ca{P}=\ca{Z}(P,\underline{\R},g)$ has classical locus equal to $Z(g)$, the zero locus of $g$.  Let $s\colon\R\to\R$ be the identity map, viewed as a section of the trivial line bundle $\underline{\R}$ over $\R$.  Let $\ca{N}$ be the quasismooth derived manifold $\ca{Z}(\R,\underline{\R},s)$, let $\ca{M}=\ca{P}\times\ca{N}$, and let $f\colon\ca{M}\to\ca{N}$ be the projection onto the second factor.  Then $f$ is a surjective fibration, the classical loci are 
\[
\pi_0(\ca{M})=Z(g)\times\{0\}\qquad\text{and}\qquad\pi_0(\ca{N})=\{0\},
\]
and the fibre over the classical point is $f^{-1}(0)=P\times\{0\}$.  Suppose $g$ has no zeroes. Then $\pi_0(\ca{M})$ is empty, so $f$ is not essentially surjective and does not have a local splitting at $0$.  Now suppose $g$ has an isolated zero at $p_0\in P$.  Let $p\in\ca{P}$.  Then $f(p,0)=0$, but $f$ has a local splitting at $0\in\ca{N}$ that maps $0$ to $(p,0)\in f^{-1}(0)$ if and only if $p=p_0$.
\end{example}

\begin{lemma}\label{lemma;cover-sections}\mbox{}
Locally split fibrations serve as covers for a pretopology $\covers_\lsf$ on $\DMfd$.
\end{lemma}

\begin{proof}
This is a straightforward verification of the axioms \ref{P1}--\ref{P4} of \S\,\ref{sec:pre-points}.  An isomorphism is a locally split fibration, and the composition of two locally split fibrations with local splittings is locally split, so axioms \ref{P1} and \ref{P2} hold.  

Let $f\colon\ca{N}_1\to\ca{N}_2$ be a fibration and let $g\colon\ca{M}_2\to\ca{N}_2$ be any morphism, where $\ca{N}_1$, $\ca{N}_2$, and $\ca{M}_2$ are derived manifolds.  The fibred product $\ca{M}_1=\ca{M}_2\times_{\ca{N}_2}\ca{N}_1$ exists and the pullback morphism $\bar{f}=\pr_1\colon\ca{M}_1\to\ca{M}_2$ is a fibration, because $\DMfd$ is a CFO (Thm.~\ref{theorem;blx-cfo}).
Suppose $f$ is locally split.  Let $x\in M_2$ and let $s\colon\ca{V}\to\ca{N}_1$ be a local splitting of $f$ at $y=g(x)\in N_2$.  Let $\ca{U}=g^{-1}(\ca{V})$ and let $i\colon\ca{U}\to\ca{M}$ be the inclusion.  Then the morphism $i\times s\circ g\colon\ca{U}\to\ca{M}_2\times\ca{N}_1$ descends to a unique morphism $\bar{s}\colon\ca{U}\to\ca{M}_1$, which is a local splitting at $x$ of the fibration $\bar{f}\colon\ca{M}_1\to\ca{M}_2$. This proves that axiom \ref{P3} holds.

For $i=1$, $2$, let $f_i\colon\ca{M}_i\to\ca{N}_i$ be a fibration and let $s_i\colon\ca{U}_i\to\ca{M}_i$ be a local splitting of $f_i$ at $y_i\in M_i$.  Then $f=f_1\times f_2\colon\ca{M}_1\times\ca{M}_2\to\ca{N}_1\times\ca{N}_2$ is a fibration and $s_1\times s_2\colon\ca{U}_1\times\ca{U}_2\to\ca{M}_1\times\ca{M}_2$ is a local splitting of $f$ at $y=(y_1,y_2)$.  This proves that axiom \ref{P4} holds.
\end{proof}

\begin{caution}\label{remark;terminal}
The category of derived manifolds has a terminal object, namely the one-point space $*$, viewed as a $0$-dimensional manifold with trivial graded structure and zero cohomological vector field.  For every derived manifold $\ca{M}$ the unique morphism $\ca{M}\to*$ is a fibration, but it is not a cover (i.e.\ it does not have a splitting $*\to\ca{M}$) unless the classical locus of $\ca{M}$ is nonempty.  This is the reason why in the definition of a pretopology instead of the requirement \ref{P4'} we impose the weaker requirement \ref{P4} (see Rem.~\ref{rmk:p4}).
\end{caution}

\begin{remark}\label{rmk:man-to-dman}
Under the fully faithful embedding $i\colon\Mfd\to\DMfd$ (see \eqref{equation;mfd-dmfd}) the pretopology $\covers_\lsf$ restricts to the pretopology $\covers_\sursub$ whose covers are the surjective submersions (see Rem.~\ref{remark;fibration}).  The embedding $i$ \emph{reflects} and \emph{creates} covers in the sense that $i(f)$ is a cover if and only if $f$ is a cover.
\end{remark}

There is a more mundane pretopology on the category of derived manifolds, namely the pretopology $\covers_\open$ whose covers $c\colon\huaU\to\huaM$ are open covers in the usual sense: $\huaU$ is a disjoint union $\coprod_{i\in I}\huaU_i$ of derived manifolds $\huaU_i$ and for each $i$ we have an open embedding $c_i\colon\huaU_i\to\huaM$ such that the resulting morphism $c=\coprod_ic_i$ is surjective.  We will prefer the pretopology $\covers_\lsf$, as it appears to be more suitable for the construction of higher groupoids.  See App.~\ref{section;alternative} for a comparison with some other pretopologies.  However, the pretopology $\covers_\open$ has the following redeeming features.

\begin{lemma}\label{lemma;sheaves}
\begin{enumerate}
\item\label{refinement}
Every cover $f\colon\huaM\to\huaN$ in $\covers_\lsf$ has a refinement $c\colon\huaV\to\huaN$ in $\covers_\open$.
Let $\bar{c}\colon\huaU=f^*\huaV\to\huaM$ be the pullback open cover and let $\bar{f}\colon\huaU\to\huaV$ the fibration induced by $f$.  We can choose $c$ in such a way that $\bar{f}$ has a global splitting $\bar{s}\colon\huaV\to\huaU$ and the morphism $s=\bar{c}\circ\bar{s}$ satisfies $c=f\circ s$, as in the following commutative diagram. 
\[
\begin{tikzcd}
\huaU\ar[r,shift left=0.6ex,"\bar{f}"]\ar[d,"\bar{c}"']&
\huaV\ar[d,"c"]\ar[l,shift left=0.6ex,"\bar{s}"near end]\ar[dl,"s"]\\
\huaM\ar[r,"f"']&\huaN
\end{tikzcd}
\]
\item\label{sheaves}
A presheaf of sets on derived manifolds is a sheaf with respect to $\covers_\lsf$ if and only if it is a sheaf with respect to $\covers_\open$\textup: $\Sh(\DMfd,\covers_\lsf)=\Sh(\DMfd,\covers_\open)$.
\item\label{subcanonical}
The pretopology $\covers_\lsf$ is subcanonical.
\end{enumerate}
\end{lemma}

\begin{proof}
\eqref{refinement}~We follow the argument of \cite[Ex.~B.6]{Rogers-Zhu:2016}: since $f$ is a locally split fibration, there exists a cover $\{\huaV_i\}_{i\in I}$ of $\huaN$ by open derived submanifolds together with splittings $s_i\colon\huaV_i\to\huaM$ of $f$.  Let $\huaV$ be the disjoint union $\coprod_{i\in I}\huaV_i$.  The natural morphism $c\colon\huaV\to\huaN$ is an open cover, and the morphism $s\colon\huaV\to\huaM$ induced by the local splittings $s_i$ satisfies $c=f\circ s$.  Let $\huaU_i$ be the open derived submanifold $f^{-1}(\huaV_i)$ of $\huaM$ and let $\huaU=\coprod_i\huaU_i=f^*(\huaV)$.  The natural morphism $\bar{c}\colon\huaU\to\huaM$ is an open cover, the natural morphism $\bar{f}\colon\huaU\to\huaV$ is a fibration and has a global splitting $\bar{s}=\coprod_is_i$, and we have $c=f\circ s$.

\eqref{sheaves}~A sheaf with respect to $\covers_\lsf$ is a sheaf with respect to $\covers_\open$ because the pretopology $\covers_\lsf$ is stronger than $\covers_\open$
($\covers_\lsf\supseteq\covers_\open$).  Let $F$ be a sheaf with respect to $\covers_\open$ and let $f\colon\huaM\to\huaN$ be a locally split fibration.  Choose an open cover $c\colon\huaV\to\huaN$ and a factorization $c=f\circ s$ as in part~\eqref{refinement}.  We must show that the top line of the commutative diagram below is an equalizer diagram.  Since $c$ and $\bar{c}$ are open covers and $F$ is a sheaf with respect to open covers, we know that the middle and bottom lines are equalizer diagrams.  In particular $c^*$ and $\bar{c}^*$ are injective.
\begin{equation}\label{equation;equalizers}
\begin{tikzcd}
F(\huaN)\ar[r,"f^*"]\ar[d,"="']&F(\huaM)\ar[r,shift left=0.8ex,"\pr_1^*"]\ar[r,shift right=0.8ex,"\pr_2^*"']\ar[d,"s^*"']&F(\huaM\times_\huaN\huaM)\ar[d,"s^*"]
\\
F(\huaN)\ar[r,"c^*"]&F(\huaV)\ar[r,shift left=0.8ex,"\pr_1^*"]\ar[r,shift right=0.8ex,"\pr_2^*"']\ar[d,shift right=0.6ex,"\bar{f}^*"']&F(\huaV\times_\huaN\huaV)\ar[d,shift right=0.6ex,"\bar{f}^*"']
\\
F(\huaM)\ar[r,"\bar{c}^*"]&F(\huaU)\ar[r,shift left=0.8ex,"\pr_1^*"]\ar[r,shift right=0.8ex,"\pr_2^*"']\ar[u,shift right=0.6ex,"\bar{s}^*"']&F(\huaU\times_\huaM\huaU)\ar[u,shift right=0.6ex,"\bar{s}^*"']
\end{tikzcd}
\end{equation}
Let $A$ be a set and let $g\colon A\to F(\huaM)$ be a map with ${\pr_1^*}\circ g={\pr_2^*}\circ g$.  Then the map $s^*\circ g\colon A\to F(\huaV)$ satisfies
\[
{\pr_1^*}\circ s^*\circ g=s^*\circ{\pr_1^*}\circ g=s^*\circ{\pr_2}^*\circ g={\pr_2}^*\circ s^*\circ g,
\]
so by the equalizer property of the middle line of~\eqref{equation;equalizers} we obtain a unique map $\bar{g}\colon A\to F(\huaN)$ satisfying $c^*\circ\bar{g}=s^*\circ g$.  We assert that $f^*\circ\bar{g}=g$.  To see this, observe that the pair of morphisms $(\bar{c},s\circ\bar{f})\colon\huaU\to\ca{M}\times\ca{M}$ has the property 
\[
f\circ\bar{c}=c\circ\bar{f}=c\circ\bar{f}\circ\bar{s}\circ\bar{f}=f\circ\bar{c}\circ\bar{s}\circ\bar{f}=f\circ s\circ\bar{f},
\]
and therefore induces a morphism $h\colon\huaU\to\huaM\times_\huaN\huaM$.  Once again using ${\pr_1^*}\circ g={\pr_2^*}\circ g$ yields $h^*\circ{\pr_1^*}\circ g=h^*\circ{\pr_2^*}\circ g$; in other words $\bar{c}^*\circ g=\bar{f}^*\circ s^*\circ g$.  Substituting $s^*\circ g=c^*\circ\bar{g}$ gives
\[
\bar{c}^*\circ g=\bar{f}^*\circ c^*\circ\bar{g}=\bar{c}^*\circ f^*\circ\bar{g}.
\]
The map $\bar{c}^*$ is injective, so we obtain $g=f^*\circ\bar{g}$.  Finally, if $\bar{\bar{g}}\colon A\to F(\huaM)$ is any other map with $g=f^*\circ\bar{\bar{g}}$,
then we have $s^*\circ g=s^*\circ f^*\circ\bar{\bar{g}}$, so $c^*\circ\bar{g}=c^*\circ\bar{\bar{g}}$.  Since $c^*$ is injective, this implies $\bar{g}=\bar{\bar{g}}$.  Thus $f^*$ in the top line of~\eqref{equation;equalizers} is an equalizer.

\eqref{subcanonical}~The pretopology $\covers_\open$ is subcanonical (this is a consequence of the locality of morphisms: giving a morphism $f\colon\huaM\to\huaN$ of derived manifolds is tantamount to giving an open cover $\huaU_i$ of $\huaM$ and morphisms $f_i\colon\huaU_i\to\huaN$ that match on the overlaps) and therefore, by part~\eqref{sheaves}, $\covers_\lsf$ is subcanonical as well.
\end{proof}

\subsubsection{Points for derived manifolds}

We now introduce a jointly conservative collection of points (in the sense of \S\,\ref{sec:points}) for the category $\DMfd$ equipped with the pretopology $\covers_\lsf$.  Let $F$ be a presheaf of sets on $\DMfd$, let $\ca{X}$ be a derived manifold, and let $x\in\huaX$.  The \emph{stalk} of $F$ at $(\ca{X},x)$ is the set
\begin{equation}\label{eq:dm-point}
\ppt_{\ca{X},x}(F)=\colim_{\ca{U}\ni x}F(\ca{U}),
\end{equation}
where the colimit is taken over all open neighbourhoods $\ca{U}$ of $x$ (equipped with the induced derived structure).  The image of an element $\sigma\in F(\huaU)$ in $\ppt_{\huaX,x}(F)$ is the \emph{germ} $[\sigma]=[\sigma]_{\huaX,x}$ of $\sigma$.  The assignment $F\mapsto\ppt_{\ca{X},x}(F)$ defines a functor
\[\ppt_{\ca{X},x}\colon\PSh(\DMfd)\longto\Set.\]
We will denote the restriction of this functor to the category of sheaves $\Sh(\DMfd,\covers_\lsf)$ also by $\ppt_{\ca{X},x}$.

\begin{remark}\label{remark;ball}
Let $B^n_r$ be the open ball of radius $r$ about the origin in $\R^n$. Let $(B^n,\ca{S}_B^\bu)$ be a graded chart with body $B^n=B^n_1$ as in \S\,\ref{subsection;graded}, let $Q$ be a cohomological vector field on $B^n$, and let $\huaB^{n, Q}$ be the derived manifold $(B^n,\ca{S}_B^\bu,Q)$.  For each $r\le1$ let $\ca{B}^{n, Q}_r$ be the open submanifold $B^n_r$ furnished with the induced derived structure.  Then every functor $\ppt_{\ca{X},x}$ is isomorphic to a functor
\begin{equation} \label{eq:pt-ball}
\ppt_{\ca{B}^{n, Q}}(F)=\colim_{r\to0}F(\ca{B}^{n, Q}_r), 
\end{equation}
for a certain $\ca{B}^{n, Q}$.  Observe that the collection of functors~\eqref{eq:pt-ball} is a small set, whereas the collection~\eqref{eq:dm-point} is not.
\end{remark}

\begin{lemma}\label{lem:sheafification}
Let $F$ be a presheaf of sets on the category $\DMfd$ and let $F^\sharp$ be its sheafification with respect to the pretopology $\covers_\lsf$. Then for every derived manifold $\ca{X}$ and for every $x\in\ca{X}$ the natural map $F\to F^\sharp$ induces a bijection
\[\ppt_{\huaX,x}(F)\cong\ppt_{\huaX,x}(F^\sharp).\] 
\end{lemma}

\begin{proof}
The sheafification of $F$ is the presheaf $F^\sharp=(F^+)^+$, where $F^+$ is the presheaf defined by
\[
F^+(\ca{M}) = \colim_{\ca{U}\to\ca{M}} \{\,\sigma\in F(\ca{U})\mid\pr_1^*\sigma=\pr_2^*\sigma\,\};
\]
see e.g.\ \cite[\href{https://stacks.math.columbia.edu/tag/00W1}{Tag 00W1}]{stacks-project}.  Here the colimit is taken over all covers $\ca{U}\to\ca{M}$ of $\ca{M}$, and $\pr_1$, $\pr_2\colon\huaU\times_\huaM\huaU\to\huaU$ denote the canonical projections.
Let $\theta\colon F\to F^+$ be the canonical map.  We must show that $\theta^2\colon F\to F^\sharp$ induces a bijection $\ppt_{\ca{X},x}(\theta^2)\colon\ppt_{\ca{X},x}(F)\to\ppt_{\ca{X},x}(F^\sharp)$.

We assert that in fact $\ppt_{\ca{X},x}(\theta)\colon\ppt_{\ca{X},x}(F)\to\ppt_{\ca{X},x}(F^+)$ itself is a bijection.  By Rem.\ \ref{remark;ball}, it is enough to show this for $(\huaX,x)=(\huaB,0)$, where $\huaB=\huaB^{n,Q}$.

To show surjectivity, consider a germ $[\sigma]\in \ppt_\huaB(F^+)$, which is represented by an element $\sigma\in F^+(\huaB_r)$ for some $r>0$.  By the definition of $F^+$, the element $\sigma$ is in turn represented by an element $\sigma_\huaU \in F(\huaU)$ for some cover $c\colon\huaU\to\huaB_r$.  Since $c$ is locally split, there exists a splitting $s\colon\huaB_{r'}\to\huaU$ of $c$ for some $r'<r$.  Let $\huaU'$ be the open derived submanifold $c^{-1}(\huaB_{r'})$ of $\huaU$.  It follows that
\[
\sigma|_{\huaB_{r'}} =s^*c^*\sigma|_{\huaB_{r'}} =s^*\theta(\sigma_{\huaU}|_{\huaU'})=\theta(s^*\sigma_\huaU|_{\huaU'})=\theta(\sigma'),
\]
where $\sigma'=s^*\sigma_\huaU|_{\huaU'}\in F(\ca{B}_{r'})$.  Thus $[\sigma]=\theta([\sigma'])$ with $[\sigma']\in\ppt_\huaB(F)$. 

To show injectivity, suppose we have two germs $[\sigma]$, $[\sigma']\in\ppt_\huaB(F)$ with $\ppt_\huaB(\theta)([\sigma])=\ppt_\huaB(\theta)([\sigma'])$.  Then for a sufficiently small $r$ we can pick representatives $\sigma$, $\sigma'\in F(\huaB_r)$ so that $\theta(\sigma)=\theta(\sigma')\in F^+(\huaB_r)$.  This means that there exists a cover $c\colon\huaU\to\huaB_r$ such that $c^*\sigma=c^*\sigma'\in F(\huaU)$.  Choose $r'<r$ and a splitting $s\colon\huaB_{r'}\to\huaU$ of the cover $c$.  Then 
\[
\sigma|_{\huaB_{r'}}=s^*c^*\sigma=s^*c^*\sigma'=\sigma'|_{\huaB_{r'}},
\]
which yields $[\sigma]=[\sigma']\in\ppt_\huaB(F)$. 
\end{proof}

\begin{proposition}\label{prop:points}
\mbox{}
\begin{enumerate}
\item\label{point} The functors
$\ppt_{\huaX,x}\colon\Sh(\DMfd,\covers_\lsf)\to\Set$ defined in \eqref{eq:dm-point} are points.
\item\label{conservative} The collection of points 
\[
\pts_{\DMfd}=\{\,\ppt_{\,\huaX,x}\mid\text{$\huaX\in\DMfd$, $x\in\huaX$}\,\} 
\]
is jointly conservative.
\end{enumerate}
\end{proposition}

\begin{proof} 
\eqref{point}~Since the functor $\ppt_{\huaX,x}$ is a small filtered colimit and the limit of a diagram of sheaves is the same as its limit in the category of presheaves, the functor $\ppt_{\huaX,x}$ commutes with finite limits.

To show that $\ppt_{\huaX,x}$ 
also preserves small colimits of sheaves, consider a diagram of sheaves $F_i\in\Sh(\DMfd,\covers_\lsf)$ indexed by $i\in I$.  We have  
\begin{align*}
\colim_i\ppt_{\huaX,x}(F_i)&\cong\colim_i\colim_{\huaU\ni x}F_i(\huaU)\cong\colim_{\huaU\ni x}\colim_iF_i(\huaU)\cong\colim_{\huaU\ni x}\bigl(\colim_i{}^\PSh F_i\bigr)(\huaU)
\\
&\cong\ppt_{\huaX,x}\bigl(\colim_i{}^\PSh F_i\bigr)\cong\ppt_{\huaX,x}\bigl(\bigl(\colim_i{}^\PSh F_i\bigr)^\sharp\bigr)\cong\ppt_{\huaX,x}\bigl(\colim_i{}^\Sh F_i\bigr).
\end{align*}
Here we used that small colimits commute with each other and that the colimit in the category of sheaves (denoted by $\colim^{\Sh}$) is the sheafification of the colimit in the category of presheaves (denoted by $\colim^{\PSh}$).  

\eqref{conservative}~It suffices to show that the collection of points $\{\ppt_{\huaB^{n,Q}}\}$ is jointly conservative.  Let $\phi\colon F\to G$ be a morphism of sheaves such that for all ${\huaB^{n, Q}}$ the map 
 \[\ppt_{\huaB^{n, Q}}(\phi)\colon\ppt_{\huaB^{n, Q}}(F)\longto\ppt_{\huaB^{n, Q}}(G)\]
is a bijection. We must show that $\phi_\huaM\colon F(\huaM)\to G(\huaM)$ is an isomorphism for all derived manifolds $\huaM$. This is the usual ``stalkwise isomorphism implies isomorphism of sheaves'' argument.  Full details in the setting of a pretopology on Banach spaces are given in \cite[Prop.~4.2]{Rogers-Zhu:2016}. The proof there carries over verbatim once one replaces the Banach open balls $B_V(r)$ by our open balls $\huaB^{n, Q}_r$.  
\end{proof}

\begin{caution}
Lemma~\ref{lem:sheafification} and Prop.~\ref{prop:points} are versions of well-known results of classical sheaf theory, but they are false for many singleton pretopologies other than $\covers_\lsf$.  See App.~\ref{section;alternative} for counterexamples.  
\end{caution}

\begin{remark}
Let us call a point $\ppt_{\huaX,x}$ \emph{classical} if $x$ is a classical point of $\huaX$.  It is plain that the collection of classical points is not jointly conservative: for a sheaf morphism to be an isomorphism it must induce an isomorphism of stalks at all points, not just the classical points. 
\end{remark}

\subsubsection{Locally stalkwise pretopology}

The main result of \S\,\ref{subsection;topology} is Thm.~\ref{pro:LSW-pret-dmfd} below.  The proof follows the pattern of \cite[\S\,6.2]{Rogers-Zhu:2016}. 

\begin{lemma}\label{lem:stalkwise-surj-prop}
A morphism of derived manifolds $f\colon\huaM\to\huaN$ is locally split if and only if the sheaf morphism $\yon(f)\colon\yon(\huaM)\to\yon(\huaN)$ is a stalkwise surjection.  In particular, if $\yon(f)$ is a stalkwise surjection, then $f$ is surjective and essentially surjective.
\end{lemma}

\begin{proof}
Suppose $\yon(f)$ is a stalkwise surjection.  Let $y\in\huaN$.  Then $\ppt_{\huaN,y}(f)\colon\ppt_{\huaN,y}(\huaM)\to\ppt_{\huaN,y}(\huaN)$ is surjective, which by definition (see~\eqref{eq:dm-point}) tells us that $y$ has an open neighbourhood $\huaV$ in $\huaN$ such that the inclusion $\huaV\to\huaN$ lifts to a morphism $s\colon\huaV\to\huaM$, i.e.\ a local splitting of $f$.  Thus $f$ is locally split.  Now suppose $f$ is locally split.  Let $\ppt_{\huaX,x}$ be a point of $\DMfd$, let $\huaU$ be an open neighbourhood of $x$ in $\huaX$, and let $g\colon\huaU\to\huaN$ be a morphism.  A local splitting $s\colon\huaV\to\huaM$ of $f$ defined in a neighbourhood $\huaV$ of $y=g(x)\in\huaN$ allows us to lift the germ $[g]_{\huaX,x}\in\yon(\huaN)(\huaU)$ to the germ $[s\circ g]_{\huaX,x}\in\yon(\huaM)(\huaU)$, showing that $\yon(f)$ is stalkwise surjective.
\end{proof}

\begin{theorem}\label{pro:LSW-pret-dmfd}
The pretopology $\covers_\lsf$ on the category of derived manifolds is locally stalkwise with respect to the jointly conservative collection of points $\pts_{\DMfd}$ defined in Prop.~\ref{prop:points}.
\end{theorem}

\begin{proof}
To show that \ref{LSP1} is satisfied, consider morphisms of derived manifolds $f\colon\huaM\to\huaN$ and $g\colon\huaN\to\huaP$.  Suppose that $\yon(f)$ is a stalkwise surjection and that the composition $g\circ f$ is a cover.  Then $\yon(g)$ is a stalkwise surjection by Rem.~\ref{rk:2-out-of-3}\eqref{item;2-out-of-3} and $g$ is a fibration by Rem.~\ref{remark;fibration}\eqref{item;compose}.  It now follows from Lemma~\ref{lem:stalkwise-surj-prop} that $g$ is a cover.

Next we verify \ref{LSP2}.  Let $f\colon\huaM\to\huaP$ and $g\colon\huaN\to\huaP$ be morphisms of derived manifolds.  For every derived manifold $\huaU$ define $F(\huaU)$ to be the set of pairs $h=(h_\huaM,h_\huaN)$ of morphisms $h_\huaM\colon\huaU\to\huaM$ and $h_\huaN\colon\huaU\to\huaN$ satisfying $f\circ h_\huaM=g\circ h_\huaN$.  Then $F$ is a sheaf of sets on the category of derived manifolds and it is a fibred product in the category of sheaves,
\[
F=\yon(\huaM)\times_{\yon(\huaP)}\yon(\huaN).
\]
Suppose that $\yon(g)\colon\yon(\huaN)\to\yon(\huaP)$ is a  stalkwise surjection and that the morphism
\[
\phi\colon F\longto\yon(\huaN)
\]
induced by $\yon(f)\colon\yon(\huaM)\to\yon(\huaN)$ is a locally stalkwise cover, i.e.\ there exist a derived manifold $\huaQ$, a cover $k\colon\huaQ\to\huaN$ and a stalkwise surjective sheaf morphism $\zeta\colon\yon(\huaQ)\to F$ such that $\phi\circ\zeta=\yon(k)\colon\yon(\huaQ)\to\yon(\huaN)$.  We must prove that $f$ is a cover.

To simplify the notation, in the following commutative diagram we identify the category $\DMfd$ with its image under the Yoneda embedding $\yon$, which is a full subcategory of the category of sheaves $\Sh(\DMfd,\covers_\lsf)$ because the pretopology $\covers_\lsf$ is subcanonical (Lemma~\ref{lemma;cover-sections}).  The square in the middle is a pullback diagram.  We define $\tilde{g}\colon\huaQ\to\huaM$ to be the unique morphism satisfying $\yon(\tilde{g})=\psi\circ\zeta$.
\begin{equation}\label{equation;stalkwise}
\begin{tikzcd}[row sep=scriptsize]
\huaQ\arrow[drr,bend left,"k"]\arrow[ddr,bend right,"\tilde{g}"]\arrow[dr,"\zeta"]&&
\\
&F\arrow[r,"\phi"]\arrow[d,"\psi"]&
\huaN\arrow[d,"g"]
\\
&\huaM\arrow[r,"f"]&\huaP
\\
\huaU\arrow[ur,dashed,"j_\huaU"]\arrow[uuu,dashed,bend left,"\tilde{h}"]\arrow[rrr,dashed,"f_\huaU"]&&&
\huaW\arrow[ul,dashed,"j_\huaW"']\arrow[uul,dashed,bend right,"s"']
\end{tikzcd}
\end{equation}
Since $k$ is a cover and $g$ is locally split (Lemma~\ref{lem:stalkwise-surj-prop}), both $k$ and $g$ are surjective.  It follows that $f\circ\tilde{g}=g\circ k$ is surjective, and hence $f$ is surjective.  Let $x\in\huaM$ and let $p=f(x)\in\huaP$.  We show that $f$ is a fibration by showing that it is a graded submersion at $x$.  Choose an open neighbourhood $\huaW$ of $p$ (equipped with the induced derived structure) and a local splitting $s\colon\huaW\to\huaN$ of $g$.  Let $\huaU$ be any open neighbourhood of $x$ contained in $f^{-1}(\huaW)$ (again with the induced derived structure), let $j_\huaU\colon\huaU\to\huaM$, $j_\huaW\colon\huaW\to\huaP$ be the respective inclusions, and let $f_\huaU\colon\huaU\to\huaW$ be the restriction of $f$.  Let $h=(h_\huaM,h_\huaN)\colon\huaU\to\huaM\times\huaN$ be the morphism given by $h_\huaM=j_\huaU$, $h_\huaM=s\circ f_\huaU$.  Then $h$ is an element of $F(\huaU)$.  Since $\zeta$ is stalkwise surjective, there exists a germ $[\tilde{h}]_{\huaU,x}$ of $\yon(Q)$ such that $\zeta([\tilde{h}]_{\huaU,x})$ is equal to the germ $[h]_{\huaU,x}$.  This means that (after replacing $\huaU$ with a smaller open neighbourhood of $x$ if necessary) we have a morphism $\tilde{h}\colon\huaU\to\huaQ$ satisfying $\tilde{g}\circ\tilde{h}=j_\huaU$ and $k\circ\tilde{h}=s\circ f_\huaU$.  Let $q=\tilde{h}(x)\in\huaQ$; then $\tilde{g}(q)=x$.  Also let $y=s(p)=h(q)\in\huaN$.  Differentiating the identity $f\circ\tilde{g}=g\circ k$ at $q$ gives
\[
T_x^lf\circ T_q^l\tilde{g}=T_y^lg\circ T_q^lk\colon T_x^l\huaM\longto T_p^l\huaP
\]
for all $l\ge0$.  Since $k$ is a fibration and $g$ has a local splitting passing through $y$, both maps $T_q^lk$ and $T_y^lg$ are surjective.  It follows that $T_x^lf$ is surjective, i.e.\ $f$ is a fibration.

To see that $f$ has a local splitting at $p\in\huaP$, take an open neighbourhood $\huaV$ of $y$ in $\huaN$ and a local splitting $t\colon\huaV\to\huaQ$ of $k$.  Replace $\huaW$ with $\huaW\cap g^{-1}(\huaV)$ and $\huaU$ with $(s\circ f_\huaU)^{-1}(\huaV)$, and put $\bar{s}=\tilde{g}\circ t\circ s\colon\huaW\to\huaU$.  Using $f\circ\tilde{g}=g\circ k$ gives $f_\huaU\circ\bar{s}=\id_\huaW$.
\end{proof}

The following properties are immediate from the proof of Thm.~\ref{pro:LSW-pret-dmfd}: the fibred product $\huaF=\huaM\times_\huaP\huaN$ in \eqref{equation;stalkwise} exists, the sheaf $F$ is represented by the derived manifold $\huaF$, the morphism $\bar{f}=\yon^{-1}(\phi)\colon\huaF\to\huaN$ is a cover, and the morphism $\bar{g}=\yon^{-1}(\psi)\colon\huaF\to\huaM$ is locally split.  The morphisms $g$ and $\bar{g}$ are not necessarily fibrations.

\subsection{The iCFO of derived Lie \texorpdfstring{$n$}{n}-groupoids}\label{subsection;derived-higher}

We can now state one of the main definitions of this article.

\begin{definition}\label{def:derived-Lie-n-gpd-gp}
Let $n\in\N\cup\{\infty\}$.  A  \emph{derived Lie $n$-groupoid} is an $n$-groupoid object $\huaG_\bu$ in the category of derived manifolds $\DMfd$ equipped with the pretopology $\covers_\lsf$.  A \emph{derived Lie $n$-group} is a derived Lie $n$-groupoid $\huaG_\bu$ with $\huaG_0=*$, the terminal object.  We denote by $\mathsf{Gpd}_n[\DMfd, \covers_\lsf]$ 
the category whose objects are derived Lie $n$-groupoids and whose morphisms are morphisms of simplicial objects in $\DMfd$.    
\end{definition}

\begin{remark}\label{rmk:gp}
A \emph{derived Lie group} is a group object in $\DMfd$, that is a derived manifold $\huaG$ together with three morphisms: multiplication $m\colon\huaG\times\huaG\to\huaG$, inverse $i\colon\huaG\to\huaG$ and identity $e\colon*\to \huaG$, which satisfy the usual equations
\begin{align*}
       m\circ(\id_{\huaG} \times m)&=m\circ(m\times \id_{\huaG}),\\
        m\circ(\id_{\huaG}\times e)&=\id_{\huaG}=m\circ(e\times \id_{\huaG}),\\
        m\circ(\id_{\huaG}\times i)\circ\diag&=e\circ t=m\circ(i\times \id_{\huaG})\circ\diag,
\end{align*}
where $t\colon\huaG\to*$ is the terminal map and $\diag\colon\huaG\to \huaG\times\huaG$ is the diagonal map.  The simplicial nerve $N_\bu\huaG$ of a derived Lie group $\huaG$ is a derived Lie $1$-group: the only Kan condition to check is that $t\colon \huaG\to *$ is a locally split fibration, which is true because the identity $e\colon *\to\huaG$ provides a splitting; cf.\ Rem.~\ref{remark;terminal}.  Thus we have a one-to-one correspondence between derived Lie groups and derived Lie $1$-groups analogous to the correspondence between Lie groups and Lie $1$-groups (\cite[\S\,1]{z:tgpd-2}).
\end{remark}

\begin{remark}\label{ex:lie=dlie}
Via the fully faithful embedding $\Mfd\to\DMfd$ (see \eqref{equation;mfd-dmfd}) we can regard every simplicial manifold as a simplicial derived manifold.  It follows from Rem.~\ref{remark;fibration}\eqref{item;mfd-dmfd} and Rem.~\ref{rmk:man-to-dman} that a simplicial manifold is a derived Lie $n$-groupoid if and only if it is a Lie $n$-groupoid in the sense of \cite{z:tgpd-2}.
\end{remark}

Recall that a \emph{VB-groupoid}  is a Lie groupoid in the category of vector bundles; see \cite[\S 11.2]{MK2}.  It can be represented by a commutative diagram
\begin{equation}\label{equation;vb}
\begin{tikzcd}
E_1\ar[d,"\pi_1"']\ar[r,shift left]\ar[r,shift right]&E_0\ar[d,"\pi_0"]\\
G_1\ar[r,shift left]\ar[r,shift right]&G_0,
\end{tikzcd}
\end{equation}
where the columns are vector bundles, the rows are Lie groupoids, and the groupoid operations in the top row are linear with respect to the vector bundle structures.
A \emph{multiplicative section} of the VB-groupoid~\eqref{equation;vb} is a Lie groupoid morphism $s_1\colon G_1\to H_1$ covering $s_0\colon G_0\to E_0$ such that $s_1$ is a section of the vector bundle $E_1$ and $s_0$ of $E_0$. Let us denote by $\mathsf{VBms}$ the category whose objects are VB-groupoids equipped with a multiplicative section and the morphisms are VB-groupoid morphisms that commute the given multiplicative sections. Then we have the following groupoid version of Prop.~\ref{prop:quasi-smooth}.

\begin{proposition}\label{ex:VBgrpd}
The functor $\ca{Z}\colon\mathsf{VBms}\to\mathsf{Gpd}_1[\DMfd, \covers_\lsf]$ defined by
\[
\ca{Z}(G_\bu,E_\bu,s_\bu):=\big(\ca{Z}(G_1,E_1,s_1)\rightrightarrows\ca{Z}(G_0,E_0,s_0)\big)
\]
is fully faithful and its essential image is the category of quasi-smooth derived Lie $1$-groupoids.
\end{proposition}

\begin{proof}
The structure maps of this groupoid are given by those of the VB-groupoid. The fact that $(s_1,s_0)$ is a multiplicative section ensures that the structure maps are morphisms of derived manifolds.  Therefore, the nerve of $\ca{Z}(G_\bu,E_\bu,s_\bu)$ is a simplicial derived manifold.  For the Kan conditions,  observe that the source and target maps are covers in $\covers_{\lsf}$ because they are surjective submersions of manifolds and the unit of the VB-groupoid provides the existence of local splittings at each point. 
\end{proof}

\begin{example}\label{ep:odd-line}
A particularly relevant example of a multiplicative section of a VB-groupoid is the zero section.  This special case includes 
\begin{itemize}
    \item the \emph{odd line} $\R[-1]\rightrightarrows*$, which is a derived Lie group;
    \item Severa's \emph{fat point} \cite{severa:diff, lrwz} $\R[-1]\times\R[-1]\rightrightarrows \R[-1]$, which is defined as the pair groupoid of the odd line $\R[-1]$.
\end{itemize}
More examples of derived Lie groupoids will appear in \S \ref{sec:red}.
\end{example}

An immediate consequence of  Thm.~\ref{thm:RZ} and Thm.~\ref{pro:LSW-pret-dmfd} is the following.

\begin{theorem}\label{thm:iCFO-n-details} 
Let $n\in\N\cup\{\infty\}$. The category $\mathsf{Gpd}_n[\DMfd, \covers_\lsf]$ is an iCFO in which:
\begin{itemize}
\item the weak equivalences are the stalkwise weak equivalences (Def.\
  \ref{def:stalk_weq}), or equivalently defined in Cor.~\ref{cor:w-eq-comb}; 

\item the fibrations are the Kan fibrations (Def.\ \ref{def:Kan_arrow}); 

\item the acyclic fibrations are hypercovers (Def.\ \ref{def:equivalence}), or equivalently defined in Cor.~\ref{cor:hypercover-n};

\item the path object of a derived Lie $n$-groupoid $\huaG_\bu$ is $\huaG_\bu^I$ defined by 
\begin{equation*}
    \huaG^I_k = \Hom(\Simp{k}\times \Simp{1}, \huaG_\bu). 
\end{equation*}
\end{itemize}
\end{theorem}

\begin{remark}\label{rmk:path}
More explicitly, the path object of a derived Lie $n$-groupoid $\huaG_\bu$ is given by
\[
\huaG^I_k=\huaG_{k+1}\times_{d_1, \huaG_k, d_1} \huaG_{k+1} \times_{d_2, \huaG_k, d_2} \dots \times_{d_k, \huaG_k, d_k} \huaG_{k+1}.
\]
This is a derived manifold because the face maps $d_i$ of $\huaG_\bu$ are covers; see \cite[Lemmas~3.6--3.7]{Rogers-Zhu:2016} or \cite[Lemma~2.44]{li:thesis}.  It follows from \cite[Prop.~7.2, Rem.~7.6]{Rogers-Zhu:2016} that $\huaG_\bu^I$ is a derived Lie $n$-groupoid.
\end{remark}

Thus we have a tower of iCFO's,
\begin{equation*} 
    \DMfd=\mathsf{Gpd}_0[\DMfd, \covers_\lsf] \subset \mathsf{Gpd}_1[\DMfd, \covers_\lsf] 
\subset \dots \subset  \mathsf{Gpd}_n[\DMfd, \covers_\lsf] \subset \dots \subset \gpd{\DMfd, \covers_\lsf}. 
\end{equation*}

\begin{remark}\label{rmk:2icfo}
The category $\DMfd$ carries two completely different iCFO structures.  The first one is given by Thm.~\ref{theorem;blx-cfo} with:
    \begin{itemize}
        \item fibrations given by submersions; and
        \item the weak equivalences are the maps that induce a  bijection on classical loci and quasi-isomorphism between the tangent complexes,
    \end{itemize}
       while the second one is given by Thm.~\ref{thm:iCFO-n-details}, where
       \begin{itemize}
           \item any map is a fibration; and
           \item the weak equivalences are the isomorphisms between derived manifolds.
       \end{itemize}
The first iCFO is a CFO while the second one is not.  The second iCFO does not see weak equivalences in the first CFO.  This leads to the fact that our Morita equivalences of derived Lie $n$-groupoids defined in \S\,\ref{subsection;ME} contain both weak equivalences in the first CFO of $\DMfd$ as well as hypercovers in the iCFO of derived Lie $n$-groupoids.   Another possible way is to use Pridham's homotopy hypergroupoid \cite{Pridham:higher-stack} to integrate the weak equivalence directly into a homotopy version of hypercovers; see Rem.~\ref{rmk:ME}. 
\end{remark}

\subsection{Morita equivalence of derived Lie \texorpdfstring{$n$}{n}-groupoids}\label{subsection;ME}

Recall that for Lie $n$-groupoids, once we construct the iCFO,  we may define their Morita equivalence to be a span of acyclic fibrations. This follows from the fact that the simplicial localization (i.e. the underlying $\infty$-category) of an iCFO has a simple description in terms of the nerve of a category of spans \cite[Thm.~2.13]{Rogers-Zhu:2016}. 

However, in the case of derived Lie $n$-groupoids this is not enough. Since derived manifolds themselves form a CFO, see \S \ref{subsection;fibrant},  we want to incorporate the weak equivalence in $\DMfd$ into the Morita equivalence for derived Lie $n$-groupoids, see Rem.\ \ref{rmk:2icfo}.

\begin{definition}\label{def:Morita}
A morphism $f_\bu\colon \huaG_\bu\to\huaH_\bu$ between derived Lie $n$-groupoids is called a \emph{levelwise weak equivalence} if $f_i\colon  \huaG_i \to \huaH_i$  is a weak equivalence between derived manifolds for every $i\in\N$.  We call a morphism a \emph{Morita morphism} if it is either a levelwise weak equivalence or a weak equivalence 
in $\dnGpd$. Morita equivalence is an equivalence relation generated by Morita morphisms, that is, two derived Lie $n$-groupoids $\huaG_\bu, \huaH_\bu \in \dnGpd$ are \emph{Morita equivalent} if there is a sequence of maps in $\dnGpd$, 
\begin{equation}\label{eq:zig-zag}
\huaG_\bu \leftarrow  \huaG_\bu^1  \to  \huaG_\bu^2 \leftarrow\dots \leftarrow\huaG_\bu^k \to\huaH_\bu, 
\end{equation}
each of which is a Morita morphism towards the left or right. Further, a {\em generalised morphism} $\huaG_\bu \dashrightarrow \huaH_\bu$ is made of a sequence of maps as in \eqref{eq:zig-zag} such that the ones going left are Morita morphisms.  
\end{definition} 

\begin{remark}\label{rmk:ME}

To establish a good notion of Morita equivalence for \dgngpds\ is to make sure that Morita equivalent \dgngpds\ present the same derived $n$-stack. This in particular requires the theory of derived higher differentiable stacks, see \cite{Pridham:higher-stack, toen-vezzosi:hagI} for the algebraic case. 

One way to proceed it is to embed the CFO of $\DMfd$ into $\dCalg^\op$, the opposite category of non-positively graded dg $C^\infty$-algebras as introduced in \cite[Def.3.18]{carchedi:23}.  It was shown in \cite{carchedi-roytenberg:2} that $\dCalg^\op$ forms a model category and \cite{carchedi:23} demonstrates that weak equivalence in $\DMfd$ goes to weak equivalence in $\dCalg^\op$ \cite[Cor.~5.17]{carchedi:23} and the embedding mentioned above extends to an equivalence as $\infty$-categories \cite[Thm.~1.5]{carchedi:23}. 

Furthermore, if one modifies the category $\dCalg^\op$ a little bit to $\dCalg^\op_0$ according to \cite[Prop.~3.12]{pridham:derived-analytic}, one should expect that the embedding  $\DMfd \to \dCalg^\op_0$ is not only  fully faithful, preserving and detecting weak equivalences, but also carrying fibrations to fibrations \cite{steffens;private,pridham;private}.

Next, we may follow \cite{Pridham:higher-stack} and introduce homotopy hypergroupoid in $(\dCalg^\op_0, \covers_\new)$, for a suitable pretopology $\covers_\new$, and  homotopy hypercovers between them. Then the relative category of  homotopy derived  Lie $n$-hypergroupoids $\dgnGpdCscho$ localized at homotopy hypercovers should model the $\infty$-category of derived differentiable $n$-stacks.  

Using the technique of collapsible extensions, one shows that homotopy Kan conditions reduce to usual Kan conditions for $\dnGpd$. Thus based on the embedding of $\DMfd \to \dCalg^\op_0$, we can further embed $\dnGpd$ to $\dgnGpdCscho$. It turns out that, restricting to $\dnGpd$, a usual hypercover and a levelwise weak equivalence are both examples of homotopy hypercovers. 

In theory, Morita equivalence in $\dnGpd$ should be zigzags of homotopy hypercovers travelling though homotopy derived Lie $n$-hypergroupoids. The Morita equivalence would be significantly simplified thanks to Reedy fibrant objects \cite{getzler;private}. Hence, Morita equivalence for $\dnGpd$ might be more general than Def.~\ref{def:Morita} above. Nevertheless, what we define is certainly right and is general enough to cover all equivalences in the examples. 

To work out the entire theory of derived  Lie $n$-hypergroupoids is beyond the scope of this work since our focus are shifted symplectic structures on derived Lie $n$-groupoids.  We thus delay the part on hypergroupoids and a more detailed study of general Morita equivalence to a future work. 
\end{remark}

\subsection{Coarse quotients and stabilizers}\label{subsection:coarse}

The simplest Morita invariants of a higher derived Lie groupoid are its coarse quotient and the stabilizers of its classical points.  Let $\huaG_\bu$ be a derived Lie $n$-groupoid.  For each $i$ we regard the classical locus $\pi_0(\huaG_i)$ of the derived manifold $\huaG_i$ as a topological subspace of the body of $\huaG_i$.  The object $\pi_0(\huaG_\bu)$ is then a topological $n$-groupoid.  The \emph{orbit space} or \emph{coarse quotient} of $\huaG_\bu$ is the set
\[\pi_0(\huaG_0)/\pi_0(\huaG_1)\]
obtained by identifying two classical points $x_0$ and $x_1\in\pi_0(\huaG_0)$ if there exists $g\in\pi_0(\huaG_1)$ with $d_0(g)=x_0$ and $d_1(g)=x_1$.  We equip the coarse quotient with the quotient topology defined by this equivalence relation (which will frequently be non-Hausdorff).

\begin{lemma}\label{lemma;coarse}
Let $\huaG_\bu$ be a derived Lie $n$-groupoid.  A Morita equivalence $\huaG_\bu\simeq\huaH_\bu$ of derived Lie $n$-groupoids induces a homeomorphism 
\[\pi_0(\huaG_0)/\pi_0(\huaG_1)\cong\pi_0(\huaH_0)/\pi_0(\huaH_1)\]
of coarse quotients.
\end{lemma}

\begin{proof}
By Def.~\ref{def:Morita} it suffices to consider the case of a single morphism $f_\bu\colon\huaG_\bu\to\huaH_\bu$ which is either a weak equivalence of derived $n$-groupoids or a levelwise weak equivalence.  By Cor.~\ref{cor:w-eq-comb}, if $f_\bu$ is a weak equivalence of derived $n$-groupoids, then $f_\bu$ induces locally split fibrations
\[
\bar{f}_0\colon\huaG_0\times_{\huaH_0}\huaH_1\longto\huaH_0,\qquad
\bar{f}_1\colon\huaG_1\longto(\huaG_0\times\huaG_0)
\times_{\huaH_0\times\huaH_0}\huaH_1.
\]
The fact that $\bar{f}_0$ is locally split implies that the continuous map
\[
\bar{f}\colon\pi_0(\huaG_0)/\pi_0(\huaG_1)\longto\pi_0(\huaH_0)/\pi_0(\huaH_1)
\]
induced by $f_\bu$ is topologically locally split; the fact that $\bar{f}_1$ is locally split implies that $\bar{f}$ is injective and hence a homeomorphism.  A levelwise weak equivalence $\huaG_\bu\to\huaH_\bu$ induces an isomorphism of topological $n$-groupoids $\pi_0(\huaG_\bu)\to\pi_0(\huaH_\bu)$ and hence a homeomorphism of coarse quotients.
\end{proof}

The coarse quotient, being a subquotient of a manifold, can also be regarded as a diffeological space (see e.g.\ \cite{barbieri-watts-ziegler;frobenius}), and the homeomorphism of Lemma~\ref{lemma;coarse} is in fact an isomorphism of diffeological spaces.

The \emph{stabilizer} (or \emph{isotropy} or \emph{inertia}) of a classical point $x\in\pi_0(\huaG_0)$ is the pullback
\[
\begin{tikzcd}
\Stab(\huaG_\bu,x)\ar[r]\ar[d]&\pi_0(\huaG_\bu^I)\ar[d]\\
*\ar[r,"{(x,x)}"]&\pi_0(\huaG_\bu)\times\pi_0(\huaG_\bu),
\end{tikzcd}
\]
which is an $n$-group object in the category of topological spaces.  For $n=1$ this is the usual stabilizer of $x$; for general $n$ we state without proof the following result.

\begin{lemma}\label{lemma;stabilizer}
Let $\huaG_\bu$ be a derived Lie $n$-groupoid.  For each $x\in\pi_0(\huaG_0)$ the stabilizer $\Stab(\huaG_\bu,x)$ is a Lie $n$-group.  Given a Morita equivalence $\huaG_\bu\simeq\huaH_\bu$ of derived Lie $n$-groupoids, if $[x]\in\pi_0(\huaG_0)/\pi_0(\huaG_1)$ corresponds to $[y]\in\pi_0(\huaH_0)/\pi_0(\huaH_1)$ under the homeomorphism $\pi_0(\huaG_0)/\pi_0(\huaG_1)\cong\pi_0(\huaH_0)/\pi_0(\huaH_1)$, then the stabilizers $\Stab(\huaG_\bu,x)$ and $\Stab(\huaH_\bu,y)$ have isomorphic homotopy groups.
\end{lemma}

\subsection{Homotopy pullback}\label{subsection;pull}

Here is a useful weakening of the notion of an acyclic fibration.

\begin{definition}\label{definition;strong}
A morphism of derived Lie $n$-groupoids $\Phi_\bu\colon\huaH_\bu\to\huaG_\bu$ is a \emph{strong fibration} if it is a fibration of derived Lie $n$-groupoids with the additional property that $\Phi_0\colon\huaH_0\to\huaG_0$ is a fibration of derived manifolds.
\end{definition}

Given morphisms of derived Lie $n$-groupoids $f_\bu\colon \huaG_\bu\to \huaK_\bu$ and $g_\bu\colon \huaH_\bu\to \huaK_\bu$, a homotopy pullback $(\huaG\times^h_\huaK \huaH)_\bu$ is defined as
\[
(\huaG\times^h_\huaK \huaH)_\bu:=\huaG_\bu\times_{\huaK_\bu} \widetilde{\huaH}_\bu, 
\]where $\widetilde{\huaH}_\bu $ is a fibrant replacement of $\huaH_\bu$. That is, there exists a Morita morphism $ \sigma_\bu\colon \huaH_\bu \to\widetilde{\huaH}_\bu$ and a strong fibration $\tilde{g}_\bu\colon \widetilde{\huaH}_\bu \to\huaK_\bu$. Such a fibrant replacement might not exist. It is not guaranteed by factorization lemma because of the fibration requirement on the level 0. However, this requirement guarantees the existence of the fiber product $\huaG_\bu\times_{\huaK_\bu} \widetilde{\huaH}_\bu$ because it was shown in \cite[Prop.~7.10]{Rogers-Zhu:2016} that if 
$\huaG_0 \times_{\huaK_0} \widetilde{\huaH}_0 $ exists in $\DMfd$, the pullback $\huaG_\bu\times_{\huaK_\bu} \widetilde{\huaH}_\bu$ is a derived Lie $n$-groupoid.  In this article, for most of the examples we find explicit fibrant replacements. 

\begin{remark} \label{rmk:hom-pb}
The homotopy pullback defined above is not the most general. One should allow $\widetilde{\huaH}_\bu$ being simply a Morita equivalent copy of $\huaH_\bu$, that is, one should allow $\sigma_\bu\colon \huaH_\bu \dashrightarrow\widetilde{\huaH}_\bu$being a Morita equivalence instead of a Morita morphism. Then $g_\bu$ and $\tilde{g}_\bu\circ\sigma_\bu$ should differ by a 2-morphism. Moreover, we might need to adapt the more genenral concept of Morita equivalence traveling through a bigger world than just $\dnGpd$ to demonstrate that the concept of homotopy pullback is independent of the choice of fibrant replacement. Nevertheless, for the examples in this article, the homotopy pullback defined above is very helpful to guide and illustrate the nature of many concrete constructions. 
\end{remark}

\section{Shifted symplectic structures on derived Lie \texorpdfstring{$n$}{n}-groupoids}\label{sec:symp}

In this chapter we extend the $m$-shifted symplectic structures on Lie $n$-groupoids introduced in \cite{Lesdiablerets} (see also \cite{cueca-zhu}) to the realm of derived Lie $n$-groupoids. A similar approach in the setting of DG Artin $n$-hypergroupoids was developed in \cite{prid:shift}.

\subsection{The simplicial graded de Rham complex}\label{sec:sgdR}

Recall that differential forms  on a simplicial manifold $M_\bu$  are elements of the Bott-Shulman-Stasheff double complex $(\Omega^\bu(M_\bu),\delta, d)$, see e.g. \cite{bss:derham,cueca-zhu}. Here we introduce its analogue for a simplicial derived manifold.

Given a  simplicial derived manifold $\mathcal{M}_\bu=(M_\bu,\mathcal{S}^\bu_{M_\bu},Q_\bu)$, its \emph{simplicial graded de Rham complex} $$\Omega^{\bu,\bu,\bu}(\mathcal{M}):=(\Omega^{\bu,\bu}(\mathcal{M}_\bu),\ d, \ \widetilde{\Lie}_{Q},\  \delta)$$ is defined as the triple complex of graded differential forms on $\mathcal{M}_\bu$. Here, $\Omega^{a,b,c}(\mathcal{M})$ denotes the set of differential $a$-forms of internal degree $b\le 0$ on the derived manifold $\mathcal{M}_c$ as introduced in \S \ref{subsection;cartan}; the map $\delta$ is the simplicial differential
\[
\delta\colon \Omega^{a,b,c}(\mathcal{M})\rightarrow \Omega^{a,b,c+1}(\mathcal{M})\quad\text{}\quad \delta=\sum_{i=0}^{c+1}(-1)^id_i^{\ast};
\]
the differential $\widetilde{\Lie}_{Q}$ is defined as  $$\widetilde{\Lie}_{Q_c}=(-1)^b\Lie_{Q_c}\colon \Omega^{a,b,c}(\mathcal{M})\rightarrow \Omega^{a,b+1,c}(\mathcal{M}),$$ where $\Lie_{Q_c}$ 
stands for the Lie derivative along the homological vector field $Q_c$ on the derived manifold $\huaM_c$, and $d$ is the de Rham differential  
$$d\colon \Omega^{a,b,c}(\mathcal{M})\rightarrow \Omega^{a+1,b,c}(\mathcal{M})$$
on the derived manifold $\huaM_c$. Relevant for this article is the associated differential
    $$D:=\delta+(-1)^c \widetilde{\Lie}_{Q}+(-1)^{b+c}d$$
on the total complex of the triple complex  $\Omega^{\bu,\bu,\bu}(\mathcal{M}).$

Observe that so far we only used the face maps of the simplicial derived manifold $\huaM_\bu$. Now we also consider degeneracy maps and define the \emph{normalized subcomplex} as
 \begin{equation} \label{eq:norm-Ome}
     \hat{\Omega}^{\bu,\bu,\bu}(\mathcal{M}):=\{\alpha \in \Omega^{\bu,\bu,\bu}(\mathcal{M})\,|\,s_i^{\ast}\alpha=0 \:\text{for all possible $i$}\}
 \end{equation}
consisting of differential forms vanishing on degeneracies.  (One can show that the normalized subcomplex is quasi-isomorphic to the triple complex by a strategy similar to \cite[Prop.~2.5]{WoZh16} combined with the dual version of Dold-Kan as in \cite[Cor.~8.4.3]{wei:hom}.) 
 
With all the above,  we say that  $\alpha_\bu$ is an \emph{$m$-shifted $k$-form} if it is of the form 
\[
   \alpha_\bu=\bigoplus_{a+b+c=m+k, \ a\geq k
   }
   \alpha_{a,b,c} \qquad\text{with}\qquad\alpha_{a,b,c}\in \hat{\Omega}^{a,b,c}(\mathcal{M}).
\]
 It follows immediately from our conventions that $b\leq 0$ and $c\ge 0$.  
If $\alpha_\bu$ is an $m$-shifted $k$-form, then $D\alpha_\bu$ is an $m+1$-shifted $k$-form.  We say that an $m$-shifted $k$-form $\alpha_\bu$ is \emph{closed} if $D\alpha_\bu=0$.

\subsection{The tangent complex of a derived Lie \texorpdfstring{$n$}{n}-groupoid}
\label{tangentcomplex}

Recall that \cite{Lesdiablerets} introduced the non-degeneracy condition for $m$-shifted sympelctic structures on Lie $n$-groupoids using the tangent complex, see also \cite{cueca-zhu}. Here we introduce a tangent complex for derived Lie $n$-groupoids. The tangent complex of a derived Lie $n$-groupoid $\huaG_\bu$ will be  defined only  at classical points  of $\huaG_0$. Since the classical locus $\pi_0(\huaG_0)$ is in general not a smooth manifold, see \S \ref{subsection:coarse}, we confine ourselves to a pointwise definition of the tangent complex. 

Given a derived Lie $n$-groupoid $\mathcal{G}_\bu$ and a classical point $x_0\in \pi_0(\mathcal{G}_0)$, equation \eqref{equation;morphism-loci} implies that $s_0\circ\dots\circ s_0(x_0)\in\pi_0(\huaG_l)$ for all $l\in\N$. Hence, we  consider the tangent complex of each derived manifold $(T^\bu_{s_0\circ\dots\circ s_0(x_0)}\mathcal{G}_l, \ell_{Q_l})$ as introduced in  \eqref{notation;tangent-complex}. For now on, we denote $s_0\circ\dots\circ s_0(x_0)$ by $x_0$.  Since taking the tangent complex of each derived manifold is functorial, see \eqref{equation;morphism-tangent}, 
$$(T^\bu_{x_0}\mathcal{G}_\bu,\ell_{Q_\bu},T_{x_0}d_i, T_{x_0}s_j)$$
forms a simplicial cochain complex.

For each fix $k\in\Z$, we apply the normalization functor of the Dold-Kan correspondence  to the simplicial vector spaces $(T_{x_0}^k\mathcal{G}_\bu, T_{x_0}^kd_i, T_{x_0}^ks_j)$ and obtain a cochain complex concentrated in non-positive degrees $(N^\bu(T_{x_0}^k\mathcal{G}),\partial)$, where
$$N^{-l}(T_{x_0}^k\mathcal{G}) :=\Big( \bigcap_{i=0}^{l-1}\:\text{ker}\,T_{x_0}^kd_i^l\Big)\subseteq T_{x_0}^k\huaG_l
\quad\text{and}\quad \partial_{-l}:=(-1)^l d_l,$$ for $l\in\Z$.   Since $N^\bu$ is a functor, if we free $k\in\Z$,
we obtain the following double cochain complex 
\begin{equation}\label{diag:double-tang}
\begin{tikzcd}
	& \vdots & \vdots & \vdots \\
	\cdots & {N^{-2}(T_{x_0}^2\mathcal{G}_2)} & {N^{-1}(T_{x_0}^2\mathcal{G}_1)} & {N^0(T_{x_0}^2\mathcal{G}_0)} \\
	\cdots & {N^{-2}(T_{x_0}^1\mathcal{G}_2)} & {N^{-1}(T_{x_0}^1\mathcal{G}_1)} & {N^0(T_{x_0}^1\mathcal{G}_0)} \\
	\cdots & {N^{-2}(T_{x_0}^0\mathcal{G}_2)} & {N^{-1}(T_{x_0}^0\mathcal{G}_1)} & {N^0(T_{x_0}^0\mathcal{G}_0).}
	\arrow["\partial", from=2-1, to=2-2]
	\arrow["{N(\ell_{Q_2})}", from=2-2, to=1-2]
	\arrow["\partial", from=2-2, to=2-3]
	\arrow["{N(\ell_{Q_1})}", from=2-3, to=1-3]
	\arrow["\partial", from=2-3, to=2-4]
	\arrow["{N(\ell_{Q_0})}"', from=2-4, to=1-4]
	\arrow["\partial", from=3-1, to=3-2]
	\arrow["{N(\ell_{Q_2})}", from=3-2, to=2-2]
	\arrow["\partial", from=3-2, to=3-3]
	\arrow["{N(\ell_{Q_1})}", from=3-3, to=2-3]
	\arrow["\partial", from=3-3, to=3-4]
	\arrow["{N(\ell_{Q_0})}"', from=3-4, to=2-4]
	\arrow["\partial", from=4-1, to=4-2]
	\arrow["{N(\ell_{Q_2})}", from=4-2, to=3-2]
	\arrow["\partial", from=4-2, to=4-3]
	\arrow["{N(\ell_{Q_1})}", from=4-3, to=3-3]
	\arrow["\partial", from=4-3, to=4-4]
	\arrow["{N(\ell_{Q_0})}"', from=4-4, to=3-4]
\end{tikzcd}    
\end{equation}

\begin{definition}
\label{tangentcomplexLiengroupoid}
   Let $\huaG_\bu$ be a derived Lie $n$-groupoid and $x_0\in\pi_0(\huaG_0)$ a classical point. The total complex of the double cochain complex $(N^\bu(T_{x_0}^\bu\mathcal{G}),\partial,N(\ell_Q))$, as in  \eqref{diag:double-tang},
    is called the \emph{tangent complex of $\mathcal{G}_\bu$ at $x_0$} and will be denoted by $\mathbb{T}_{x_0}^\bu\mathcal{G}$. More concretely,
    \begin{equation} \label{eq:T-total-G}
        \bT_{x_0}^i\huaG:=\bigoplus_{j+k=i}N^{j}(T^k_{x_0}\huaG)\quad\text{with differential}\quad
    \mathbb{D}=\partial+(-1)^jN(\ell_{Q_j}).
    \end{equation}
\end{definition}    

\subsection{IM-pairings and main definition}\label{subsection;IM}

In order to define shifted symplectic structures, we need a non-degeneracy condition. For $m$-shifted $2$-forms on Lie $n$-groupoids that was done in \cite{Lesdiablerets} by defining an \emph{infinitesimally multiplicative pairing}, IM-pairing for short. His construction could be understood in terms of the van Est map, see e.g. \cite{BuCa12, Stefiduals}.
Here we introduce the IM-pairing of an $m$-shifted $2$-form on a derived Lie $n$-groupoid. 

\subsubsection{An adjusted Eilenberg-Zilber map for simplicial cochain complexes}\label{EZ}

To define the IM-pairing induced by an $m$-shifted $2$-form on a derived Lie $n$-groupoid, we first need an adjusted Eilenberg-Zilber map for simplicial cochain complexes. 

Totally parallel to  \S \ref{tangentcomplex}, for a given simplicial cochain complex $(\mathcal{V}^\bu_\bu,D, d_i, s_j)$, the normalized functor produces a double cochain complex $(N^\bu(\mathcal{V}^\bu),\partial,N(D))$. We similarly define its total complex and denote it by $(\tot^\bu(N(\huaV)),\mathbb{D})$. In this notation, the tangent complex as in \eqref{eq:T-total-G} of a derived Lie $n$-groupoid $\huaG_\bu$ at a classical point $x_0$ reads as $\bT^\bu_{x_0}\huaG=\tot^\bu(N(T_{x_0}\huaG))$.

We first recall that the Dold-Kan correspondence holds in general for simplicial objects
in abelian categories \cite[\S\,8.4]{wei:hom}. The Eilenberg-Zilber theorem describes
how the Dold-Kan correspondence acts on the tensor products on each side. The usual Eilenberg-Zilber map \cite{eilenberg-zilber:53} (see also \cite[\S\,29]{may}) for simplicial abelian groups also works for simplicial cochain complexes.  If $(\mathcal{V}^\bu_\bu,D^{\huaV}, d_i^{\huaV}, s_j^{\huaV})$ and $(\mathcal{W}^\bu_\bu,D^{\huaW}, d_i^{\huaW}, s_j^{\huaW})$ are simplicial cochain complexes,  the \emph{Eilenberg-Zilber map}  $EZ\colon N^\bu(\mathcal{V}^\bu)\otimes N^\bu(\mathcal{W}^\bu)\rightarrow N^\bu((\mathcal{V}\otimes \mathcal{W})^\bu)$, given by 
\begin{equation}\label{eq:EZ}
    EZ(v\otimes w)=\sum_{\sigma\in \Sh(p,q)}\sgn(\sigma)\bigl(s_{\sigma(p+q-1)}\cdots s_{\sigma(p)}(v)\bigr)\otimes\bigl(s_{\sigma(p-1)}\cdots s_{\sigma(0)}(w)\bigr)
\end{equation}
for $v \in N^{-p}(\huaV)$ and $w\in N^{-q}(\huaW)$, is a map of double cochain complexes. 

\begin{remark}
    On the domain of the Eilenberg-Zilber map, the tensor product is that of cochain complexes in the category of cochain complexes, i.e. 
    \begin{equation}\label{eq:tensor-normal}
        \big(N(\mathcal{V})\otimes N(\mathcal{W})\big)^{i,\bu}=\bigoplus_{j+k=i} \big(N^j(\huaV)\otimes N^k(\huaW)\big)^\bu\quad\text{with}\quad \partial^{\otimes}=\partial^{\huaV}\otimes 1+(-1)^j1\otimes\partial^{\huaW},
    \end{equation}
    while on the codomain of the Eilenberg-Zilber map we have the simplicial tensor product in the category of cochain complexes, i.e. at simplicial degree $k$ we have
    $$\big((\huaV\otimes\huaW)_k^\bu, D^{\huaV\otimes\huaW}, d_i^{\huaV\otimes\huaW},s_i^{\huaV\otimes\huaW}\big)=\big((\huaV_k\otimes\huaW_k)^\bu, D^{\huaV_k\otimes\huaW_k}, d_i^{\huaV}\otimes d_i^{\huaW}, s_i^{\huaV}\otimes s_i^{\huaW}\big).$$
\end{remark}

The Eilenberg-Zilber map can be slightly adjusted, so that one obtains a cochain map between $\tot^\bu(N(\huaV))\otimes\tot^\bu(N(\huaW))$ and $\tot^\bu(N(\huaV\otimes\huaW))$.
For cochains $v^s_p\in N^{-p}(\mathcal{V}^s)$ and $w^t_q\in N^{-q}(\mathcal{W}^t)$ 
the \emph{adjusted Eilenberg-Zilber map} $\widetilde{EZ}$ is defined as $$\widetilde{EZ}(v^s_p\otimes w^t_q)=(-1)^{sq}EZ(v^s_p\otimes w^t_q).$$
\begin{proposition}\label{prop:aEZ}
    Given simplicial cochain complexes $(\mathcal{V}^\bu_\bu,D^{\huaV}, d_i^{\huaV}, s_j^{\huaV})$ and $(\mathcal{W}^\bu_\bu,D^{\huaW}, d_i^{\huaW}, s_j^{\huaW})$, the adjusted Eilenberg-Zilber map $$\widetilde{EZ}\colon \big(\tot^\bu(N(\huaV))\otimes\tot^\bu(N(\huaW)),\mathbb{D}^{\otimes}\big)\rightarrow \big(\tot^\bu(N(\huaV\otimes\huaW)),\mathbb{D}^{\huaV\otimes\huaW}\big)$$ is a cochain map. Here, $\mathbb{D}^\otimes$ is the differential of tensored cochain complexes defined similar to \eqref{eq:tensor-normal} and $\mathbb{D}^{\huaV \otimes \huaW}$ is the total differential defined similar to \eqref{eq:T-total-G}. 
\end{proposition}
\begin{proof}
Given cochains $x\in N^p(\mathcal{V}^s)$ and $y\in N^q(\mathcal{W}^t)$, we set $n=p+q$ and compute
\begin{align*}
\widetilde{EZ}(\mathbb{D}^{\otimes}(x\otimes y))&=\widetilde{EZ}\bigl(\mathbb{D}^{\huaV}(x)\otimes y+(-1)^{p+s}x\otimes \mathbb{D}^{\huaW}(y)\bigr)\\
&=\widetilde{EZ}\Bigl(\bigl(\partial^{\huaV}+(-1)^pN(D^{\huaV})\bigr) (x)\otimes y+(-1)^{s+p} x\otimes \bigl(\partial^{\huaW}+(-1)^{q} N(D^{\huaW})\bigr) (y)\Bigr)\\
&=(-1)^{sq} EZ(\partial^{\huaV} (x)\otimes y)+(-1)^{s(q+1)+s+p}EZ( x\otimes \partial^{\huaW} (y))\\
&\qquad+(-1)^{p+(s+1)q}EZ(N(D^{\huaV})(x)\otimes y)+(-1)^{p+s+q+sq}EZ(x\otimes N(D^{\huaW})(y))\\
&=(-1)^{sq}EZ(\partial^\huaV (x)\otimes y+(-1)^p x\otimes \partial^\huaW (y))\\
&\qquad+(-1)^{sq+n}EZ(N(D^{\huaV})(x)\otimes y+(-1)^{s}x\otimes N(D^{\huaW})(y))\\
&=(-1)^{sq}\partial^{\huaV\otimes\huaW} EZ(x\otimes y)+(-1)^{sq+n}N(D^{\huaV\otimes\huaW})EZ(x\otimes y)\\
&=\big(\partial^{\huaV\otimes\huaW} +(-1)^nN(D^{\huaV\otimes\huaW})\big)(-1)^{sq}EZ(x\otimes y)=\mathbb{D}^{\huaV\otimes\huaW}\widetilde{EZ}(x\otimes y).\qedhere
\end{align*}
\end{proof}

\subsubsection{IM-pairings}

Let $\huaG_\bu$ be a derived Lie $n$-groupoid and $\omega_\bu$ an $m$-shifted $2$-form. For each point $x_0\in \pi_0(\mathcal{G}_0)$ the evaluation of $\omega_\bu$ gives a morphism
\begin{equation} \label{eq:f-omega}
    \begin{split}
        f_{x_0}^{\omega_\bu}\colon  &\tot^\bu\big(N(T_{x_0}\mathcal{G}\otimes T_{x_0}\mathcal{G})\big)\rightarrow \mathbb{R}[m],\\ 
       & u^s_{n}\otimes v^t_{n}\mapsto
\begin{cases}
        {(\omega_{2,-(s+t),n})}_{x_0}(u^s_n,v^t_n)&\text{if $s+t-n=-m$}\\
        0&\text{else.}
    \end{cases}
    \end{split}
\end{equation} 
In general, $f^{\omega_\bu}_{x_0}$ is only a cochain morphism if $\omega_\bu$ is closed. However, as proven in Cor.~\ref{bluppblupp}, the map 
\begin{equation}\label{eq:f-omega-Domega}
    \begin{split}
        f_{x_0}^{\omega_\bu, D\omega_\bu}\colon &\tot^\bu\big(N(T_{x_0}\mathcal{G}\otimes T_{x_0}\mathcal{G})\big)\rightarrow \mathbb{R}[m,m+1],\\ 
        & u^s_{n}\otimes v^t_{n}\mapsto
\begin{cases}
        {(\omega_{2,-(s+t),n})}_{x_0}(u^s_n,v^t_n)&\text{if $s+t-n=-m$}\\
        {((D\omega)_{2,-(s+t),n})}_{x_0}(u^s_n,v^t_n)&\text{if $s+t-n=-(m+1)$}\\
        0&\text{else.}
    \end{cases}
    \end{split}
\end{equation} 
is always a cochain morphism, where $\mathbb{R}[m,m+1]$ denotes the cochain complex given by $\mathbb{R}$ concentrated in degrees $-(m+1)$ and $-m$ whose differential in degree $-(m+1)$ is $id_{\mathbb{R}}\colon \mathbb{R}\rightarrow \mathbb{R}$.

\begin{definition}
    Let $\huaG_\bu$ be a derived Lie $n$-groupoid and $x_0\in\pi_0(\huaG_0)$. The \emph{IM-pairing $\lambda^{\omega_{\bu}}_{x_0}$ induced by an $m$-shifted $2$-form $\omega_\bu$ at $x_0$}  is the composition 
\begin{gather*}
\lambda^{\omega_\bu}_{x_0}\colon    \bT_{x_0}^\bu\huaG \otimes \bT_{x_0}^\bu\huaG \xrightarrow[]{\widetilde{EZ}} \tot^\bu\big(N(T_{x_0}\mathcal{G}\otimes T_{x_0}\mathcal{G})\big) \xrightarrow[]{f_{x_0}^{\omega_\bu}} \mathbb{R}[m],  
\end{gather*}
where $\widetilde{EZ}$ is the adjusted Eilenberg-Zilber morphism introduced in \S \ref{EZ} and we used the identification $\bT_{x_0}^\bu\huaG=\tot^\bu\big(N(T_{x_0}\mathcal{G})\big)$. More explicitly, fix  $l,\Tilde{l}\in \mathbb{Z}$ such that $l+\Tilde{l}=m$ and choose tangent vectors $u\in\bT^{-l}_{x_0}\huaG$ and $v\in\bT^{-\tilde{l}}_{x_0}\huaG$. Writing 
$$u=\oplus_s u^s\quad \text{with}\quad u^s\in N^{-l-s}(T_{x_0}^s\huaG)
\quad \text{and}\quad v=\oplus_t v^t\quad \text{with}\quad v^t\in N^{-\tilde{l}-t}(T_{x_0}^t\huaG),$$
the IM-pairing $\lambda_{x_0}^{\omega_\bu}(u^s,v^t)$ of the components $u^s$ and $v^t$ is given by
\begin{gather*}
    \sum_{\sigma\in \Sh(l+s,\tilde{l}+t)}(-1)^{s(\tilde{l}+t)}\sgn(\sigma)({\omega_{2,-r,m+r}})_{x_0}\big(T(s_{\sigma(m-r-1)}\dots s_{\sigma(l+s)})u^s, T(s_{\sigma(l+s-1)}\dots s_{\sigma(0)})v^t\big),
\end{gather*}
with $r=(s+t)$.
\end{definition}

The IM-pairing $\lambda^{\omega_\bu}$ is easily seen to be graded anticommutative. In order to see other properties, first notice that for a closed $m$-shifted 2-form $\omega_\bu$ on a derived Lie $n$-groupoid $\mathcal{G}_\bu$, the $2$-form part satisfies
\[
0=D\Bigl(\sum_{r\ge 0}\omega_{(2,-r,m+r)}\Bigr)=\sum_{r\ge 0}(\delta+(-1)^{m}\Lie_Q)\omega_{(2,-r,m+r)}.
\]
Thus we get the following equation
\begin{equation}\label{eq:2fpart}
\delta\omega_{(2,-r,m+r)}=(-1)^{m+1}(\Lie_Q\omega)_{(2,-(r-1),m+(r-1))}.
\end{equation}
Secondly we can relate the Lie derivative along $Q$ and the differential for the tangent complex in the following way.

\begin{lemma}\label{pairingproperties}
Let $\huaM$ be a derived manifold, $x_0\in\pi_0(\huaM)$ a classical point, and 
$\omega\in \Omega^{2,i}(\mathcal{M})$. Then 
\[(\Lie_Q\omega)_{x_0}(u,v)= (-1)^{|u|+|v|+1}\omega_{x_0}(\ell_Qu,v)+(-1)^{|v|+1}\omega_{x_0}(u,\ell_Qv)\]
for all $u,v\in T^\bu_{x_0}\huaM$ with $|u|+|v|=-(i+1)$. 
\end{lemma}

\begin{proof}
Choose vector fields $X$, $Y\in \mathfrak{X}^1(\mathcal{M})$ such that $X_{x_0}=u$ and $Y_{x_0}=v$.  Then $(\Lie_Q\omega)_{x_0}(u,v)=((\Lie_Q\omega)(X,Y))_{x_0}$. Since
    \begin{gather*}
      (\Lie_Q\omega)(X,Y)=\Lie_Q(\omega(X,Y))-(-1)^{i}\omega([Q,X],Y)-(-1)^{i+|X|}\omega(X,[Q,Y]),
    \end{gather*}
    it follows that
    \begin{align*}
        (\Lie_Q\omega)_{x_0}(v,w)&=Q_{x_0}(\omega(X,Y))-(-1)^{i}\omega_{x_0}([Q,X]_{x_0},Y_{x_0})-(-1)^{i+|X|}\omega_{x_0}(X_{x_0},[Q,Y]_{x_0})\\
        &=(-1)^{|u|+|v|+1}\omega_{x_0}(\ell_Qu,v)+(-1)^{|v|+1}\omega_{x_0}(u,\ell_Qv),
    \end{align*}
where in the last equality we used the fact that $Q_{x_0}=0$ (since $x_0$ is a classical point), $[Q,X]_{x_0}=\ell_Qv$, and $[Q,Y]_{x_0}=\ell_Qw$ (since $\ell_Q$ is the linearization of $\ad_Q$).
\end{proof}

\begin{corollary}
\label{bluppblupp}
    Let $\omega_\bu$ be an $m$-shifted 2-form on a derived Lie $n$-groupoid $\mathcal{G}_\bu$. The map \eqref{eq:f-omega-Domega}
    \begin{gather*}
        f_{x_0}^{\omega_\bu, D\omega_\bu}\colon \tot^\bu\big(N(T_{x_0}\mathcal{G}\otimes T_{x_0}\mathcal{G})\big)\rightarrow \mathbb{R}[m,m+1]
    \end{gather*}
is a cochain morphism. If $\omega_\bu$ is closed, the map \eqref{eq:f-omega}
     \begin{gather*}
         f_{x_0}^{\omega_\bu}\colon \tot^\bu\big(N(T_{x_0}\mathcal{G}\otimes T_{x_0}\mathcal{G})\big)\rightarrow \mathbb{R}[m]
     \end{gather*}
     is a cochain morphism as well.
\end{corollary}
 \begin{proof}
 It is enough to verify that $f^{\omega_\bu,D\omega_\bu}$ commutes with the differentials in degree $-(m+1)$, i.e. we want to show that  for all $x_0\in \pi_0(\mathcal{G}_0)$ and $u^s\otimes v^t\in  N^{-l}(T^s_{x_0} \huaG \otimes T^t_{x_0} \huaG) \subset \tot^{s+t-l}\big(N(T_{x_0}\mathcal{G}\otimes T_{x_0}\mathcal{G})\big)$ with   $s+t-l=-(m+1)$ then
 \begin{gather*}
    (D\omega)_{x_0}(u^s\otimes v^t)=\omega_{x_0}(\mathbb{D}^{T_{x_0}\huaG\otimes T_{x_0}\huaG}(u^s\otimes v^t)).
 \end{gather*}
Denote by $r=s+t$,  from Lemma~\ref{pairingproperties} it follows that 
\begin{equation*}
\begin{split}
(D\omega)_{x_0}(u^s\otimes v^t)&=(\delta\omega_{2,-r,l-1})_{x_0}(u^s\otimes v^t)+(-1)^{l+r}(\mathcal{L}_Q\omega_{2,-r-1,l})_{x_0}(u^s\otimes v^t)\\
&=\omega_{2,-r,l-1}(\partial u^s\otimes \partial v^t)+(-1)^{l+r+1}(-1)^{r+1}(\omega_{2,-r-1,l})_{x_0}(\ell_Qu^s,v^t)\\
&\qquad+(-1)^{l+r+1}(-1)^{t+1}(\omega_{2,-r-1,l})_{x_0}(u^s,\ell_Qv^t)\\
&=\omega_{x_0}\bigl((\partial^{T_{x_0}\huaG\otimes T_{x_0}\huaG}+(-1)^lN(\ell_Q\otimes 1+1\otimes \ell_Q))(u^s\otimes v^t)\bigr)\\
&=\omega_{x_0}(\mathbb{D}^{T_{x_0}\huaG\otimes T_{x_0}\huaG}(u^s\otimes v^t)),
\end{split}
\end{equation*}
so $f^{\omega_\bu,D\omega_\bu}$ is indeed a cochain morphism. The second statement can be proven similarly.
 \end{proof}

With this in mind, we are able to state the following essential property of the IM-pairing.

\begin{proposition}\label{proposition;IM}
    Let $\omega_\bu$ be an $m$-shifted 2-form on a derived Lie $n$-groupoid $\mathcal{G}_\bu$ and $x_0\in \pi_0(\huaG_0$) be a classical point.
    For $u\in \mathbb{T}_{x_0}^{-(l+1)}\mathcal{G}$, $v\in \mathbb{T}_{x_0}^{-\Tilde{l}}\mathcal{G}$ with $l+\Tilde{l}=m$ one has
        \begin{equation}\label{eq:IMchain}
             \lambda_{x_0}^{\omega_\bu}(\mathbb{D} u,v)+(-1)^{l+1}\lambda_{x_0}^{\omega_\bu}(u,\mathbb{D} v)=\lambda_{x_0}^{D\omega_\bu}(u,v).
        \end{equation}
\end{proposition}

\begin{proof}
It is enough to consider the case $u\in N^{-(l+s+1)}(T^s_{x_0}\huaG)$ and $v\in N^{-(\tilde{l}+t)}(T^t_{x_0}\huaG)$. By the previous propositions we get
\begin{align*}
\lambda^{\omega_\bu}_{x_0}(\mathbb{D} u,v)+(-1)^{l+1}\lambda_{x_0}^{\omega_\bu}(u,\mathbb{D} v)&=f_{x_0}^{\omega_\bu}\widetilde{EZ}(\mathbb{D}^{\bT_{x_0}\huaG\otimes\bT_{x_0}\huaG}(u\otimes v))\\ 
&=f^{\omega_\bu}_{x_0}\mathbb{D}^{T_{x_0}\huaG\otimes T_{x_0}\huaG}\widetilde{EZ}(u\otimes v)\\
&=p_mf_{x_0}^{\omega_\bu,D\omega_\bu}\mathbb{D}^{T_{x_0}\huaG\otimes T_{x_0}\huaG}\widetilde{EZ}(u\otimes v)\\
&=p_m\mathbb{D}^{\mathbb{R}[m,m+1]}f_{x_0}^{\omega_\bu,D\omega_\bu}\widetilde{EZ}(u\otimes v)\\
&=f^{D\omega_\bu}_{x_0}\widetilde{EZ}(u\otimes v)=\lambda^{D\omega_\bu}(u\otimes v),
\end{align*}
where the first equality is the definition of the differential on the tensor product, the second equality is Prop.~\ref{prop:aEZ}, the fourth equality is Cor.~\ref{bluppblupp}, and $p_m$ denotes the obvious projection $p_m\colon \mathbb{R}[m,m+1]\rightarrow \mathbb{R}[m]$.
\end{proof}

Due to the above properties, we expect that the IM-pairing can be also obtained via a van Est map for derived Lie $n$-groupoids. A detailed discussion of this point of view for Lie $n$-groupoids can be found in \cite{Dorsch:thes}. 

\subsubsection{Main definition}

Let $\huaG_\bu$ be a derived Lie $n$-groupoid and pick a classical point $x_0\in\pi_0(\huaG)$. The \emph{$m$-shifted cotangent complex of $\huaG_\bu$ at $x_0$}   is the cochain complex $\bT^{*, \bu}_{x_0}\huaG[m]=\Hom(\bT_{x_0}\huaG, \R[m])$, more concretely
\begin{equation}\label{eq:co-tan}
    \bT^{*, k}_{x_0}\huaG[m] :=(\bT^{-k-m}_{x_0}\huaG)^*\quad\text{and}\quad  \mathbb{D}^*[m]^k=(-1)^{k+1}\mathbb{D}_{-k-1-m}^t\colon  \bT^{*, k}_{x_0}\huaG[m] \to \bT^{*, k+1}_{x_0}\huaG[m] .
\end{equation}
When the shift and degree is clear from context, the differential on the cotangent complex is denoted by $\mathbb{D}^*$ for simplicity.

With the above definition, given an $m$-shifted $2$-form $\omega_\bu$ on $\huaG_\bu$ and $x_0\in\pi_0(\huaG_0)$,  its IM-pairing defines a map by 
\begin{gather}\label{non-deg}
  \lambda_{x_0}^{\omega_\bu,\#}\colon \mathbb{T}^\bu_{x_0}\mathcal{G}\rightarrow \mathbb{T}_{x_0}^{\ast,\bu}\mathcal{G}[m],\qquad v\mapsto \lambda^{\omega_\bu,\#}(v)=\lambda^{\omega_\bu}(v,\cdot).
\end{gather}
Thus by Prop.~\ref{proposition;IM}, we immediately have
\begin{equation}\label{eq:lambda-sharp}
    \lambda^{\omega_\bu, \#} \mathbb D - \mathbb D^* \lambda^{\omega_\bu, \#} = \lambda^{D\omega_\bu, \#},
\end{equation} which holds for all $m$.

\begin{corollary}\label{gaugeinvar}
Let $\huaG_\bu$ be a derived Lie $n$-groupoid.  Let $\omega_\bu$ be a closed $m$-shifted $2$-form and $\eta_\bu$ an $(m-1)$-shifted $2$-form on $\huaG_\bu$.  For every classical point $x_0\in \pi_0(\huaG_0)$ the following statements hold:
\begin{enumerate}
\item\label{item;cochain} the IM-pairings 
\[
\lambda_{x_0}^{\omega_\bu,\#}, \lambda_{x_0}^{\omega_\bu+D\eta_\bu,\#}\colon \mathbb{T}^\bu_{x_0}\mathcal{G}\rightarrow \mathbb{T}_{x_0}^{\ast, \bu}\mathcal{G}[m]
\]
are cochain maps;
\item\label{item;homotopy} the IM-pairing $\lambda_{x_0}^{\eta_\bu,\#}\colon \mathbb{T}^\bu_{x_0}\mathcal{G}\rightarrow \mathbb{T}_{x_0}^{\ast, \bu}\mathcal{G}[m-1]$ provides a homotopy between the above cochain maps.
\end{enumerate}
\end{corollary}

\begin{proof}
Item \eqref{item;cochain} follows from \eqref{eq:lambda-sharp}.  We deduce item \eqref{item;homotopy} from Prop.~\ref{proposition;IM} as follows: given $u\in\bT^l_{x_0}\huaG$ and $v\in\bT^{-m-l-1}_{x_0}\huaG$, we have
\begin{align*}
\big(\lambda_{x_0}^{\omega_\bu+D\eta_\bu,\#}(u)-\lambda_{x_0}^{\omega_\bu,\#}(u)\big)(v)&=\big(\lambda_{x_0}^{D\eta_\bu,\#}(u)\big)(v)=\lambda^{D\eta_\bu}_{x_0}(u,v)=\lambda^{\eta_\bu}_{x_0}(\mathbb{D}u,v)+(-1)^{l}\lambda^{\eta_\bu}_{x_0}(u,\mathbb{D}v)\\
&=\big(\lambda_{x_0}^{\eta_\bu,\#}(\mathbb{D}u)-\mathbb{D}^*\lambda_{x_0}^{\eta_\bu,\#}(u)\big)(v).\qedhere
\end{align*}
\end{proof}

\begin{definition}\label{definition;shifted-symplectic}
Let $\huaG_\bu$ be a derived Lie $n$-groupoid.  An \emph{$m$-shifted symplectic form on $\huaG_\bu$} is a closed $m$-shifted $2$-form $\omega_\bu$ which is \emph{cohomologically non-degenerate} in the sense that the cochain map $\lambda_{x_0}^{\omega_\bu,\#}$ is a quasi-isomorphism for all $x_0\in\pi_0(\huaG_0)$.  If $\omega_\bu$ is an $m$-shifted symplectic form on $\huaG_\bu$, we refer to  the pair $(\huaG_\bu,\omega_\bu)$ as an \emph{$m$-shifted symplectic derived Lie $n$-groupoid}.
\end{definition}

\begin{remark}\label{ex:sympnond}
Let $G_\bu$ be an ordinary Lie $n$-groupoid regarded as a derived Lie $n$-groupoid as in Ex.~\ref{ex:lie=dlie}.  Then the triple de Rham complex is just the usual Bott-Shulman-Stasheff double complex $(\Omega^\bu(G_\bu),\delta,d)$ of the simplicial manifold $G_\bu$.  The tangent complex at $x_0\in\pi_0(G_0)=G_0$ is
\[\bT^{-i}_{x_0}G=(\huaT_{i}G)_{|x_0},\]
where $\huaT_\bu G$ is the tangent complex of a Lie $n$-groupoid as introduced in \cite[Def 2.8]{cueca-zhu}; observe the shift between homological and cohomological conventions.  It follows that an $m$-shifted symplectic structure on $G_\bu$ viewed as a derived Lie $n$-groupoid is the same thing as an $m$-shifted symplectic structure on $G_\bu$ viewed as a Lie $n$-groupoid.  (For the latter see \cite{Lesdiablerets} and also \cite[\S 2.2]{cueca-zhu}.
\end{remark}

\begin{example}\label{ex:-1cot}
Let $M$ be a manifold and $f\colon M\to\R$ a smooth function.  The \emph{derived critical locus} of $f$ is the derived manifold
\[(T^*[-1]M,\iota_{df})=\huaZ(M, T^*M, df).\]
The canonical symplectic form $\omega_\can\in\Omega^{2,-1}(T^*[-1]M)$ is a  $-1$-shifted symplectic structure on the derived critical locus; see e.g.\ \cite{vezzosi;critical}.
\end{example}

\begin{example}\label{ex:0cot}
By combining the previous example with Prop.~\ref{ex:VBgrpd}, we can produce a $0$-shifted symplectic derived Lie $1$-groupoid as follows.  Let $G\rightrightarrows M$ be a Lie groupoid with Lie algebroid $A\to M$.  The \emph{cotangent VB-groupoid}
\[
\begin{tikzcd}[row sep=large]
T^*G\ar[d]\ar[r,shift left]\ar[r,shift right]&A^*\ar[d]\\
G\ar[r,shift left]\ar[r,shift right]&M,
\end{tikzcd}
\]
is the VB-groupoid dual to the tangent bundle, see e.g.\ \cite[\S 11.2]{MK2}.  Let $0_G$ be the zero section of the cotangent VB-groupoid.  From Prop.~\ref{ex:VBgrpd} we obtain a quasi-smooth derived Lie $1$-groupoid
\begin{equation}\label{eq:def0cot}
\huaC ot_\bu(G):=\huaZ(T^*G\rightrightarrows A^*,0_G)=\bigl(\huaZ(G,T^*G,0_G)\rightrightarrows \huaZ(M,A^*, 0_M)\bigr).        
\end{equation}
We can view the space of arrows $\huaC ot_1(G)=\huaZ(G,T^*G,0_G)$ as the derived critical locus of the zero function $G\to\R$, and Ex.~\ref{ex:-1cot} shows that the canonical symplectic form on $T^*G$ gives rise to an element
\[\omega_\can\in\Omega^{2,-1}(\huaZ(G,T^*G,0_G))=\Omega^{2,-1,1}(\huaC ot_\bu(G)).\]
Since $T^*G\rightrightarrows A^*$ is a symplectic groupoid with the canonical symplectic structure, we get that $D\omega_\can=0$.  The classical locus of $\huaC ot_0(G)$ is $\pi_0(\huaZ(M, A^*, 0_M))=M$. For any $x_0\in M$ the tangent complex is concentrated in degree $-1$ to $1$ and given by 
$$\bT_{x_0}^\bu\huaH\colon\quad A_{x_0}\xrightarrow{\bigl(\begin{smallmatrix}
    \rho_{x_0}\\ 0\end{smallmatrix}\bigr)} T_{x_0}M\oplus T_{x_0}^*M\xrightarrow{(\begin{smallmatrix}
        0& \rho_{x_0}^*
    \end{smallmatrix})} A_{x_0}^*.$$ 
An easy computation shows that $\lambda^{\omega_\can,\sharp}$ is the identity map and therefore $(\huaC ot_\bu(G),\omega_\can)$ is a $0$-shifted symplectic derived Lie $1$-groupoid.
\end{example}

\subsection{Symplectic Morita equivalence}\label{sec:symp-ME}

A \emph{strict symplectic equivalence} between two $m$-shifted symplectic derived Lie $n$-groupoids $(\mathcal{G}_\bu,\alpha_\bu)$ and $(\mathcal{H}_\bu,\beta\bu)$ is defined as a Morita morphism $\Phi_\bu\colon \huaG_\bu\to\huaH_\bu$  together with an $(m-1)$-shifted $2$-form $\eta_\bu$ in $\huaG_\bu$ such that     $$\alpha_\bu-\Phi^{\ast}\beta_\bu=D\eta_\bu.$$
{\em Symplectic Morita equivalence} is the equivalence relation generated by strict symplectic equivalence. That is, two $m$-shifted symplectic derived Lie $n$-groupoids $(\mathcal{G}_\bu,\alpha_\bu)$ and $(\mathcal{H}_\bu,\beta\bu)$ are \emph{symplectic Morita equivalent} if they are connected by a sequence of strict symplectic equivalences.
\begin{remark}\label{rmk:symp-ME}
   The symplectic Morita equivalence for $m$-shifted Lie $n$-groupoids in \cite{cueca-zhu} is given by a span of hypercovers. Notice that a hypercover is a weak equivalence\footnote{We need to weaken hypercover to weak equivalence in this article to make the embedding into the path object, e.g. $BG \xrightarrow[]{\sigma} BG^I$ into a Morita morphism. Otherwise it would be a span of hypercovers only. }, span of weak equivalence can be replaced by a span of hypercovers \cite[\S\,2]{Rogers-Zhu:2016} and a weak equivalence of manifolds viewed as derived manifolds is simply an isomorphism. Thus our Symplectic Morita equivalence defined above restricting to the world of $m$-shifted Lie $n$-groupoids coincides with the symplectic Morita equivalence defined therein. 
\end{remark}

In this work we will not show that $m$-shifted symplectic structures can be transported under Morita morphisms.  Nevertheless, we work towards it.  We begin with the invariance of tangent complexes under Morita morphisms.

\begin{lemma}\label{lem:MEtcom}
Let $\Phi_\bu\colon \mathcal{G}_\bu\to\mathcal{H}_\bu$ be a Morita morphism between derived Lie $n$-groupoids and $x\in \pi_0(\mathcal{G}_0)$. Then, the induced morphism of tangent complexes 
\begin{gather*}
    \mathbb{T}_{x}^\bu\Phi\colon \mathbb{T}^\bu_x\mathcal{G}\rightarrow \mathbb{T}^\bu_{\Phi_0(x)}\mathcal{H}
\end{gather*}
is a quasi-isomorphism if there are only finitely many non-zero $N^{p}(T^q_{x}(\huaG))$ and  $N^{p}(T^q_{\Phi_0(x)}(\huaH))$ with $p+q=l$ for all $l\in \Z$. In particular, if $n$ is finite, this condition is automatically satisfied since our derived manifolds have finite amplitudes. 
\end{lemma}
\begin{proof}
    First assume that $\Phi_\bu\colon \mathcal{G}_\bu\rightarrow \mathcal{H}_\bu$ is a weak equivalence of higher derived Lie $n$-groupoids.  Using the forgetful functor from derived manifolds to nonpositively graded manifolds, we view $\huaG_\bu$ and $\huaH_\bu$ as $n$-groupoids in the category of graded manifolds. Applying the tangent functor, for a fixed $k$, $T_x^k\mathcal{G}_\bu$ and $T^k_{\Phi_0(x)}\mathcal{H}_\bu$ are still $n$-groupoids in the category of vector spaces with surjective maps as covers.  All face and degeneracy
maps come from taking tangent. 

For each $k$,  the map $T_x^k\Phi_\bu\colon T_x^k\mathcal{G}_\bu\rightarrow T^k_{\Phi_0(x)}\mathcal{H}_\bu$ is a weak equivalence of such $n$-groupoid objects by a similar argument as in the proof of \cite[Lemma~2.27]{cueca-zhu}.  Applying the Dold-Kan functor, we obtain a quasi-isomorphism $N^\bu (T_x^k\Phi)\colon N^\bu(T_x^k\mathcal{G})\rightarrow N^\bu(T^k_{\Phi_0(x)}\mathcal{H})$. 

If we free $k$, the induced morphism of double cochain complexes \eqref{diag:double-tang} 
is a column-wise quasi-isomorphism.  Therefore the map between the spectral sequences $$E_0^{p,q}=N^{p}(T^{q}_{x}\huaG)\quad\text{and}\quad {E'}_0^{p,q}=N^{p}(T^q_{\Phi(x)} \huaH)$$ associated to the two double complexes is an isomorphism after turning one page \cite[\href{https://stacks.math.columbia.edu/tag/012X}{Tag 012X}]{stacks-project}.  Such a spectral sequence converges to the associated graded of the cohomology of the total complex. A filtered morphism $f\colon  V\to W$ for bounded filtered vector spaces is an isomorphism if the associated graded is an isomorphism. Our finiteness condition implies that both column-wise filtrations and row-wise filtrations are bounded. Thus $\mathbb{T}_x^\bu \Phi$ is a quasi-isomorphism on the level of total complexes.

Now assume that $\Phi_\bu\colon \mathcal{G}_\bu\rightarrow \mathcal{H}_\bu$ is a levelwise weak equivalence of derived manifolds. For each $l\in\N$, the simplicial morphism $T_x^\bu\Phi_{l}\colon T_x^\bu\mathcal{G}_{l}\rightarrow T_{\Phi_0(x)}^\bu\mathcal{H}_{l}$ restricts to a quasi-isomorphism of cochain complexes. 
    
Recall that the $n$-th level of a simplicial object $A_\bu$ in an abelian category can be splitted functorially as $$A_n=N^n(A)\oplus D^n(A),$$ where $N(A)$ and $D(A)$ denote the normalized and degenerate complex associated to $A$, respectively (cf.~\cite[Lemma~8.3.7]{wei:hom}). As a consequence 
$N^{-l}(T_x^\bu\Phi)\colon N^{-l}(T_x^\bu\mathcal{G})\rightarrow N^{-l}(T_{\Phi_0(x)}^\bu\mathcal{H})$ is a quasi-isomorphism for every $l\in \Z^{\ge 0}$. Since the induced morphism of double cochain complexes $$N^\bu(T^\bu_x\Phi)\colon N^\bu(T_x^\bu\mathcal{G})\rightarrow N^\bu(T^\bu_{\Phi_0(x)}\mathcal{H})$$ is a row-wise quasi-isomorphism, by a similar reason as above, it induces a quasi-isomorphism of the total complexes. 
\end{proof}

\begin{proposition}\label{prop:ME-symp-forms}
    Let $\Phi_\bu\colon \huaG_\bu\to\huaH_\bu$ be a Morita morphism of derived Lie $n$-groupoids satisfying the same finiteness condition as in Lemma~\ref{lem:MEtcom}, i.e. there are only finite many non-zero $N^{p}(T^q_{x}(\huaG))$ and  $N^{p}(T^q_{\Phi_0(x)}(\huaH))$ with $p+q=l$ for all $l\in \Z$.  Consider $\omega_\bu^\huaG$ and $\omega_\bu^\huaH$, closed $m$-shifted $2$-forms on $\huaG_\bu$ and $\huaH_\bu$, respectively, satisfying 
    \begin{equation}\label{eq:omega-GH}
        \omega_\bu^\huaG-\Phi_\bu^*\omega_\bu^\huaH=D\beta_\bu
    \end{equation}
    for some $(m-1)$-shifted $2$-form $\beta_\bu$ on $\huaG_\bu$. Then $\omega^\huaG_\bu$ is non-degenerate if and only if $\omega^\huaH_\bu$ is non-degenerate.
\end{proposition}

\begin{proof}
Cor.~\ref{gaugeinvar} implies that $\omega_\bu^{\huaG}$ is non-degenerate if and only if $\Phi_\bu^*\omega_\bu^\huaH$ is non-degenerate. 

Now first suppose that $\Phi_\bu$ is levelwise equivalence or a hypercover. Then every classical point $y\in \pi_0(\huaH_0)$ has a preimage $x\in \huaG_0$.  We show that $\omega_\bu^\huaH|_y$ is non-degenerate at $y$ if and only if  $\Phi_\bu^*\omega_\bu^\huaH|_x$ is non-degenerate. This is clear from the following diagram and Lemma~\ref{lem:MEtcom}. 
\begin{equation}\label{diag:lambda-pb}
\begin{tikzcd}[row sep=large]
\mathbb{T}^\bu_y\huaH\ar[r,"\lambda^{\omega^\huaH_\bu}"] & \mathbb{T}_y^{*,\bu}\huaH[m]\ar[d,"{\bT^*\Phi_\bu[m]}"]\\
   \bT^\bu_x \huaG \ar[r,"\lambda^{\Phi^*\omega^\huaH_\bu}"]\ar[u,"\bT\Phi_\bu"] & \bT^{*,\bu}_x \huaG [m]
\end{tikzcd} 
\end{equation}
If $\Phi_\bu$ is a weak equivalence of derived higher groupoids, by the proof of \cite[Prop.~2.12]{Rogers-Zhu:2016}, a weak equivalence in an iCFO can always be replaced by a span of hypercovers. Combining with the result above,  $\omega_\bu^\huaH$ is non-degenerate if and only if  $\Phi_\bu^*\omega_\bu^\huaH$ is non-degenerate for all Morita morphism $\Phi_\bu$. 
\end{proof}

What remains open in order to conclude that $m$-shifted symplectic structures can be transported by Morita morphism is to show that Morita morphisms induce  quasi-isomorphisms between the simplicial graded de Rham complexes. See \cite{weiershausen2025} for a detailed proof of these facts in the realm of Lie $n$-groupoids as announced in \cite{Lesdiablerets}. For a related approach, see \cite{taroyan;de-rham}.

\section{Shifted lagrangian structures}\label{shifted-lagrangian}

Motivated by works in derived algebraic geometry, see e.g.\ \cite{cal:lag, ptvv, saf:qua}, a notion of lagrangian structures on $m$-shifted symplectic Lie groupoids was introduced in \cite{ABC:lag}. Here we extend that definition to the world of derived Lie $n$-groupoids and prove that, under transversality, shifted lagrangian correspondences can be composed. Thus it provides a transversal version of \cite[Thm.~4.4]{cal:lag}. 

\subsection{Isotropic and lagrangian morphisms}

In order to generalize the notion of a lagrangian submanifold from classical to shifted symplectic geometry we need the following definitions.

The \emph{normal complex} of a morphism $\Phi_\bu\colon \mathcal{H}_\bu\rightarrow\mathcal{G}_\bu$ between derived Lie $n$-groupoids at a classical point $x_0\in\pi_0(\huaH_0)$ is defined as the mapping cone of the induced cochain map $\bT^\bu_{x_0}\Phi\colon \bT_{x_0}^\bu\huaH\to\bT_{\Phi_0(x_0)}^\bu\huaG$. More explicitly, it is the cochain complex 
\begin{equation}\label{eq:normal-cx}
\N^{\Phi,\bu}_{x_0}:=\bT^\bu_{x_0}\mathcal{H}\oplus \bT^\bu_{\Phi_0(x_0)}\mathcal{G}[-1]\quad \text{with differential}\quad \mathbb{D}_\Phi=\begin{pmatrix}
\mathbb{D}_\huaH & 0\\
\bT^\bu_{x_0}\Phi &- \mathbb{D}_\huaG
\end{pmatrix}.
\end{equation}
   
Let $(\huaG_\bu,\omega_\bu)$ be an $m$-shifted symplectic derived Lie $n$-groupoid.  An \emph{$m$-shifted isotropic morphism} to $\huaG_\bu$ is a pair $(\Phi_\bu,\beta_\bu)$ consisting of a morphism of derived Lie $n$-groupoids $\Phi_\bu\colon\huaH_\bu\to\huaG_\bu$ and an $(m-1)$-shifted $2$-form $\beta_\bu$ on $\huaH_\bu$ such that $\Phi^*\omega_\bu=D\beta_\bu$.  We also call $\beta_\bu$ an \emph{$m$-shifted isotropic structure} on $\Phi_\bu$.

\begin{proposition}\label{proposition;cochain}
Let $(\huaG_\bu,\omega_\bu)$ be an $m$-shifted symplectic derived Lie $n$-groupoid and let $(\Phi_\bu\colon\huaH_\bu\to\huaG_\bu,\beta_\bu)$ be an $m$-shifted isotropic morphism.  For every $x_0 \in \pi_0(\huaH_0)$ the map 
\begin{equation}\label{eq:lambda-Phi}
\lambda^{\beta_\bu,\Phi,\omega_\bu}_{x_0}\colon\N^{\Phi,\bu}_{x_0}\longto\bT^{\ast,\bu}_{x_0}\huaH[m-1],\qquad
u\oplus v\longmapsto\lambda^{\beta_\bu,\sharp}_{x_0}(u)-(\bT_{x_0}^\bu \Phi)^*\circ\lambda_{\Phi_0(x_0)}^{\omega_\bu,\sharp}(v),
\end{equation}
for $u\in\bT_{x_0}^\bu\mathcal{H}$ and $v\in\bT_{\Phi_0(x_0)}^\bu\mathcal{G}[-1]$, is a cochain map.
\end{proposition}

\begin{proof}
Let $u\oplus v\in\bT_{x_0}^p\mathcal{H}\oplus\bT_{\Phi_0(x_0)}^p\mathcal{G}[-1]=\N^{\Phi,p}_{x_0}$.  By Prop.~\ref{proposition;IM} together with the fact that $\beta_\bu$ is an isotropic structure, we have
\begin{align*}
\mathbb{D}^*_\huaH\lambda_{x_0}^{\beta_\bu,\sharp}(u)&=\lambda_{x_0}^{\beta_\bu,\sharp}(\mathbb{D}_{\huaH} u)-\lambda_{x_0}^{D\beta_\bu,\sharp}(u)\\
&=\lambda_{x_0}^{\beta_\bu,\sharp}(\mathbb{D}_\huaH u)-\lambda^{\Phi^{\ast}\omega_\bu,\sharp}_{x_0}(u)=\lambda_{x_0}^{\beta_\bu,\sharp}(\mathbb{D}_\huaH u)-(\bT^\bu_{x_0}\Phi)^*\big(\lambda^{\omega_\bu,\sharp}(\bT_{x_0}\Phi u)\big).
\end{align*}
Therefore, by \eqref{eq:lambda-Phi} and \eqref{eq:lambda-sharp}, it follows that
\begin{align*}
\lambda^{\beta_\bu,\Phi,\omega_\bu}_{x_0}\mathbb{D}_\Phi(u\oplus v)&=\lambda^{\beta_\bu,\Phi,\omega_\bu}\bigl(\mathbb{D}_\huaH u\oplus( \bT^\bu_{x_0}\Phi(u)-\mathbb{D}_{\huaG} v) \bigr)\\
&=\lambda_{x_0}^{\beta_\bu,\sharp}(\mathbb{D}_\huaH u)-(\bT^\bu_{x_0}\Phi)^*\big(\lambda_{\Phi_0(x_0)}^{\omega_\bu,\sharp}(\bT^\bu_{x_0}\Phi u)\big)+(\bT_{x_0}^\bu\Phi)^*\big(\lambda^{\omega_\bu,\sharp}_{\Phi_0(x_0)}(\mathbb{D}_\huaG v)\big)\\
&=\mathbb{D}^*_\huaH[m-1]^p\lambda_{x_0}^{\beta_\bu,\sharp}(u)+(\bT_{x_0}^\bu\Phi)^*\big(\mathbb{D}^*_\huaG[m]^p\lambda_{\Phi_0(x_0)}^{\omega\bu,\sharp}(v)\big)\\
&=\mathbb{D}^*_\huaH[m-1]^p \lambda^{\beta_\bu,\Phi,\omega_\bu}_{x_0}(u\oplus v).
\end{align*}
In the last step we used that $(\bT^\bu_{x_0}\Phi)^*\mathbb{D}^*_\huaG[m]^p = \mathbb{D}_{\huaH}^*[m]^p(\bT_{x_0}^\bu \Phi)^*=-\mathbb{D}_{\huaH}^*[m-1]^p (\bT_{x_0}^\bu \Phi)^*$. 
\end{proof}

\begin{definition}\label{def:lag}
Let $(\huaG_\bu,\omega_\bu)$ be an $m$-shifted symplectic derived Lie $n$-groupoid.  An \emph{$m$-shifted lagrangian morphism} to $\huaG$ is an $m$-shifted isotropic morphism $(\Phi_\bu\colon\huaL_\bu\to\huaG_\bu,\beta_\bu)$ such that the map
\[
\lambda^{\beta_\bu,\Phi,\omega_\bu}_{x_0}\colon \N^{\Phi,\bu}_{x_0}\longto\bT^{*,\bu}_{x_0}\huaL[m-1]
\]
is a quasi-isomorphism for all $x_0\in\pi_0(\huaL_0)$.  We also say that the pair $(\Phi_\bu,\beta_\bu)$, or the triple $(\huaL_\bu,\Phi_\bu,\beta_\bu)$, is an \emph{$m$-shifted lagrangian} on $(\huaG_\bu,\omega_\bu)$, and that $\beta_\bu$ defines an \emph{$m$-shifted lagrangian structure} on $\Phi_\bu$.
\end{definition}

\begin{example}\label{ex:lagpt}
For every integer $m$ the zero form defines an $m$-shifted symplectic form on a point $*$, viewed as a derived Lie $0$-groupoid.  Every derived Lie $n$-groupoid $\huaL_\bu$ has a unique morphism
$t\colon\huaL_\bu\to*$, and it follows directly from Prop.~\ref{proposition;cochain} and Def.~\ref{def:lag} that an $m$-shifted lagrangian structure on $t$ is the same thing as an $(m-1)$-shifted symplectic structure on $\huaL_\bu$.
\end{example}

In \S\,\ref{sec:red} we will see several examples of $0$-shifted lagrangian structures given by quasi-smooth derived Lie $1$-groupoids.  Let us give here one example between quasi-smooth derived manifolds.

The height function $\mu$ on the sphere $S^2$ generates a Hamiltonian circle action, whose orbits are the latitudinal circles.  The two poles are singular points of $\mu$ even though they should be lagrangians just like the other level sets of $\mu$.  The following example is an attempt to understand such singular lagrangians in terms of derived geometry. 

\begin{example}\label{ep:north-pole}
Let $M$ be the plane $\R^2$ equipped with the standard symplectic form $\omega_\can=dx\wedge dy$.  By Rem.\ \ref{ex:sympnond} we can view $M$ as a $0$-shifted derived symplectic $0$-groupoid.  Let $L=\R^2$ and let $\huaL$ be the derived manifold $\huaZ(L,\underline{\R},s)$, where $s$ is the function $s(u,v)=u^2+v^2$, regarded as a section of the trivial line bundle $\underline{\R}$.  The cohomological vector field on $\huaL$ corresponding to the section $s$ is $Q=(u^2+v^2)\pardif{}{\xi}$, where $\xi$ is the fibre coordinate on $\underline{\R}$, and the classical locus is $\pi_0(\huaL)=\{(0,0)\}$.
Let $\Phi\colon\huaL\to M$ be the blow-down map defined by $\Phi(u,v)=(uv, u)$.  It is easy to verify that the $(-1)$-shifted $2$-form 
\[
\beta=\frac{1}{2}dv\wedge d\xi\in\Omega^{2,-1}(\huaL)
\]
defines a $0$-shifted isotropic structure on the morphism $\Phi$.  This isotropic structure is lagrangian: at the classical point $x_0=(0,0)\in\pi_0(\huaL)$ the lagrangian condition holds because the vertical maps in the diagram
\begin{equation*}\label{eq:lagexstrange}
\begin{tikzcd}[ampersand replacement=\&,column sep=large,row sep=large]
\N^{\Phi,\bu}_{x_0}\colon\ar[d,"\lambda^{\beta_\bu,\Phi,\omega_\bu}_{x_0}"']\&\R^2
\arrow[r,"{\Bigl(\begin{smallmatrix}0&0\\0&0\\1&0\end{smallmatrix}\Bigr)}"]
\arrow[d,"{(\begin{smallmatrix}0&\frac12\end{smallmatrix})}"']
\&\R^3
\arrow[d,"{\Bigl(\begin{smallmatrix}0&-1&0\\-\frac12&0&0\end{smallmatrix}\Bigr)}"]
\\
\bT^{\ast,\bu}_{x_0}\huaL[-1]\colon\&\R
\arrow[r,"{\bigl(\begin{smallmatrix}0\\0\end{smallmatrix}\bigr)}"]
\&\R^2
\end{tikzcd}
\end{equation*}
define a quasi-isomorphism from the complex in the top row to the complex in the bottom row.
\end{example}

Although our shifted lagrangians are not Morita invariant, it is still necessary to identify certain ones among them via equivalences.

\begin{definition}\label{def:eqSlag}
Let $(\huaG_\bu,\omega_\bu)$ be an $m$-shifted symplectic derived Lie $n$-groupoid and $(\huaL_\bu^1,\Phi_\bu^1,\beta_\bu^1)$, $(\huaL_\bu^2,\Phi_\bu^2,\beta_\bu^2)$ $m$-shifted lagrangians on it. We say that they are \emph{strictly equivalent} if there exists a Morita morphism $\Psi_\bu\colon \huaL_\bu^1\to\huaL_\bu^2$ together with an $(m-2)$-shifted $2$-form $\gamma_\bu$ on $\huaL_\bu^1$ such that
$$\Phi_\bu^1=\Phi_\bu^2\circ\Psi_\bu\quad\text{and}\quad \beta^1_\bu-\Psi_\bu^*\beta_\bu^2=D\gamma_\bu.$$ The {\em equivalence of $m$-shifted lagrangians} is the equivalence relation generated by these strict equivalences.  
\end{definition}

It follows from Lemma \ref{lem:MEtcom} that if we  have an $m$-shifted lagrangian $\mathcal{L}^2_\bu$, a Morita morphism $\Psi_\bu:\huaL^1_\bu\to\huaL_\bu^2$ and a $\gamma_\bu$ as in the previous definition, then $(\huaL^1_\bu, \Phi^2_\bu\circ\Psi_\bu,\Psi_\bu^*\beta_\bu^2+D\gamma_\bu)$ gives us an $m$-shifted Lagrangian. In order to illustrate the Def. \ref{def:eqSlag}, let us show how the lagrangian tubular neighborhood provides a strict equivalence between $0$-shifted symplectic derived manifolds (seen as derived Lie $0$-groupoids).

\begin{example}
Let $(M,\omega)$ be an ordinary symplectic manifold (viewed as a $0$-shifted symplectic derived manifold) and $j\colon L\to M$ a lagrangian submanifold (viewed as a $0$-shifted lagrangian morphism).  Let $\omega_\can$ be the canonical symplectic form on $T^*L$ and let $0_L\colon L\to T^*L$ be the zero section.  By Weinstein's lagrangian tubular neighbourhood theorem there exist open neighbourhoods $U$ of $L$ in $M$ and $V$ of the zero section in $T^*L$ and a diffeomorphism $\Phi\colon V\to U$ satisfying
\[\Phi^*\omega=\omega_\can\qquad\text{and}\qquad\Phi\circ 0_L=j.\]
Lemma~\ref{lemma;path}\eqref{item;submanifold} implies that the zero section provides a weak equivalence of derived manifolds
    $$0_L\colon  L\to \huaZ(T^*L,\pi^*_{T^*L}T^*L,\epsilon_{T^*L})$$
and since $V\subseteq T^*L$ is a neighborhood of the zero section we have that 
$$0_L\colon  L\to \huaZ(V,(\pi^*_{T^*L}T^*L)_{|V},\epsilon_{T^*L})$$
remains a weak equivalence. Denote this derived manifold by $\huaL= \huaZ(V,(\pi^*_{T^*L}T^*L)_{|V},\epsilon_{T^*L})$.

Let the composition of $\Phi\colon V\to U$ with the natural inclusion $U\hookrightarrow M$ again be denoted by $\Phi\colon V\to M$. This morphism can also be viewed as a morphism of derived manifolds 
$\widetilde{\Phi}\colon \huaL\to M$ 
by defining $\widetilde{\Phi}$ on the base as $\Phi\colon V\to M$ and sending the derived part of $\huaL$ to zero. 

If we define $\beta\in\Omega^{2,-1}(\huaL)$ as the linear form corresponding to $(\omega_\can)_{|V}$, see \S \ref{subsection;linear}, one has 
$$D\beta=(d+\Lie_{\epsilon_{T^*L}})\beta=(\omega_\can)_{|V}=\widetilde{\Phi}^*\omega,$$
thus obtaining an isotropic structure. We next show that
it is $0$-shifted lagrangian. Observe that $\pi_0(\huaL)=\{p\in V\ | \ p=0_L(l)\text{ for } l\in L\}$ and the tangent complex can be identified as
$$\bT_{0_L(l)}\huaL\equiv ( T_{0_L(l)}V\xrightarrow{-d\epsilon_{T^*L}} T^*L) \equiv (T_lL\oplus T^*_lL\xrightarrow{(0,-\id)}T^*_lL).$$
Hence the map $\lambda^{\beta,\widetilde{\Phi},\omega}\colon \mathbb{N}^{\widetilde{\Phi},\bu}_{0_L(l)}\to \bT^{*,\bu}_{0_{L}(l)}\huaL[-1]$ reads as the following commutative diagram
\begin{equation}\label{diag:lam}
\begin{tikzcd}[ampersand replacement=\&,column sep=large,row sep=large]
T_lL\oplus T^*_lL
\arrow[r,"{\bigl(\begin{smallmatrix}0&-\id\\T_lj&A\end{smallmatrix}\bigr)}"]
\arrow[d,"{(\begin{smallmatrix}\id&0\end{smallmatrix})}"']
\&T_l^*L\oplus T_{\Phi(l)}M
\arrow[d,"{\Bigl(\begin{smallmatrix}\id&T_l^*j\circ\omega^\sharp_l\\0&A^t\circ\omega^\sharp_l\end{smallmatrix}\Bigr)}"]
\\
T_lL
\arrow[r,"{\bigl(\begin{smallmatrix}0\\\id\end{smallmatrix}\bigr)}"]
\&T^*_lL\oplus T_lL,
\end{tikzcd}
\end{equation}
where $A$ denotes the restriction of $T_l\Phi$ to the vertical directions. One has $A^t\circ\omega^\sharp_l\circ A=0$ and $T^*j\circ \omega^\sharp_l\circ A=\id$.  A direct computation shows that the vertical arrows in \eqref{diag:lam} indeed form a quasisomorphism.

Lastly, we show that the two different $0$-shifted lagrangian structures on $(M,\omega)$ given by $(L,j,0)$ and $(\huaL,\widetilde{\Phi},\beta)$ are strictly equivalent. For that recall that the zero section is a weak equivalence of derived manifolds and makes the following diagram commutative.
\[\begin{tikzcd}                                                 & \huaL \arrow[rd, "\widetilde{\Phi}" description] & \\
L \arrow[rr, "j" description] \arrow[ru, "0_L" description] &   & M
\end{tikzcd}\]

Further, since $0_L^*\beta=0$ and $0_L$ is a weak equivalence of derived manifolds, we obtain that the lagrangian tubular neighborhood theorem provides a factorization of the morphism $j\colon L\to M$ into a weak equivalence followed by a fibration $\widetilde{\Phi}$, in the sense of \cite{blx}, in the category of shifted lagrangians. 
\end{example}

\subsection{Towards Weinstein's category for shifted symplectic structures}\label{subsection;symplectic-category}

Inspired by ideas from geometric quantization, it was proposed in \cite{wei:sympcat} that symplectic geometry should be encoded in a ``category'' whose objects are symplectic manifolds and morphisms are lagrangian correspondences.  However, the composition of morphisms was not always defined. More recently,  \cite{cal:lag, ptvv} where able to define a composition law for
 shifted symplectic derived Artin stacks using shifted lagrangian correspondences. Higher-categorical refinements of these are later developed by \cite{amorim-ben-bassat,cal:AKSZ}. Here we define a partial  composition law in our smooth setting of shifted symplectic structures on derived Lie $n$-groupoids. 

Let $(\mathcal{G}^1_\bu,\omega^1_\bu)$ and $(\mathcal{G}^2_\bu,\omega^2_\bu)$ be two $m$-shifted symplectic derived Lie $n$-groupoids. An \emph{$m$-shifted lagrangian correspondence} between $(\mathcal{G}^1_\bu,\omega^1_\bu)$ and $(\mathcal{G}^2_\bu,\omega^2_\bu)$ is given by an $m$-shifted lagrangian $$(\mathcal{L}_\bu,\Phi^1_\bu\times\Phi^2_\bu,\beta_\bu)\quad\text{on the product}\quad(\mathcal{G}^1_\bu\times \mathcal{G}^2_\bu,-\pr_1^{\ast}\omega^1_\bu+\pr_2^{\ast}\omega^2_\bu).$$  We will denote an $m$-shifted lagrangian  correspondence by $$(\mathcal{L}_\bu,\Phi^1_\bu\times\Phi^2_\bu,\beta_\bu)\colon(\ca{G}^1_\bu,\omega^1_\bu)\dashrightarrow(\ca{G}^2_\bu,\omega^2_\bu),$$ or simply by $\ca{L}_\bu\colon\ca{G}^1_\bu\dashrightarrow\ca{G}^2_\bu$
if the morphisms and forms are clear from the context.

We shall see in \S\,\ref{subsection;epilogue} that derived symplectic reduction is a lagrangian correspondence.  Here are some further examples.

\begin{example}[Symplectic Morita equivalences gives rise to lagrangian correspondences]\label{ep:symp-ME-lag-corr}
Let 
\[
\begin{tikzcd}
(\mathcal{H}^1_\bu,\omega^1_\bu)&\ar[l,"\Phi_\bu"'](\mathcal{G}_\bu, \eta_\bu)\ar[r,"\Psi_\bu"]&(\mathcal{H}^2_\bu,\omega^2_\bu)
\end{tikzcd}
\]
be a symplectic Morita equivalence of $m$-shifted symplectic derived higher Lie groupoids, where $\eta_\bu$ is the $(m-1)$-shifted $2$-form on $\mathcal{G}_{\bu}$ such that $-\Phi^{\ast}\omega^1_\bu+\Psi^{\ast}\omega^2_\bu=D\eta_\bu$.  We assert that 
\[
\begin{tikzcd}
(\mathcal{G}_\bu,\Phi_\bu\times\Psi_\bu,\eta_\bu)\colon(\ca{H}^1_\bu,\omega_\bu^1)\ar[r,dashed]&(\ca{H}^2_\bu,\omega_\bu^2)
\end{tikzcd}
\]
is an $m$-shifted lagrangian correspondence.  To prove this we must show that the cochain map 
\[
\lambda=\lambda_{x_0}^{(\eta_\bu,\Phi\times \Psi,-\omega_\bu^1\times \omega^2_\bu)}\colon\mathbb{T}^\bu_{x_0}\mathcal{G}\oplus \mathbb{T}^\bu_{\Phi_0(x_0)}\mathcal{H}^1[-1]\oplus \mathbb{T}^\bu_{\Psi_0(x_0)}\mathcal{H}^2[-1]\longto \mathbb{T}_{x_0}^{\ast,\bu}\mathcal{G}[m-1]
\]
defined by 
\[
\lambda(u,v_1,v_2)=\lambda_{x_0}^{\eta_\bu,\sharp}(u)+(\bT_{x_0}^\bu\Phi)^*\lambda_{\Phi_0(x_0)}^{\omega^1_\bu,\sharp}(v_1)-(\bT^\bu_{x_0}\Psi)^*\lambda_{\Psi_0(x_0)}^{\omega_\bu^2,\sharp}(v_2)
\]
is a quasi-isomorphism for all classical points $x_0\in\pi_0(\huaG_0)$. 
Define the cochain maps 
\begin{align*}
\Gamma_{x_0}&\colon \mathbb{T}^\bu_{x_0}\mathcal{G}[-1]\longto \mathbb{T}^\bu_{x_0}\mathcal{G}\oplus \mathbb{T}^\bu_{x_0}\mathcal{G}[-1]\oplus \mathbb{T}^\bu_{x_0}\mathcal{G}[-1],\\
\alpha_{x_0}&\colon \mathbb{T}^\bu_{ x_0}\mathcal{G}\oplus( \mathbb{T}^\bu_{x_0}\mathcal{G}\oplus \mathbb{T}^\bu_{x_0}\mathcal{G})[-1]\longto\mathbb{T}^\bu_{x_0}\mathcal{G}\oplus( \mathbb{T}^\bu_{\Phi_0(x_0)}\mathcal{H}^1\oplus\mathbb{T}^\bu_{\Psi_0(x_0)}\mathcal{H}^2)[-1]
\end{align*}
by
\begin{align*}
\Gamma_{x_0}(v)&=\Bigl(0,\frac{v}{2},-\frac{v}{2}\Bigr)\\
\alpha_{x_0}&=\id\oplus (\mathbb{T}^\bu_{x_0}\Phi\oplus \mathbb{T}^\bu_{x_0}\Psi)[-1]
\end{align*}
These are quasi-isomorphisms by Rem.~\ref{rmk:homG} and Lemma~\ref{lem:MEtcom}, respectively.  Further define  
\begin{gather*}
       \tilde{\lambda}:= \lambda\circ\alpha_{x_0}\circ\Gamma_{x_0}=\lambda_{x_0}^{\frac{1}{2}\Phi^{\ast}\omega_\bu^1+\frac{1}{2}\Psi^{\ast}\omega_\bu^2,\sharp}[-1]
\end{gather*}
Since $\alpha\circ\Gamma$ is a quasi-isomorphism, it is sufficient to show that $\tilde{\lambda}$ is a quasi-isomorphism.  We note that 
    \begin{gather*}
        \tilde{\lambda}+\lambda_{x_0}^{\frac{1}{2}D\eta_\bu,\sharp}[-1]=\lambda_{x_0}^{\frac{1}{2}\Phi^{\ast}\omega_\bu^1+\frac{1}{2}\Psi^{\ast}\omega_\bu^2,\sharp}[-1]+\lambda_{x_0}^{-\frac{1}{2}\Phi^{\ast}\omega_\bu^1+\frac{1}{2}\Psi^{\ast}\omega_\bu^2,\sharp}[-1]=\lambda_{x_0}^{\Psi^{\ast}\omega_\bu^2,\sharp}[-1].
    \end{gather*}
    Because $\lambda_{x_0}^{\frac{1}{2}D\eta_\bu,\sharp}$ is nullhomotopic by Prop.~\ref{proposition;IM} and $\lambda_{x_0}^{\Psi^{\ast}\omega_\bu^2,\sharp}$ is a quasi-isomorphism, it follows that $\tilde{\lambda}$ is a quasi-isomorphism as well.
\end{example}

\begin{example}[Diagonal]\label{ep:diag} 
    Let $(\huaG_\bu, \omega_\bu)$ be an $m$-shifted symplectic derived Lie $n$-groupoid. The identity map $\id:(\huaG_\bu, \omega_\bu)\to (\huaG_\bu, \omega_\bu)$ gives rise to a  symplectic Morita equivalence. Then Ex.~\ref{ep:symp-ME-lag-corr} tells us that the diagonal map $\Delta:(\huaG_\bu, 0) \to (\huaG_\bu \times \huaG_\bu, -pr_1^*\omega+pr_2^*\omega ) $ is an $m$-shifted lagrangian. 
\end{example}

A classical example of a lagrangian correspondence is as follows: let $f\colon M\to N$ be a smooth map and $\Gamma_f\subset M\times N$ denote its graph, then 
\[
f^*T^*N \cong \overline{\ann}(T\Gamma_f):=\{\,(\alpha, \beta)\mid\alpha(v)=\beta(Tf(v)) \text{ for all $v\in TM$}\,\}
\]
is a lagrangian correspondence $T^*M\dashrightarrow T^*N$. Notice that the usual annulator $\ann(T\Gamma_f)$ is a lagrangian in $T^*M \times T^*N$ and $\overline{\ann}(T\Gamma_f)$ is a lagrangian in $\overline{T^*M}\times T^*N$. A multiplicative analogue of this fact is as follows.

\begin{example}
Let $G\rightrightarrows M$ and $H\rightrightarrows N$ be Lie groupoids and $\Phi_\bu\colon G_\bu\to H_\bu$ a groupoid morphism then
$$\ann(T\Gamma_\Phi)\rightrightarrows \ann(\Gamma_{\text{Lie}(\Phi)})\subset T^*(G\times H)\rightrightarrows (A_G\times A_H)^*$$
is a lagrangian VB-subgroupoid with base given by the Lie groupoid $\Gamma_{\Phi_1}\rightrightarrows\Gamma_{\Phi_0}$. Using Prop.~\ref{ex:VBgrpd}, we obtain that 
    \begin{equation} \label{eq:gr-f-ann}
        \huaA nn(T\Gamma \Phi):= \Big(\huaZ\big(\overline{\ann}(T\Gamma_\Phi)\rightrightarrows \overline{\ann}(\Gamma_{\text{Lie}(\Phi)}), 0_{\Gamma_{\Phi_1}}\big), \pr_1\times \pr_2,0\Big)\colon \huaC ot_\bu(G)\dashrightarrow \huaC ot_\bu(H)
    \end{equation}
is a $0$-shifted lagrangian correspondence, where $\huaC ot_\bu(G)$ and $\huaC ot_\bu(H)$ denotes the $0$-shifted symplectic derived Lie groupoids defined by \eqref{eq:def0cot}. If $\Phi$ is an \'etale hypercover, that is $\Phi_0\colon  G_0\to H_0$ is \'etale (this implies that $\Phi_i\colon  G_i \to H_i$ is \'etale), then $\huaA nn(T\Gamma \Phi)$ provides a span of hypercovers between $\huaC ot_\bu(G)$ and $\huaC ot_\bu(H)$ by natural projections. If $pr_{pt}\colon  (V\times V \Rightarrow V)\to (pt \Rightarrow pt)$ is the projection from the pair groupoid to a point, then $\huaA nn( T\Gamma pr) = (0_{V \times V} \Rightarrow 0_V)$ provides a span of weak equivalence whose right leg is a hypercover. Since a hypercover $\Phi\colon  G\to H$ can be always decomposed to a composition of \'etale hypercovers and $\id \times \pr_\pt$, $\huaA nn(T\Gamma \Phi)$ is a composition of spans of hypercovers and a weak equivalences, thus a Morita equivalence. Since a Morita equivalence of Lie groupoids can be written as a span of hypercovers, this shows that if $G$ and $H$ Morita equivalent,  then $\huaC ot_\bu(G)$ and $\huaC ot_\bu(H)$ are also Morita equivalent.  
\end{example}

The main result of this section is the following lagrangian composition theorem.

\begin{theorem}\label{thm:complag}
Consider the $m$-shifted lagrangian correspondences
\[\begin{tikzcd}
	{(\mathcal{G}^1_\bu,\omega^1_\bu)} &&& {(\mathcal{G}^2_\bu,\omega^2_\bu)} &&& {(\mathcal{G}^3_\bu,\omega^3_\bu)}
	\arrow["{(\mathcal{L}_\bu,\Phi^1_\bu\times\Phi^2_\bu,\beta_\bu)}", dashed, from=1-1, to=1-4]
	\arrow["{(\mathcal{L}'_\bu,\Psi^2_\bu\times\Psi^3_\bu,\beta'_\bu)}", dashed, from=1-4, to=1-7]
\end{tikzcd}\]
between $m$-shifted symplectic derived Lie $n$-groupoids. If $\Phi^2_\bu$ and $\Psi_\bu^2$ are levelwise transversal and the derived simplicial manifold $\huaL_\bu\times_{\huaG^2}\huaL'_\bu$ is a derived Lie $n$-groupoid,  then 
\begin{equation}\label{eq:comp-lag}
    \big(\huaL_\bu\times_{\huaG^2_\bu}\huaL'_\bu, (\Phi^1_\bu\circ\pr_1)\times(\Psi^3_\bu\circ\pr_2), \pr^*_1\beta_\bu+\pr_2^*\beta'_\bu\big)\colon \huaG^1_\bu\dashrightarrow\huaG^3_\bu
\end{equation}
is an $m$-shifted lagrangian correspondence.
\end{theorem}
 
In \cite{Hor71} it was observed that, under transversality, the set-theoretic composition of canonical relations is again a canonical relation and represents the classical limit of the composition of Fourier integral operators. The work \cite{wei:sympcat} is a geometric abstraction of this principle and Thm.~\ref{thm:complag} fit in that framework.  There, composition is defined by first taking the fibre product over the middle symplectic object and then projecting to the outer factors. Under suitable transversality and injectivity assumptions, this projection is an isomorphism onto its image, so one can identify the composition with the fibre product itself. (For a detailed proof see  \cite[Lemma~2.0.5]{wehrheim-woodward:10}.)  Under the weaker hypotheses of Thm.~\ref{thm:complag} these assumptions typically fail, even in linear examples, which is why we deliberately do not build the projection into the definition of composition.

Specializing to the case where $\huaG_\bu^1=\huaG_\bu^3=*$ is an $m$-shifted symplectic point (Rem.\ \ref{ex:lagpt}), we obtain the following lagrangian intersection theorem.

\begin{theorem}\label{theorem;lagintersection}
    Let $(\huaL_\bu, \Phi_\bu, \beta_\bu)$ and $(\huaL'_\bu, \Phi'_\bu, \beta'_\bu)$ two $m$-shifted lagrangians of the $m$-shifted sympelctic derived Lie $n$-groupoid $(\huaG_\bu, \omega_\bu)$. If $\Phi$ and $\Phi'$ are levelwise transversal and the fibre product $\huaL_\bu\times_{\huaG_\bu}\huaL'_\bu$ is a derived Lie $n$-groupoid, then $\pr_1^*\beta_\bu-\pr_2^*\beta'_\bu$ is an $(m-1)$-shifted symplectic structure on it.
\end{theorem}

\begin{remark}
It is proved in \cite{Dorsch:24} that simplicial manifolds always satisfy Kan conditions locally, an analogous property holds for derived simplicial manifolds (cf. \cite{Dorsch:thes}). This suggests that the assumption that the derived simplicial manifold $\huaL_\bu\times_{\huaG^2_\bu}\huaL'_\bu$ is a derived Lie $n$-groupoid is not unrealistically strong.  Moreover, in the following situations 
\begin{enumerate}
    \item one of the maps $\Phi^2_\bu\colon \huaL_\bu \to \huaG_\bu^2$ and $\Psi^2_\bu\colon \huaL_\bu'\to \huaG_\bu^2$ is a fibration of derived Lie $n$-groupoids;
    \item $n=1$ and all derived manifolds involved are quasi-smooth;
\end{enumerate}
this added condition is automatically satisfied. If (1) holds, by levelwise transversality,  $\huaL_0\times_{\huaG_0^2} \huaL'_0$ is a derived manifold. Then a similar argument as in  \cite[Prop.~7.10]{Rogers-Zhu:2016} shows that $\huaL_\bu\times_{\huaG^2_\bu}\huaL'_\bu$ is a derived Lie $n$-groupoid. 

In situation (2), notice that by Prop.~\ref{ex:VBgrpd}, quasi-smooth Lie groupoids are simply VB-groupoids with multiplicative sections. Then we apply \cite[Prop.~A.1.4]{BuCaHo16} to both levels of the VB-groupoids. Notice that transversal pairs are special cases of good pairs therein. We obtain that the fibred product of levelwise transversal VB-groupoids is a VB-groupoid again. It is clear that the fibred product of sections is again a section, and the fibred product of multiplicative sections remains multiplicative. Thus the fibred product of quasi-smooth Lie groupoids stays a quasi-smooth Lie groupoid. 

Our examples all fall into one of the two situations. In fact, a situation we often meet in examples is that one of the maps  $\Phi_\bu^2\colon \huaL \to \huaG^2$ and $\Psi_\bu^2\colon \huaL'\to \huaG^2$ is a strong fibration. Then both conditions in the above Thm.~\ref{thm:complag} are implied. 
\end{remark}

\begin{remark}\label{rmk:ptvv-thm5}
In the case that the transversality condition is not satisfied in Thm.~\ref{thm:complag}, we may do a fibrant replacement: Using the same notation as in Thm.~\ref{thm:complag} suppose that $$\Phi^2_\bu \times \Psi^3_\bu\colon  \huaL'_\bu \to \huaG^2_\bu \times \huaG^3_\bu$$ splits into a Morita morphism $\sigma\colon \huaL'_\bu\to \tilde{\huaL}'_\bu$ followed by a strong fibration $p\colon \tilde{\huaL'}_\bu \to\huaG^2_\bu \times \huaG^3_\bu $ and further assume\footnote{We expect this to be automatic as remarked in \S~\ref{sec:symp-ME}.} that there exists an $(m-1)$-shifted 2-form $\tilde{\beta}'_\bu$ on $\tilde{\huaL'}_\bu$ and an $(m-2)$-shifted 2-form $\gamma'_\bu$ on $\huaL'_\bu$ such that $\beta'_\bu-\sigma^*\tilde{\beta}'_\bu = D\gamma_\bu$.  Notice that for a derived Lie $n$-groupoid $\huaG_\bu$, the projection to the terminal object $*$ is not just a fibration, but also a strong fibration because the terminal map in $\DMfd$ is a fibration. Thus $pr_1\circ p\colon \tilde{\huaL}'_\bu \to \huaG^2_\bu$ is also a strong fibration. Therefore, 
$\huaL_\bu \times_{\huaG^2} \tilde{\huaL}'_\bu $ is a homotopy pullback $ \huaL_\bu \times_{\huaG^2}^h \huaL'_\bu $. Clearly, the replacement $\tilde{\huaL}'_\bu $ satisfies all the conditions in Thm.~\ref{thm:complag} even though $\huaL'_\bu$ does not. Thus implied by Thm.~\ref{thm:complag}, when such fibrant replacement exists, the homotopy pullback $ \huaL_\bu \times_{\huaG^2}^h \huaL'_\bu $ (see \S\,\ref{subsection;pull}) is an $m$-shifted lagrangian correspondence. This confirms one of the main results in \cite[Thm.~0.5]{ptvv}. 
\end{remark}

\begin{remark}\label{rmk:compare-ptvv-comp}
    Despite Thm.~\ref{thm:complag} implies \cite[Thm.~0.5]{ptvv} as explained in Rem.\ \ref{rmk:ptvv-thm5}, the composition in Thm.~\ref{thm:complag} and the composition in \cite[Thm.~0.5]{ptvv} could be different. This is because,  in general, transversality is not strong enough to guarantee that usual pullbacks are the same as homotopy pullbacks. The following example illustrates this subtleness.  
    When $G$ is a discrete Lie group the zero is a $2$-shifted symplectic structure on the nerve  $N_\bu G$ as given in Rem. \ref{rmk:gp}, see \cite{saf:qua, cueca-zhu}. Thus the diagonal map $\Delta:(N_\bu G, 0) \to \overline{N_\bu G}\times N_\bu G $ provides a 2-shifted  lagrangian (see Ex.~\ref{ep:diag}). By dimensional reason, $\Delta$ is transversal to itself levelwise, and the (strict) fibre-product $$N_\bu G \times_{N_\bu G \times N_\bu G} N_\bu G = N_\bu G$$ is still a Lie 1-groupoid endowed with $0$ as a $1$-shifted symplectic structure. On the other hand,  a fibrant replacement of the diagonal is provided by the path object 
$(N_\bu G)^I$. A calculation shows that the homotopy pullback is 
$$N_\bu G\times^h _{N_\bu G\times N_\bu G} N_\bu G = N_\bu G \times _{N_\bu G\times N_\bu G} (N_\bu G)^I= (G\ltimes G)_\bu$$ 
where $(G\ltimes G)_\bu$ denotes the action groupoid of $G$ on itself by conjugation with the $0$ as a $1$-shifted symplectic structure. Therefore these two compositions are different even when all assumptions are satisfied.
\end{remark}

In order to prove Thm.~\ref{thm:complag} we need the following auxiliary lemma that is of independent interest.

\begin{lemma}\label{lem:transversality}
Let $\Phi_\bu\colon \huaG_\bu\to\huaK_\bu$ and $\Psi_\bu\colon \huaH_\bu\to\huaK_\bu$ be morphisms of derived Lie $n$-groupoids.  If $\Phi_\bu$ and $\Psi_\bu$ are levelwise transversal and $\huaG_\bu\times_{\huaK_\bu}\huaH_\bu$ is a derived Lie $n$-groupoid, then its tangent complex is 
\[
\bT^\bu_{(g_0,h_0)}(\huaG\times_{\huaK}\huaH)=\bT^\bu_{g_0}\huaG\times_{\bT^\bu_{k_0}\huaK}\bT^\bu_{h_0}\huaH
\]
which is quasi-isomorphic to
\begin{equation}\label{eq:fibre-pd-new}
\bT_{g_0}^\bu\huaG\oplus\bT_{h_0}^\bu\huaH\oplus\bT^\bu_{k_0}\huaK[-1]\quad\text{with differential}\quad 
\mathbb{D}=\begin{pmatrix}
        \mathbb{D}_\huaG &0&0\\
        0&\mathbb{D}_\huaH& 0\\
        \bT_{g_0}\Phi&-\bT_{h_0}\Psi&-\mathbb{D}_\huaK
    \end{pmatrix},
\end{equation}
for all $g_0\in\pi_0(\huaG_0)$, $h_0\in\pi_0(\huaH_0)$, and $k_0\in\pi_0(\huaK_0)$ such that $\Phi_0(g_0)=k_0=\Psi_0(h_0)$.  An explicit quasi-isomorphism is
    \begin{equation}\label{eq:jmap}
        J_{(g_0,h_0)}\colon \bT^\bu_{g_0}\huaG\times_{\bT^\bu_{k_0}\huaK}\bT^\bu_{h_0}\huaH\to  \bT_{g_0}^\bu\huaG\oplus\bT_{h_0}^\bu\huaH\oplus\bT^\bu_{k_0}\huaK[-1],\quad J_{(g_0,h_0)}(u,v)=(u,v,0).
    \end{equation}
\end{lemma}

\begin{proof}
We first prove that the tangent complex $\bT^\bu_{(g_0,h_0)}(\huaG\times_{\huaK}\huaH)$ is equal to the fiber product of cochain complexes $\bT^\bu_{g_0}\huaG\times_{\bT^\bu_{k_0}\huaK}\bT^\bu_{h_0}\huaH$. By Definition \ref{tangentcomplexLiengroupoid} and Prop.~\ref{proposition;derived-fibred-product} we have $$\bT^\bu_{(g_0,h_0)}(\huaG\times_{\huaK}\huaH)=\tot^\bu\Big(N\big(T_{(g_0,h_0)}(\mathcal{G}\times_{\mathcal{K}}\mathcal{H})\big)\Big)=\tot^\bu\big(N(T_{g_0}\mathcal{G}\times_{T_{k_0}\mathcal{K}}T_{h_0}\mathcal{H})\big).$$ So we need to show that the functor $\bT\colon \mathsf{sCoCh}\xrightarrow[]{N}\mathsf{BiCoCh}\xrightarrow[]{\tot}\mathsf{CoCh}$, defined in sec. \ref{EZ}, preserves fiber products. By the Dold-Kan adjunction,  the functor $N\colon \mathsf{sCoCh}\rightarrow \mathsf{BiCoCh}$ being a right adjoint preserves limits. It can be easily verified that the totalization functor $\tot\colon \mathsf{BiCoCh}\rightarrow \mathsf{CoCh}$ preserves products and equalizers, thus finite limits. Hence, the composition of these functors preserve finite limits and especially fiber products. 
    
We now prove that the tangent complex $\bT^\bu_{(g_0,h_0)}(\huaG\times_{\huaK}\huaH)=\bT^\bu_{g_0}\huaG\times_{\bT^\bu_{k_0}\huaK}\bT^\bu_{h_0}\huaH$ is quasi-isomorphic to the cochain complex given in the statement. By transversality (see \S\,\ref{subsection;transverse}), for any cochain degree $l\in \mathbb{N}$ and any $n\in \mathbb{N}$, the map 
\[
T_{g_0}^l\Phi_n-T_{h_0}^l\Psi_n\colon T_{g_0}^l\mathcal{G}_n\oplus T_{h_0}^l\mathcal{H}_n\rightarrow T_{k_0}^l\mathcal{K}_n
\]
is surjective. As in the proof of Lemma~\ref{lem:MEtcom}, we recall that the
$n$-th level of a simplicial object $A_\bu$ in an abelian category can be splitted functorially as $A_n=N_n(A)\oplus D_n(A)$, where $N(A)$ and $D(A)$ denote the normalized and degenerate complex associated to $A$, respectively. Hence, the maps $$N^n(T_{g_0}^l\Phi)-N^n(T_{h_0}^l\Psi)\colon N^n(T_{g_0}^l\mathcal{G}_\bu)\oplus N^n(T_{h_0}^l\mathcal{H}_\bu)\rightarrow N^n(T_{k_0}^l\mathcal{K}_\bu),\quad  l,n\in \mathbb{Z},$$ and consequently the map $\bT^\bu_{g_0}\Phi-\bT^\bu_{h_0}\Psi\colon \bT_{g_0}^\bu\huaG\oplus\bT_{h_0}^\bu\huaH\rightarrow \bT^\bu_{k_0}\mathcal{K}$ are surjective. As a result, $$\bT^\bu_{g_0}\huaG\times_{\bT^\bu_{k_0}\huaK}\bT^\bu_{h_0}\huaH=\text{ker}(\bT^\bu_{g_0}\Phi-\bT^\bu_{h_0}\Psi)$$ is quasi-isomorphic to the mapping cone of $\bT^\bu_{g_0}\Phi-\bT^\bu_{h_0}\Psi$, which is exactly \eqref{eq:fibre-pd-new}. 
\end{proof}

\begin{remark}\label{rmk:homG}
In the situation of Lemma~\ref{lem:transversality},  the fiber product $\huaG_\bu\times_{\huaK_\bu}\huaH_\bu$ is also given by
\[
(\huaG_\bu\times\huaH_\bu)\times_{(\huaK_\bu\times\huaK_\bu)}\huaK_\bu,
\]
where we used the diagonal.  Therefore, in terms of the tangent complexes we get the following quasi-isomorphisms
\begin{equation}\label{eq:another-qi}
\begin{tikzcd}[column sep=large]
\bT_{g_0}^\bu\huaG\oplus\bT_{h_0}^\bu\huaH\oplus\bT^\bu_{k_0}\huaK[-1]\ar[r,shift left,"\Gamma_{(g_0,h_0)}"]&
\bT_{g_0}^\bu\huaG\oplus\bT_{h_0}^\bu\huaH\oplus\bT^\bu_{k_0}\huaK\oplus \bT^\bu_{k_0}\huaK[-1]\oplus \bT^\bu_{k_0}\huaK[-1]\ar[l,shift left,"\Theta_{(g_0,h_0)}"],
\end{tikzcd}
\end{equation}
where
\begin{equation}\label{eq:Gmap}
    \Gamma_{(g_0,h_0)}(u,v,k)=\Bigl(u,v,\frac{\bT\Phi(u)+\bT\Psi(v)}{2}, \frac{k}{2}, -\frac{k}{2}\Bigr)
\end{equation}
and the $\Theta$ and the cochain homotopy $H$ are given by
\begin{equation*}
    \Theta_{(g_0,h_0)}(u,v,k_a, k_b, k_c)=(u,v, k_b-k_c)\quad\text{and}\quad H_{(g_0,h_0)}(u,v,k_a, k_b, k_c)=\Bigl(0,0, \frac{k_b+k_c}{2},0,0\Bigr)
\end{equation*}
and satiefy the equations
\begin{equation*}
    \Gamma_{(g_0,h_0)}\circ\Theta_{(g_0,h_0)}-\id=H_{(g_0,h_0)}\mathbb{D}+\mathbb{D}H_{(g_0,h_0)} \quad\text{and}\quad \Theta \circ \Gamma =\id.
\end{equation*}
\end{remark}

\begin{proof}[Proof of Thm.~\ref{thm:complag}]
It is clear that $\widehat{\beta}_\bu=\pr^*_1\beta_\bu+\pr_2^*\beta'_\bu$ provides an isotropic structure on the morphism 
\[
\widehat{\Phi}_\bu=(\Phi^1_\bu\circ\pr_1)\times(\Psi^3_\bu\circ\pr_2)\colon \huaL_\bu\times_{\huaG_\bu^2}\huaL'_\bu\to \huaG_\bu^1\times\huaG_\bu^3.
\]
It remains to prove that this structure is lagrangian, in other words that
\begin{equation}\label{eq:maplag1}
        \lambda_{\widehat{l}}^{\widehat{\beta}_\bu,\widehat{\Phi},-\omega^1_\bu+\omega^3_\bu}\colon \N^{\widehat{\Phi},\bu}_{\widehat{l}}\to \bT^{*,\bu}_{\widehat{l}}\widehat{\huaL}[m-1]
\end{equation}
is a quasi-isomorphism for all $\widehat{l}\in\pi_0(\widehat{\huaL}_0)$.  It follows from Lemma~\ref{lem:transversality} that
\begin{align*}
\N^{\widehat{\Phi},\bu}_{\widehat{l}}&=\bT^\bu_{l}\huaL\times_{\bT^\bu_{g^2_0}\huaG^2}\bT^\bu_{l'}\huaL'\oplus\bT^\bu_{\Phi^1(l)}\huaG^1[-1]\oplus\bT^\bu_{\Phi^3(l')}\huaG^3[-1],\\
\bT^{*,\bu}_{\widehat{l}}\widehat{\huaL}[m-1]&=\frac{\bT^{*,\bu}_{l}\huaL[m-1]\oplus\bT^{*,\bu}_{l'}\huaL'[m-1]}{\bT^{*,\bu}_{g^2_0}\huaG^2[m-1]}
\end{align*}
where $\widehat{l}=(l,l')$ with $l\in\pi_0(\huaL_0)$ and $ l'\in\pi_0(\huaL'_0)$ such that $\Phi^2_0(l)=\Psi^2_0(l')=g_0^2$.
Therefore, the map in \eqref{eq:maplag1} is explicitly given by
\begin{equation}\label{eq:lagmaplamb}
    \lambda_{\widehat{l}}^{\widehat{\beta}_\bu,\widehat{\Phi},-\omega^1_\bu+\omega^3_\bu}(u,v,k^1,k^3)=\left[\big(\lambda^{\beta_\bu,\sharp}_l(u)-\bT^*_l\Phi^1\circ\lambda^{-\omega_\bu^1,\sharp}_{\Phi^1(l)}(k^1),\ \lambda^{\beta'_\bu,\sharp}_{l'}(v)-\bT^*_l\Psi^3\circ\lambda^{\omega^3_\bu,\sharp}_{\Psi^3(l')}(k^3)\big) \right].
\end{equation}

If one defines the cochain complex $(\mathbb{A}_{\widehat{l}}^\bu,\mathbb{D_A})$ as 
\begin{equation*}
\mathbb{A}_{\widehat{l}}^\bu=\bT_{l}^\bu\huaL\oplus\bT_{l'}^\bu\huaL'\oplus\bT^\bu_{g^2}\huaG^2\oplus\bT^\bu_{g^2}\huaG^2[-1]\oplus\bT^\bu_{g^2}\huaG^2[-1]\oplus\bT^\bu_{\Phi^1(l)}\huaG^1[-1]\oplus\bT^\bu_{\Phi^3(l')}\huaG^3[-1] 
\end{equation*}
with differential given by the matrix
\begin{equation*}
        \mathbb{D_A}=\begin{pmatrix}
            \mathbb{D}_\huaL &0&0& 0&0&0&0\\
        0&\mathbb{D}_{\huaL'}& 0&0&0&0&0\\
        0&0&\mathbb{D}_{\huaG^2}& 0&0&0&0\\
        \bT_{l}\Phi^2& 0& -\id&-\mathbb{D}_{\huaG^2}& 0&0&0\\
        0&\bT_{l'}\Psi^2&-\id&0&-\mathbb{D}_{\huaG^2}&0&0\\
        \bT_{l}\Phi^1&0&0&0 &0&-\mathbb{D}_{\huaG^1}& 0 \\ 0&\bT_{l'}\Psi^3&0&0&0&0&-\mathbb{D}_{\huaG^3}
        \end{pmatrix},
    \end{equation*}
then, by using the quasi-isomorphisms \eqref{eq:jmap} and \eqref{eq:Gmap}, we obtain that 
$$\lambda_{\widehat{l}}^{\widehat{\beta}_\bu,\widehat{\Phi},-\omega^1_\bu+\omega^3_\bu}= J_{\widehat{l}}^*\circ\Lambda_{\widehat{l}}\circ(\Gamma_{\widehat{l}},\id,\id)\circ(J_{\widehat{l}},\id,\id).$$
Here, $\Lambda_{\widehat{l}}\colon \mathbb{A}^\bu_{\widehat{l}}\to \bT^{*,\bu}_{l}\huaL[m-1]\oplus\bT^{*,\bu}_{l'}\huaL'[m-1]\oplus\bT^{*,\bu}_{g_0^2}\huaG^2[m]$ is the cochain map given by
\begin{equation}\label{eq:Lmap}
\Lambda_{\widehat{l}}(\xi)=\big(\lambda_{l}^{\beta_\bu,\Phi^1\times\Phi^2,-\omega^1_\bu+\omega^2_\bu}(u,k^1,k^2_b),\ \lambda_{l'}^{\beta'_\bu,\Psi^2\times\Psi^3,-\omega^2_\bu+\omega^3_\bu}(v,k^2_c,k^3), \ \lambda^{\omega^2_\bu,\sharp}_{g^2_0}(k^2_a)\big),   
\end{equation}
where $\xi=(u,v,k^2_a,k^2_b,k^2_c,k^1,k^3)$.
Therefore, since $J_{\widehat{l}}$ and $\Gamma_{\widehat{l}}$ are quasi-isomorphisms, we get that \eqref{eq:lagmaplamb} is a quasi-isomorphism if and only if \eqref{eq:Lmap} is a quasi-isomorphism. But since $(\mathcal{L}_\bu,\Phi^1_\bu\times\Phi^2_\bu,\beta_\bu)$ and $(\mathcal{L}'_\bu,\Psi^2_\bu\times\Psi^3_\bu,\beta'_\bu)$ are shifted lagrangians and $(\huaG^2,\omega^2)$ is $m$-shifted symplectic by hypothesis, the maps  
$$\lambda_{l}^{\beta_\bu,\Phi^1\times\Phi^2,-\omega_\bu^1+\omega_\bu^2},\qquad \lambda_{l'}^{\beta'_\bu,\Psi^2\times\Psi^3,-\omega^2_\bu+\omega^3_\bu},\qquad\text{and}\qquad \lambda_{g_0^2}^{\omega^2_\bu,\sharp}$$
are quasi-ismorphisms. Hence, the five lemma implies $\Lambda_{\widehat{l}}$ is also a quasi-isomorphism and therefore so is the map $\lambda_{\widehat{l}}^{\widehat{\beta}_\bu,\widehat{\Phi},-\omega^1_\bu+\omega^3_\bu}$.
\end{proof}

\section{Symplectic reduction at critical values}\label{sec:red}

Treatments of singular symplectic reduction from the point of view of derived algebraic geometry can be found in \cite{Calaque:14,safronov;hamiltonian-reduction,saf:qua}.  In this section we offer a version for the smooth ($C^\infty$) category in terms of derived Lie groupoids that we hope will appeal to differential geometers.  There exist many different approaches to singular symplectic reduction, such as \cite{SjLe91}, \cite[Ch.~9]{ortega-ratiu;momentum-reduction}, and \cite{barbieri-watts-ziegler;frobenius}, but as far as we know derived symplectic reduction is the only approach that gives meaning to the reduced symplectic form being non-degenerate at the singular points.  Our treatment has much in common with that of \cite{sheshko;derived-symplectic}.  

Our result is stated in the context of hamiltonian actions of quasi-symplectic Lie groupoids in the sense of \cite{moxu}, or $1$-shifted symplectic Lie $1$-groupoids in the language of this paper.  This encompasses the familiar hamiltonian Lie group actions, but also hamiltonian actions of symplectic Lie groupoids in the sense of \cite{MiWe88}, Poisson-Lie group actions in the sense of \cite{lu:mom}, and quasi-hamiltonian actions in the sense of \cite{amm}.  The main result of this section, Thm.~\ref{theorem;reduction}, is an application of the lagrangian intersection theorem, Thm.~\ref{theorem;lagintersection}, and states that the symplectic quotient is a $0$-symplectic derived $1$-groupoid.  The theorem is valid under a linearization hypothesis.  It would be desirable to eliminate this hypothesis, possibly by replacing graded manifolds with cosimplicial manifolds, or groupoids with hypergroupoids, 
but so far we have not been able to make this work.

\subsection{The regular case}

Let $(\Gamma_\bu,\Omega_\bu)$ be a $1$-shifted symplectic Lie $1$-groupoid, referred to as a \emph{quasi-symplectic Lie groupoid} in \cite{moxu}, or as a \emph{twisted presymplectic groupoid} in \cite{BCWZ}.  This means that $\Gamma_\bu$ is an ordinary (underived) Lie groupoid and that $\Omega_1\in\Omega^2(\Gamma_1)$ and $\Omega_0\in\Omega^3(\Gamma_0)$ are forms satisfying 
\[d\Omega_0=0,\quad d\Omega_1=\delta\Omega_0,\quad\delta\Omega_1=0.\]
Let $A=\Alg(\Gamma_\bu)$ be the Lie algebroid of $\Gamma_\bu$ and let $\rho\colon A\to T\Gamma_0$ be its anchor.  For every $c\in\Gamma_0$ we can regard the fibre $A_c$ and the tangent space $T_c\Gamma_0$ as subspaces of $T_c\Gamma_1$.  The nondegeneracy of $\Omega_\bu$ (Def.~\ref{definition;shifted-symplectic}) amounts to the condition that the anchor $\rho$ induces a bijection $$A_c\cap\ker(\Omega_c)\to T_c\Gamma_0\cap\ker(\Omega_c)\qquad \text{for every } c\in \Gamma_0.$$  Following \cite{moxu} we define a \emph{hamiltonian $\Gamma_\bu$-space} to be a triple $(M,\mu,\omega)$, where $M$ is a manifold that carries a $\Gamma_\bu$-action with moment map $\mu\colon M\to\Gamma_0$ and $\omega$ is a $2$-form on $M$ that satisfies the following conditions:
\begin{enumerate}
    \item[\namedlabel{H1}{(H1)}] $d\omega=\mu^*\Omega_0$;
    \item[\namedlabel{H2}{(H2)}] the graph of the action $\{\,(g,m,gm)\mid t(g)=\mu(m)\,\}$ is isotropic in $\Gamma_1\times M\times M^-$;
    \item [\namedlabel{H3}{(H3)}] $\ker(\omega)\cap\ker(T\mu)=0$.
\end{enumerate}
Here $M^-$ denotes $M$ equipped with the $2$-form $-\omega$.

\begin{example}[{\cite[Prop.~3.8]{moxu}}]\label{example;orbit}
Let $c\in\Gamma_0$, let $\orbit=\Gamma_1\cdot c\subseteq \Gamma_0$ be the $\Gamma_\bu$-orbit of $c$, and let $i\colon\orbit\to\Gamma_0$ be the inclusion.  The restriction of $\Omega$ to the target fibre $t^{-1}(\orbit)$ is basic with respect to the action of the isotropy group $\Gamma_c=\Stab(\Gamma_\bu,c)$ and therefore descends to a $2$-form $\Omega_\orbit\in\Omega^2(\orbit)$.    The triple $(\orbit,i,\Omega_\orbit)$ is a hamiltonian $\Gamma_\bu$-space.
\end{example}

Let $(M,\mu,\omega)$ be a hamiltonian $\Gamma_\bu$-space and let $c\in\Gamma_0$.  Thm.~3.18 in \cite{moxu} states that, if $c$ is a regular value of $\mu$ and the isotropy group $\Stab(\Gamma_\bu,c)$ acts freely and properly on the fibre $\mu^{-1}(c)$, then the quotient space $\mu^{-1}(c)/\Gamma_c$ is a symplectic manifold.  If we assume only that $c$ is a regular value (equivalently, that $\mu$ is transverse to the orbit $\orbit$ of $c$), then $\mu^{-1}(c)$ is a submanifold of $M$, but the quotient space is in general not a Hausdorff topological space, let alone a manifold.  Upon inspecting the proof one sees that Xu's theorem extends to this case in the following form.

\begin{theorem}\label{theorem;regular-reduction}
Let $(\Gamma_\bu,\Omega_\bu)$ be a $1$-shifted symplectic Lie $1$-groupoid, $(M,\mu,\omega)$ a hamiltonian $\Gamma_\bu$-space and $\orbit$ an $\Gamma_\bu$-orbit in $\Gamma_0$.  If $\mu$ is transverse to $\orbit$, then the action groupoid 
\begin{equation}\label{eq:redreg}
\begin{tikzcd}
\Gamma_\bu\ltimes\mu^{-1}(\orbit)\ar[r,shift left]\ar[r,shift right]&\mu^{-1}(\orbit),
\end{tikzcd}
\end{equation}
equipped with the $2$-form $\omega|_{\mu^{-1}(\orbit)}-\mu^*\Omega_\orbit$, is a $0$-shifted symplectic Lie groupoid.
\end{theorem}

(For $c\in\orbit$ there is also the $0$-shifted symplectic groupoid $\Gamma_c\ltimes \mu^{-1}(c)\rightrightarrows\mu^{-1}(c)$ with the $2$-form $\omega|_{\mu^{-1}(c)}$, but this is Morita equivalent to \eqref{eq:redreg}, and we prefer to work with the orbit $\orbit$ rather than the point $c$.)
In a departure from standard terminology, for an orbit $\orbit$ in $\Gamma_0$ transverse to $\mu$ we will call the action groupoid $\Gamma_\bu\ltimes\mu^{-1}(\orbit)$ the \emph{reduced space} or \emph{symplectic quotient} of $M$ at the level $\orbit$.  If the $\Gamma_\bu$-action on $\mu^{-1}(\orbit)$ is also proper and free, then this groupoid is symplectically Morita equivalent to the usual symplectic quotient $\mu^{-1}(\orbit)/\Gamma_\bu=\mu^{-1}(c)/\Gamma_c$.

\subsection{The singular case}

In this section we fix
\begin{enumerate}
\item a $1$-shifted symplectic Lie $1$-groupoid $(\Gamma_\bu,\Omega_\bu)$,
\item a hamiltonian $\Gamma_\bu$-space $(M,\mu,\omega)$,
\item a $\Gamma_\bu$-orbit $\orbit$ in $\Gamma_0$ with inclusion map $i_\bu\colon\Gamma_\bu|_\orbit\to\Gamma_\bu$.
\end{enumerate}
We do not assume that $\mu$ is transverse to $\orbit$.  Thm.~\ref{theorem;reduction} states that under a linearization condition on the orbit $\orbit$  the reduced space of $M$ at $\orbit$ can be modelled by a $0$-shifted symplectic derived Lie $1$-groupoid, which is is symplectically Morita equivalent to \eqref{eq:redreg} whenever $\mu$ is transverse to $\orbit$.

\subsubsection{The moment lagrangian}\label{subsubsection;moment}

Weinstein \cite[Lect.~4]{weinstein;lectures-symplectic} observed that a hamiltonian Lie group action gives rise to a natural lagrangian embedding, the \emph{moment lagrangian}.  In the present context this observation can be reinterpreted as follows; see  \cite[Ex.~1.31]{cala:der}: a triple $(M,\mu,\omega)$ consisting of a $\Gamma_\bu$-space $M$ with moment $\mu\colon M\to\Gamma_0$ and a $2$-form $\omega\in\Omega^2(M)$ is a hamiltonian $\Gamma_\bu$-space if and only if the pair $(\mu_\bu,\omega)$ is a $1$-shifted lagrangian morphism, where $\mu_\bu\colon\Gamma_\bu\ltimes M\to\Gamma_\bu$ is the natural lift of $\mu$ to a groupoid morphism.  The argument is as follows.  Conditions \ref{H1} and \ref{H2} say that $(\mu_\bu,\omega)$ is an isotropic morphism, i.e.\ $\mu_\bu^*(\Omega_\bu)=D\omega$.  For the lagrangian condition we observe that, since $\Gamma_\bu\ltimes M$ is an action groupoid, its Lie algebroid is $\Alg(\Gamma_\bu\ltimes M)=\mu^*A$ and the anchor is the infinitesimal action $a$.  It follows that the map $\lambda^{\omega,\mu_\bu,\Omega_\bu}\colon\mathbb{N}^{\mu_\bu}\to\bT^*(\Gamma_\bu\ltimes M)$ is given by
\[
\begin{tikzcd}[ampersand replacement=\&,row sep=large,column sep=large]
\mu^*A\ar[d]\ar[r,"{\bigl(\begin{smallmatrix}a\\\id\end{smallmatrix}\bigr)}"]\&
TM\oplus \mu^*A\ar[r," {(\begin{smallmatrix}T\mu&\mu^*\rho\end{smallmatrix})}"]\ar[d," {(\begin{smallmatrix}\omega^\sharp&0\end{smallmatrix})}"']\&
\mu^*TX\ar[d,"\rho^*\circ\Omega^\sharp\circ Tu"]\\
0\ar[r]\&T^*M\ar[r,"a^*"]\&\mu^*A^*,
\end{tikzcd}
\]
where $u\colon\Gamma_0\to\Gamma_1$ is the unit bisection, and is therefore a quasi-isomorphism.

Combining this fact with Ex.~\ref{example;orbit} we obtain a pair of $1$-shifted lagrangians,
\begin{equation}\label{equation;moment-lagrangian}
\begin{tikzcd}
\Gamma_\bu\ltimes M\ar[r,"\mu_\bu"]&\Gamma_\bu&\Gamma_\bu|_\orbit\ar[l,"i_\bu"'].
\end{tikzcd}
\end{equation}
Symplectic reduction is a matter of intersecting these lagrangians, but first we must perform a fibrant replacement of the orbit $\orbit$.  For an interpretation of the Lagrangian $\Gamma_\bu\ltimes M$ as a physical boundary and the Lagrangian $\Gamma_\bu|_\orbit$ as a symmetry boundary within the framework of continuous symmetry topological field theory, see \cite{HaoXu:26}.

\subsubsection{Linearizing the groupoid}

The normal bundle $\nu_\bu$ of the subgroupoid $\Gamma_\bu|_\orbit$ of
$\Gamma_\bu$ is a VB-groupoid with trivial core,
\begin{equation}\label{eq:norVB}
\begin{tikzcd}[row sep=large]
\nu_1\ar[d,"\pi_1"']\ar[r,shift left]\ar[r,shift
  right]&\nu_0\ar[d,"\pi_0"]\\
\Gamma_1|_\orbit\ar[r,shift left]\ar[r,shift right]&\orbit,
\end{tikzcd}
\end{equation}
where $\nu_0=\nu(\Gamma_0,\orbit)$ is the normal bundle of $\orbit$ in
$\Gamma_0$ and $\nu_1=\nu(\Gamma_1,\Gamma_1|_\orbit)$ the normal
bundle of $\Gamma_1|_\orbit$ in $\Gamma_1$.  We can also think of $\nu_\bu$ as the action groupoid $\Gamma_\bu\ltimes\nu_0$.  Following
\cite{del:riemgpd} we say that $\Gamma_\bu$ is \emph{weakly
linearizable at $\orbit$} if there exist an open groupoid
neighbourhood $U_\bu$ of the zero section $\Gamma_\bu|_\orbit$ of
$\nu_\bu$ and an open embedding of Lie groupoids $\Phi_\bu\colon
U_\bu\hookrightarrow\Gamma_\bu$, called a \emph{linearization} at $\orbit$, such that $\Phi_\bu$ is the identity along the zero section $\Gamma_\bu|_\orbit$ and the normal derivative of $\Phi_\bu$ is the identity at all points of the zero section.

\begin{definition}\label{definition;linear}
The $1$-shifted symplectic Lie $1$-groupoid $(\Gamma_\bu,\Omega_\bu)$
is \emph{weakly symplectically linearizable} at the orbit $\orbit$ if
\begin{enumerate}
\item[\namedlabel{L1}{(L1)}] 
$\Gamma_\bu$ is weakly linearizable at $\orbit$ with linearization given by $\Phi_\bu\colon U_\bu\to \Gamma_\bu$;
\item[\namedlabel{L2}{(L2)}]
there exist $\Omega_\lin\in\Omega^2(U_1)$ and 
$\eta\in\Omega^2(U_0)$ such that
\begin{itemize}
\item
$\Omega_\lin$ is linear and $\eta|_\orbit=\Omega_\orbit$;
\item
$\Phi_1^*\Omega_1=\Omega_\lin+\delta\eta$ and
  $\Phi_0^*\Omega_0=d\eta$.
\end{itemize}
\end{enumerate}
We call the triple $(\Phi_\bu,\Omega_\lin,\eta)$ \emph{linearization
data} at $\orbit$.
\end{definition}

In most examples in \S\S\,\ref{sec:proper}--\ref{subsection;lu} the form $\Omega_\lin$ is equal to the linear part (as defined in \S\,\ref{subsection;linear}) of the form $\Phi_1^*\Omega_1\in\Omega^2(U_1)$, but we do not require this to be the case.

\subsubsection{Replacing the orbit by a derived Lie groupoid}\label{sec:rep-dgpd}

By Prop.~\ref{ex:VBgrpd}, to the VB-groupoid \eqref{eq:norVB}
corresponds the derived Lie groupoid
\begin{equation}\label{eq:De-orbit}
\begin{tikzcd}
\huaZ(\nu_\bu,\pi_\bu^*\nu_\bu,\epsilon_{\nu_\bu})\colon\quad
\huaZ\bigl(\nu_1,\pi_1^*\nu_1,\epsilon_{\nu_1}\bigr)\ar[r,shift
  left]\ar[r,shift right]&
\huaZ\bigl(\nu_0,\pi_0^*\nu_0,\epsilon_{\nu_0}\bigr).
\end{tikzcd}
\end{equation}
It follows from Lemma~\ref{lemma;path}\eqref{item;vector-bundle} that the projection of the normal bundle $\nu_\bu$ induces a Morita morphism,
in fact a levelwise acyclic fibration, $\pi_\bu$ from the groupoid~\eqref{eq:De-orbit} to
$\Gamma_\bu|_\orbit\rightrightarrows\orbit$.  Likewise, the zero section of $\nu_\bu$ defines a levelwise weak equivalence from $\Gamma_\bu|_\orbit\rightrightarrows\orbit$ to~\eqref{eq:De-orbit}.
For any open subgroupoid $U_\bu$ of $\nu_\bu$ we will denote by
$\huaU_\bu$ the open derived Lie subgroupoid of \eqref{eq:De-orbit}
whose body is $U_\bu$.  If $U_\bu$ contains $\Gamma_\bu|_\orbit$, then
$\pi_\bu$ and $j_\bu$ restrict to Morita morphisms which we will denote by the same letter,
\begin{equation}\label{equation;levelwise}
\begin{tikzcd}
\pi_\bu\colon\huaU_\bu\ar[r,two heads,"\simeq"]&\Gamma_\bu|_\orbit,\qquad j_\bu\colon\Gamma_\bu|_\orbit\ar[r,"\simeq"]&\huaU_\bu.
\end{tikzcd}
\end{equation}

We denote by
$\eta_1=(\epsilon_{\nu_1})_*\Omega_\lin\in\Omega^{2,-1}(\huaU_1)$ the
form on the derived manifold $\huaU_1$ corresponding to the linear
differential form $\Omega_\lin\in\Omega^2(U_1)$ as in \S\,\ref{subsection;linear}.  If $(\Phi_\bu,\Omega_\lin,\eta)$ are linearization data for
$\Gamma_\bu$ at the orbit $\orbit$, we let $\eta_0=\eta$, regarded as an element of $\Omega^{2,0}(\huaU_0)$, and we define the $0$-shifted $2$-form $\eta_\bu$ on $\huaU_\bu$ by
\begin{equation}\label{equation;eta}
\eta_\bu=\eta_1+\eta_0.
\end{equation}

\begin{proposition}\label{prop:der-replace}
Let $(\Gamma_\bu,\Omega_\bu)$ be a $1$-shifted symplectic Lie
$1$-groupoid.  Let $\orbit$ be a $\Gamma_\bu$-orbit in~$\Gamma_0$, let
$i_\bu\colon\Gamma_\bu|_\orbit\to \Gamma_\bu$ be the inclusion, and
let $\nu_\bu$ be the normal bundle of $\Gamma_\bu|_\orbit$ in
$\Gamma_\bu$.
\begin{enumerate}
\item\label{item;linearizable}
Suppose $\Gamma_\bu$ is weakly linearizable at $\orbit$ with
linearization map $\Phi_\bu\colon U_\bu\to\Gamma_\bu$.  Let
$\huaU_\bu$ be the open derived Lie subgroupoid of
$\huaZ(\nu_\bu,\pi_\bu^*\nu_\bu,\epsilon_{\nu_\bu})$ whose body is
$U_\bu$.  The map $\Phi_\bu$, viewed as a morphism
$\huaU_\bu\to\Gamma_\bu$, is a strong fibration, the morphism
$\pi_\bu\colon\Gamma_\bu|_\orbit\to\huaU_\bu$ defined in
\eqref{equation;levelwise} is a levelwise acyclic fibration, the morphism $j_\bu\colon\Gamma_\bu|_\orbit\to\huaU_\bu$ defined in
\eqref{equation;levelwise} is a levelwise weak equivalence, and
$i_\bu=\Phi_\bu\circ j_\bu$.
\[
\begin{tikzcd}
&\huaU_\bu\ar[dr,two heads,"\Phi_\bu"]&\\
\Gamma_\bu|_\orbit\ar[ur,"j_\bu","\simeq"']\ar[rr,"i_\bu"']&&\Gamma_\bu.
\end{tikzcd}
\]
\item\label{item;symplectically-linearizable}
Suppose $\Gamma_\bu$ is weakly symplectically linearizable at $\orbit$
with linearization data $(\Phi_\bu,\Omega_\lin,\eta)$.  Let $\eta_\bu$
be the $1$-shifted $2$-form defined in~\eqref{equation;eta}.  The pair
$(\Phi_\bu,\eta_\bu)$ is a lagrangian morphism
$\huaU_\bu\to\Gamma_\bu$.  The morphism $j_\bu$ induces an equivalence
of lagrangian morphisms between $i_\bu$ and~$\Phi_\bu$.
\end{enumerate}
\end{proposition}

\begin{proof}
\eqref{item;linearizable}~follows immediately from the definition of a
strong fibration (Def.~\ref{definition;strong}) and from the preceding
discussion.

\eqref{item;symplectically-linearizable}~For the groupoid $U_\bu$ the
conditions \ref{L2} imply
\[
\delta\Omega_\lin=\Phi_2^*\delta\Omega_1=0\quad\text{and}\quad
d\Omega_\lin=\Phi_1^*d\Omega_1-\Phi_1^*\delta\Omega_0=0.
\]
Hence also $\delta\eta_1=d\eta_1=0$, and therefore
\begin{equation*}
D\eta_\bu=(\delta+\Lie_{Q_1}+d)\eta_1+(\delta+\Lie_{Q_0}+d)\eta_0=
\Omega_\lin+\delta\eta+d\eta=\Phi_\bu^*\Omega_\bu,
\end{equation*}
where $Q_i=\iota_{\epsilon_{U_i}}$ and where we used
$\Omega_\lin=\Lie_{Q_1}\eta_1$ (see~\eqref{equation;primitive}).  This shows that
$\eta_\bu$ is an isotropic structure for the morphism $\Phi_\bu$.  On
the other hand
\begin{equation}\label{eq:lagequiv}
    j_\bu^*\eta_\bu=j_0^*\eta=\Omega_\orbit,
\end{equation}
so the two isotropic structures are equivalent.  The lagrangian
condition for $\huaU_\bu$ follows from the fact that $j_\bu$ is a
levelwise weak equivalence: using that the classical locus of
$\huaU_0$ is $\pi_0(\huaU_0)=j_0(\orbit)$, we deduce from
Lemma~\ref{lem:MEtcom} that for all $c\in\orbit$ the vertical arrows in
the diagram
\[
\begin{tikzcd}[row sep=large,column sep=huge]
\mathbb{N}_{\Phi_\bu}=\bT_{j_0(c)}\huaU_\bu\oplus\bT_{i_0(c)}\Gamma_\bu
\ar[r,"\lambda_{j_0(c)}^{\eta_\bu,\Phi_\bu,\Omega_\bu}"]
&\bT^*_{j_0(c)}\huaU_\bu\ar[d,"\bT^*j_\bu"]\\
\mathbb{N}_{j_\bu}=\bT_c\Gamma_\bu|_\orbit\oplus\bT_{i_0(c)}\Gamma_\bu
\ar[r,"\lambda_c^{\Omega_\orbit,i_\bu,\Omega_\bu}"]\ar[u,"\bT
  j_\bu\oplus\id"] &\bT^*_c\Gamma_\bu|_\orbit
\end{tikzcd}
\]
are quasi-isomorphisms.  Since
$\lambda_c^{\Omega_\orbit,i_\bu,\Omega_\bu}$ is a quasi-isomorphism,
so is $\lambda_{j_0(c)}^{\eta_\bu,\Phi_\bu,\Omega_\bu}$.
\end{proof}

Replacing $\Gamma_\bu|_\orbit$ with $\huaU_\bu$ in
\eqref{equation;moment-lagrangian} yields the diagram
\begin{equation}\label{equation;replacement}
\begin{tikzcd}
(\Gamma_\bu\ltimes M,0)\ar[r,"\mu_\bu"]&(\Gamma_\bu,\Omega_\bu)&
  \bigl(\huaU_\bu,\eta_\bu\bigr)\ar[l,"\Phi_\bu"']
\end{tikzcd}
\end{equation}
where the middle column is a $1$-shifted symplectic Lie groupoid and
the two sides are shifted lagrangian morphisms with $\Phi_\bu$ being a
strong fibration.

\subsubsection{Computing the reduced space}\label{sec:redspace}

We can now form the fibred product $\huaG_\bu=(\Gamma_\bu\ltimes M)\times_{\Gamma_\bu}\huaU_\bu$, which is a derived Lie groupoid with
\begin{equation}\label{equation;symplectic-quotient}
\huaG_0=\huaZ(M\times_{\Gamma_0}U_0,E_0,\sigma_0),\qquad\huaG_1=\huaZ(M\times_{\Gamma_0}U_1,E_1,\sigma_1),
\end{equation}
where
\begin{alignat*}{10}
M\times_{\Gamma_0}U_0&=\{\,(x,u_0)\in M\times U_0\mid\mu(x)=\Phi_0(u_0)\,\},&
E_0&=\pr_2^*\pi_0^*\nu_0,&
\sigma_0(x,u_0)&=u_0,
\\
M\times_{\Gamma_0}U_1&=\{\,(x,u_1)\in M\times U_1\mid\mu(x)=s_\Gamma(\Phi_1(u_1))\,\},\qquad&
E_1&=\pr_2^*\pi_1^*\nu_1,\qquad&
\sigma_1(x,u_1)&=u_1.
\end{alignat*}
This derived groupoid is equipped with a $0$-shifted symplectic form $\omega^\red_\bu=\omega^\red_{2,-1,1}+\omega^\red_{2,0,0}$ with
\[
\omega^\red_{2,-1,1}=-\pr_2^*\eta_1\in\Omega^{2,-1}(\huaG_1),\qquad
\omega^\red_{2,0,0}=\pr_1^*\omega-\pr_2^*\eta_0\in\Omega^{2,0}(\huaG_0),
\]
where $\pr_1\colon\huaG_\bu\to\Gamma_\bu\ltimes M$ and $\pr_2\colon\huaG_\bu\to\huaU_\bu$ are the natural projections, and where $\eta_\bu=\eta_0+\eta_1$ is as in~\eqref{equation;eta}.
We call the $0$-shifted symplectic derived Lie groupoid $(\huaG_\bu,\omega^\red_\bu)$ the \emph{reduced space} or \emph{symplectic quotient} of $M$ by the $\Gamma_\bu$-action.  We denote $\huaG_\bu$ by $M\qu[\orbit]\Gamma_\bu$ or $M\qu\Gamma_\bu$.

\begin{theorem}\label{theorem;reduction}
Let $(\Gamma_\bu,\Omega_\bu)$ be a $1$-shifted symplectic Lie $1$-groupoid and let $(M,\mu,\omega)$ be a hamiltonian $\Gamma_\bu$-space.
Let $\orbit$ be a $\Gamma_\bu$-orbit in $\Gamma_0$.  Suppose $\Gamma_\bu$ is weakly symplectically linearizable at $\orbit$ with linearization data $(\Phi_\bu,\Omega_\lin,\eta)$.
\begin{enumerate}
\item\label{item;reduction} The symplectic quotient $(M\qu[\orbit]\Gamma_\bu,\omega^\red_\bu)$ is a quasi-smooth $0$-shifted symplectic derived Lie groupoid.  It is the homotopy pullback of the lagrangian morphisms $\mu_\bu\colon\Gamma_\bu\ltimes M\to\Gamma_\bu$ and $i_\bu\colon\Gamma_\bu|_\orbit\to\Gamma_\bu$.
\item\label{item;coarse} The coarse quotient of $M\qu[\orbit]\Gamma_\bu$ is the topological space $\mu^{-1}(\orbit)/\Gamma_\bu$.  Let $x\in\mu^{-1}(\orbit)$.  The stabilizer of $M\qu[\orbit]\Gamma_\bu$ at $[x]$ is $\Stab(\Gamma_\bu\ltimes M,x)$.
\item\label{item;tangent} Let $A=\Alg(\Gamma_\bu)$ be the Lie algebroid of $\Gamma_\bu$ and $\rho\colon A\to T\Gamma_0$ its anchor.  Let $x\in\mu^{-1}(\orbit)$ and $c=\mu(x)$.  Let $a\colon A_c\to T_xM$ be the infinitesimal $\Gamma_\bu$-action and $E_c=T_c\Gamma_0/T_c\orbit$.  Let $f\colon A_c\to E_c^*$ be the map $T^*u\circ\Omega^\sharp\circ\rho$.  The symplectic pairing on $M\qu[\orbit]\Gamma_\bu$ at $[x]$ is
\[
\begin{tikzcd}
\bT_x(M\qu\Gamma_\bu)\ar[d,"(\omega^\red)_\bu^\sharp"']\colon&A_c\ar[d,"f"']\ar[r,"a"]&T_xM\ar[r,"-T\mu"]\ar[d,"\omega^\sharp"']&E_c\ar[d,"-f^*"]\\
\bT^*_x(M\qu\Gamma_\bu)\colon&E_c^*\ar[r,"-T^*\mu"]&T_x^*M\ar[r,"a^*"]&A^*_c.
\end{tikzcd}
\]
\item\label{item;regular-singular} If $\mu$ is transverse to $\orbit$, then $M\qu[\orbit]\Gamma$ is symplectically Morita equivalent to the $0$-shifted symplectic Lie groupoid $\Gamma|_\orbit\ltimes\mu^{-1}(\orbit)$ of Thm.~\textup{\ref{theorem;regular-reduction}}.
\end{enumerate}
\end{theorem}

\begin{proof}
\eqref{item;reduction}~This follows from Thm.~\ref{theorem;lagintersection} and Prop.~\ref{prop:der-replace}.  Since $\Phi_\bu$ is a strong fibration, replacing $\orbit$ by $\huaU_\bu$ is a fibrant replacement as in Rem.~\ref{rmk:ptvv-thm5}. Therefore $(M\qu\Gamma_\bu,\omega^\red_\bu)$ is a homotopy pullback as in \S\,\ref{subsection;pull}.

\eqref{item;coarse}~The classical locus of $\huaG_0$ is $\pi_0(\huaG_0)=\mu^{-1}(\orbit)$.  The restriction of $\huaG_1$ to the classical locus is the (topological) action groupoid $\Gamma_\bu\ltimes\mu^{-1}(\orbit)$, so the quotient space $\pi_0(\huaG_0)/\huaG_1$ is $\mu^{-1}(\orbit)/\Gamma$.  Since $\Gamma$ and $\huaG_\bu$ are $1$-groupoids, the stabilizer of $x$ is the group of all $\gamma\in\Gamma$ with $s(\gamma)=\mu(x)$ and $\gamma\cdot x=x$.

\eqref{item;tangent}~Since $\huaG_\bu=M\qu\Gamma_\bu$ is a quasi-smooth derived groupoid, its tangent complex at a classical point $x\in\huaG_0$ is of the form $\Alg(\huaG_\bu)_x\to T_xM\to E_{0,x}$, where the vector bundle $E_0$ is as in \eqref{equation;symplectic-quotient}.  The body of $\huaG_\bu$ is isomorphic to an open Lie subgroupoid of $\Gamma_\bu\ltimes M$, so the fibre of its Lie algebroid at $x$ is $A_{\mu(x)}=A_c$.  The fibre $E_{0,x}$ is equal to the fibre of the normal bundle of $\orbit$ in $\Gamma_0$ at $\mu(x)=c$.

\eqref{item;regular-singular}~Suppose $\mu$ is transverse to $\orbit$.  Then the fibred product $(\Gamma_\bu\ltimes M)\times_{\Gamma_\bu}\Gamma_\bu|_\orbit$ is well-defined (Prop.~\ref{proposition;derived-fibred-product}) and is isomorphic to the action groupoid $\Gamma_\bu\ltimes\mu^{-1}(\orbit)$.  Pulling back the morphisms $\pi_\bu\colon\huaU_\bu\to\Gamma_\bu\ltimes\orbit$ and $j_\bu\colon\Gamma_\bu\ltimes\orbit\to\huaU_\bu$ defined in~\eqref{equation;levelwise} via $\mu$ we obtain morphisms
\[
\begin{tikzcd}
\tilde\pi_\bu\colon\huaG_\bu\ar[r]&\Gamma_\bu\ltimes\mu^{-1}(\orbit),\qquad \tilde{\jmath}_\bu\colon\Gamma_\bu\ltimes\mu^{-1}(\orbit)\ar[r]&\huaU_\bu.
\end{tikzcd}
\]
Because $\pi_\bu$  is a levelwise acyclic fibration (Prop.~\ref{prop:der-replace}\eqref{item;linearizable}), so is $\tilde\pi_\bu$, and therefore the right inverse $\tilde{\jmath}_\bu$ of $\tilde\pi_\bu$ is a levelwise weak equivalence.  In particular both $\tilde\pi_\bu$ and $\tilde{\jmath}_\bu$ are Morita equivalences.  The form $\omega^\red_\bu$ pulls back to 
\[
\tilde{\jmath}^*\omega^\red=\omega|_{\mu^{-1}(\orbit)}-\mu^*(\eta|_\orbit)=\omega|_{\mu^{-1}(\orbit)}-\mu^*\omega_\orbit.\qedhere
\]
\end{proof}

In the remaining sections we will look into some interesting cases where the linearization hypothesis of Thm.~\ref{theorem;reduction} is satisfied.

\subsection{Symplectic linearization of proper groupoids}\label{sec:proper}

Let $(\Gamma_\bu,\Omega_\bu)$ be a proper $1$-shifted symplectic
groupoid.  A linearization theorem for such groupoids can be found in
\cite[\S\,2.2]{pmct3} (see also \cite[Prop.~8.7]{pmct1} for the
untwisted case $\Omega_0=0$).  Let $\orbit$ be a $\Gamma_\bu$-orbit in
$\Gamma_0$ and let $c\in\orbit$.  The stabilizer
$G=\Gamma_c=\Stab(\Gamma_\bu,c)$ of $c$ is a compact Lie group and the
source fibre $P=s^{-1}(c)$ is a principal $G$-bundle over $\orbit$,
the projection being the target map $t\colon P\to\orbit$.  The group
$G$ acts freely by groupoid automorphisms on
\[
\begin{tikzcd}
P\times P\times\g^*\ar[r,shift left]\ar[r,shift right]&P\times\g^*,
\end{tikzcd}
\]
the product of the pair groupoid $P\times P\rightrightarrows P$ with
the identity groupoid $\g^*\rightrightarrows\g^*$, where $\g$ is the
Lie algebra of $G$.  The quotient groupoid
\begin{equation}\label{pair-id-quotient}
\begin{tikzcd}
(P\times P\times\g^*)/G\ar[r,shift left]\ar[r,shift right]&
  (P\times\g^*)/G
\end{tikzcd}
\end{equation}
is isomorphic to the normal bundle $\nu_\bu$ of the subgroupoid
$\Gamma_\bu|_\orbit=\Gamma_\bu\ltimes\orbit$ of $\Gamma_\bu$.  Choose a principal $G$-connection $\theta\in\Omega^1(P,\g)^G$.  The
groupoid~\eqref{pair-id-quotient} carries a multiplicative $2$-form
$\Omega^\theta\in\Omega^2((P\times P\times\g^*)/G)$ induced by
the $G$-basic $2$-form
\begin{equation}\label{basform}
\widetilde{\Omega}^\theta=p^*_1\Omega_\orbit-
p^*_2\Omega_\orbit-d\inner{\theta_1,\pr_3}+
d\inner{\theta_2,\pr_3}\in\Omega^2(P\times P\times\g^*),
\end{equation}
where $\Omega_\orbit\in\Omega^2(\orbit)$ is the twisted presymplectic
form on the orbit (Ex.~\ref{example;orbit}), $p_i\colon P\times
P\times\g^*\to\orbit$ denotes the composition of the projection to the
$i$-th factor with the bundle projection $t\colon P\to\orbit$,
$\pr_3\colon P\times P\times\g^*\to\g^*$ denotes the projection to the
third factor, while $\inner{\theta_i,\pr_3}$ denotes the $1$-form
\[
\inner{\theta_i,\pr_3}_{(y_1,y_2,\xi)}(v_1,v_2,w)=
\inner{\theta_{y_i}(v_i),\xi}.
\]
Let $U_0$ be a star-shaped open neighbourhood of the zero section of
the vector bundle $(P\times\g^*)/G\to\orbit$, let $p\colon
U_0\to\orbit$ be the projection onto the zero section, and let
$\Phi_0\colon U_0\to M$ be a tubular neighbourhood embedding.  The
$3$-form $\gamma=\Phi_0^*\phi-dp^*\Omega_\orbit$ on $U_0$ satisfies
$d\gamma=0$ and $\gamma|_\orbit=0$ and is therefore exact.  Let $\beta$ be the primitive of $\gamma$ obtained by integrating $\gamma$
over the fibres of $p$.  Then $\beta$ satisfies
\begin{equation}\label{equation;model-primitive}
\beta|_\orbit=0,\qquad d\beta=\Phi_0^*\phi-dp^*\Omega_\orbit.
\end{equation}

\begin{theorem}[{\cite[Thm.~2.3, Thm.~2.8, Rem.~2.9]{pmct3}}]\label{theorem;proper}
Let $\orbit$ be an orbit of a proper $1$-shifted symplectic groupoid
$(\Gamma_\bu,\Omega_\bu)$.  Choose $c\in\orbit$ and a connection
$\theta$ on the principal $\Gamma_c$-bundle $s^{-1}(c)\to\orbit$.
There exist an open groupoid neighbourhood $\huaU_\bu$ of the zero
section $\Gamma_\bu\ltimes\orbit$ of the groupoid~\eqref{pair-id-quotient} and a weak
linearization $\Phi_\bu\colon\huaU_\bu\to\Gamma_\bu$ such that
$\Phi_1^*\Omega_1=\Omega^\theta+\delta\beta$.   Here
$\Omega^\theta\in\Omega^2((P\times P\times\g^*)/G)$ is the form
induced by~\eqref{basform} and $\beta\in\Omega^2(U_0)$ is a form
satisfying~\eqref{equation;model-primitive}, where $\Phi_0\colon U_0\to M$
is the base map of\/ $\Phi$ and $p\colon U_0\to\orbit$ is the natural
projection.
\end{theorem}

The weak symplectic linearization result is as follows.

\begin{corollary}\label{corollary;linearization}
In the setting of Thm.~\textup{\ref{theorem;proper}} let $\Omega_\lin=\Omega^\theta-\delta p^*\Omega_\orbit$ and $\eta=\beta+p^*\Omega_\orbit$.  Then
$(\Gamma_\bu,\Omega_\bu)$ is weakly symplectically linearizable at the orbit $\orbit$ with
linearization data $(\Phi_\bu,\Omega_\lin,\eta)$.
\end{corollary}

\begin{proof}
It follows from~\eqref{basform} that $\Omega_\lin$ is linear.  The restriction of $\eta$ to the orbit is $\beta|_\orbit+(p^*\Omega_\orbit)|_\orbit=\Omega_\orbit$.  Thm.~\ref{theorem;proper} yields
\[
\Phi_1^*\Omega_1=\Omega^\theta+\delta\beta=\Omega^\theta+\delta\eta-\delta p^*\Omega_\orbit=\Omega_\lin+\delta\eta.
\]
Finally $d\eta=d\beta+dp^*\Omega_\orbit=\Phi_0^*\Omega_0$, so~\ref{L2} is satisfied.
\end{proof}

\subsection{Marsden-Weinstein reduction at the zero level}\label{sec:hamred}

If a symplectic groupoid is linear to begin with, we do not need a
linearization theorem and can apply the reduction theorem,
Thm.~\ref{theorem;reduction}, directly.  This is notably the case for
Marsden-Weinstein reduction at the zero orbit.  A treatment of this
case similar to ours can be found in~\cite{sheshko;derived-symplectic}.

Let $G$ be a Lie group.  We trivialize the cotangent bundle $T^*G\cong
G\times\g^*$ by left translations.  Under this identification the
coadjoint action groupoid is a symplectic groupoid
$(G\ltimes\g^*\rightrightarrows\g^*,\omega_\can)$, and a hamiltonian
$G\ltimes\g^*$-space is nothing other than a symplectic manifold
equipped with a hamiltonian $G$-action $(M,\mu,\omega)$ in the usual
sense.  Let $\orbit=\{0\}$ be the zero orbit.  Then
$(G\ltimes\g^*)|_\orbit=G\ltimes\{0\}=G$ and the normal bundle of $G$
is $G\ltimes\g^*$ itself.  The symplectic form $\omega_\can$ is linear
and $\Omega_\orbit=0$, so in Thm.~\ref{theorem;reduction} we can take
$\Omega_\lin=\omega_\can$, $\eta=0$, and $U_\bu=G\ltimes\g^*$, which
gives the following result.

\begin{theorem}\label{theorem;marsden-weinstein}
Let $G$ be a Lie group and $(M,\mu,\omega)$ a hamiltonian $G$-space.
The symplectic quotient $M\qu[0]G$ at the zero orbit is the
quasi-smooth action groupoid
\[G\ltimes\huaZ(M,M\times\g^*,\id\times\mu),\]
the reduced symplectic form has two components
\[
\omega^\red_{2,-1,1}=-\omega_\can\in\Omega^{2,-1,1}(M\qu[0]G)\quad\text{and}\quad
\omega^\red_{2,0,0}=\omega\in\Omega^{2,0,0}(M\qu[0]G),
\]
and the symplectic pairing at a point $x\in\mu^{-1}(0)$ is
\[
\begin{tikzcd}
\bT_x(M\qu G)\ar[d,"(\omega^\red)_\bu^\sharp"']\colon&\g\ar[d,"\id"']\ar[r,"a"]&T_xM\ar[r,"-T\mu"]\ar[d,"\omega^\sharp"']&\g^*\ar[d,"-\id"]\\
\bT^*_x(M\qu G)\colon&\g\ar[r,"-T^*\mu"]&T_x^*M\ar[r,"a^*"]&\g^*.
\end{tikzcd}
\]
\end{theorem}

\subsubsection{Reduction by using the diagonal}

An alternative way of computing the fibred product of the $1$-shifted lagrangian morphisms
\[
\begin{tikzcd}
G\ltimes M\ar[r,"\id\times\mu"]&G\ltimes\g^*&G\ar[l,"\id\times i_0"']
\end{tikzcd}
\]
is by using the diagonal
\[
\begin{tikzcd}[column sep=6em]
(G\ltimes M)\times(G\ltimes\{0\})\ar[r,"(\id\times\mu)\times(\id\times
    i_0)"]& (G\ltimes \g^*)\times (G\ltimes
  \g^*)&G\ltimes\g^*\ar[l,"\Delta"'].
\end{tikzcd}
\]
The diagonal map for smooth manifolds can be resolved by means of the
path space construction of Lemma \ref{lemma;path}.  However, this
construction is not functorial and there is in general no way of
producing a groupoid structure on $\huaP(G\ltimes\g^*)$ or a groupoid
morphism
\[
p\colon\huaP(G\ltimes\g^*)\longto(G\ltimes\g^*)\times(G\ltimes\g^*).
\]
Instead we propose the following ad hoc ``half-path object''
construction.  Letting $G$ act on $\g^*$ by the coadjoint action and
diagonally on $\g^*\times\g^*$ we obtain a $G$-action on the path
space $\huaP(\g^*)=\huaZ(\g^*\times\g^*,\g^*,\epsilon)$, and we form
the derived Lie groupoid $G\ltimes\huaP(\g^*)$.  The map
\[
p_{\g^*}\colon\huaP(\g^*)\longto\g^*\times\g^*
\]
defined by $p_{\g^*}(\xi_1,\xi_2)=(\xi_1,\xi_1-\xi_2)$ is equivariant
and therefore defines a groupoid morphism
\[
G\ltimes\huaP(\g^*)\longto(G\ltimes \g^*)\times(G\ltimes\g^*).
\]
One can show that this morphism is lagrangian.  It is not a fibration,
but it is transverse to $(\id\times\mu)\times(\id\times i_0)$, so
Thm.~\ref{theorem;lagintersection} tells us that we get a
$0$-shifted symplectic derived Lie groupoid
\[
(G\ltimes M)\times
G\ltimes\{0\}\times_{(G\ltimes\g^*)\times(G\ltimes\g^*)}G\ltimes\huaP(\g^*),
\]
which one can verify is isomorphic to $M\qu[0]G$.

\subsubsection{Differentiation and the BFV model}

Let $(M,\mu,\omega)$ be a hamiltonian $G$-space.  The
\emph{cohomological reduced space}, or \emph{BFV model}, or \emph{BRS
model} is the $\Z$-graded manifold
\[T^*\g[1]\times M=\g^*[-1]\times\g[1]\times M\]
equipped with the $0$-shifted symplectic form $\omega_\can+\omega$ and
the cohomological vector field
\[
Q=\frac12c_{\alpha\beta}^\gamma
\xi^\alpha\xi^\beta\pardif{}{\xi^\gamma}+
a^i_\alpha\xi^\alpha\pardif{}{x^i}+\mu^i_\alpha\pardif{}{p_\alpha};
\]
see e.g.~\cite{ks:brs}.  Here $\xi^\alpha$ are coordinates on $\g[1]$
with dual coordinates $p_\alpha$, $x^i$ are coordinates on $M$, $a^i$
are the components of the action, and $\mu^i$ are the components of
$\mu$.  This model differs from the derived groupoid $M\qu[0]G$ only
in that the latter has the Lie group in degree $1$ instead of the Lie
algebra.  We suspect that the BFV model is in some sense the Lie
algebroid of $M\qu[0]G$ and that its symplectic form is the image of
$\omega_\bu^\red$ under an appropriate van Est type map.

\subsection{Quasi-hamiltonian reduction at the unit}\label{sec:quamat1}

Let $G$ be a \emph{quadratic} Lie group, i.e.\ a Lie group whose Lie algebra $\g$ is equipped with an $\Ad$-invariant non-degenerate symmetric bilinear form $\langle\cdot,\cdot\rangle\colon \g\times\g\to\mathbb{R}$.  It was shown in \cite{amm,moxu} that the conjugation action groupoid $G\ltimes G\rightrightarrows G$ carries a $1$-shifted symplectic structure $\Omega_\bu$ given by
\begin{align*}
\Omega_{1,(g,x)}&=\frac{1}{2}\big(\langle\Ad_x \pr^*_1 \theta^l, \pr^*_1 \theta^l\rangle + \langle \pr^*_1 \theta^l , \pr^*_2(\theta^l + \theta^r)\rangle\big)\in\Omega^2(G\times G),\\
\Omega_0&=-\frac{1}{12}\biginner{\theta^l,[\theta^l,\theta^l]}\in\Omega^3(G),
\end{align*}
where $\theta^l,\theta^r\in\Omega^1(G,\g)$ denote respectively the left and right Maurer-Cartan $1$-forms on $G$.  As explained in \cite{moxu}, the hamiltonian $G\ltimes G$-spaces are exactly the quasi-hamiltonian $G$-spaces $(M,\mu,\omega)$ introduced in \cite{amm}.  It was pointed out in  \cite{amm} that $G\ltimes G$ is weakly symplectically linearizable at the unit element $1\in G$.  The normal bundle of $(G\ltimes G)|_{\{1\}}=G\ltimes\{1\}$ is $G\ltimes\g$, the adjoint action groupoid, and the linearization map $\Phi\colon G\ltimes\g\to G\ltimes G$ is $\Phi=\id\times\exp$.  Let $\varphi\colon\g\to\g^*$ be the equivariant isomorphism induced by the invariant bilinear form on $\g$.  The linear part of $\Omega_1$ is the $2$-form
\[\Omega_\lin=(\id\times\varphi)^*\omega_\can\in\Omega^2(G\ltimes\g),\]
where $\omega_\can$ is the canonical symplectic form on $T^*G$.  Let $\varpi\in \Omega^2(\g)$ be the image of $\exp^*H$ under the standard homotopy operator $\Omega^\bu(\g)\to \Omega^{\bu-1}(\g)$ for the linear retraction onto the origin, i.e.
\begin{equation}\label{eq:varpi}
\varpi=\frac{1}{2}\int^1_0\Biginner{\exp_s^*\theta^r, \pardif{}{s}\exp^*_s\theta^r}\,ds\in\Omega^2(\g).
\end{equation}
The results of \cite[Lemma~3.3, Prop.~3.4]{amm} say that, in our language, $(\Phi,\Omega_\lin,\varpi)$ is a weak symplectic linearization of $(G\ltimes G,\Omega_\bu)$.  We may take $U_\bu=G\ltimes\g$ and Thm.~\ref{theorem;reduction} yields the following result.

\begin{theorem}\label{theorem;quasi}
Let $G$ be a quadratic Lie group and $(M,\mu,\omega)$ a quasi-hamiltonian $G$-space.  The symplectic quotient $M\qu[1]G$ at the unit element is the quasi-smooth action groupoid $G\ltimes\huaZ(\widehat{M},\widehat{M}\times\g,s)$, where $\widehat{M}$ is the fibred product 
\[
M\times_G\g=\{\,(x,\xi)\in M\times\g\mid\mu(x)=\exp(\xi)\,\},
\]
$G$ acts on $\widehat{M}$ by the diagonal action, and $s$ is the section $s(x,\xi)=(x,\xi,\xi)$.  The reduced symplectic form is $\omega^\red_\bu$ with 
\[
\omega^\red_{2,-1,1}=-(\id\times\varphi)^*\omega_\can\in\Omega^{2,-1,1}(M\qu[1]G),\qquad\omega^\red_{2,0,0}=\pr_1^*\omega-\pr_2^*\varpi\in\Omega^{2,0,0}(M\qu[1]G),
\]
and the symplectic pairing at a point $x\in\mu^{-1}(1)$ is
\[
\begin{tikzcd}[column sep=large]
\bT_x(M\qu[1]G)\ar[d,"(\omega^\red)_\bu^\sharp"']\colon&\g\ar[d,"\id"']\ar[r,"a"]&T_xM\ar[r,"-T\mu"]\ar[d,"\omega^\sharp"']&\g\ar[d,"-\id"]\\
\bT^*_x(M\qu[1]G)\colon&\g\ar[r,"-T^*\mu\circ\varphi"]&T_x^*M\ar[r,"\varphi^{-1}\circ a^*"]&\g.
\end{tikzcd}
\]
\end{theorem}

\subsection{Lu reduction for Poisson-Lie group actions}\label{subsection;lu}

Let $(\mathfrak{d},\g,\g^*)$ be a Manin triple.  Let $D$ be a Lie group integrating $\mathfrak{d}$ and let $(G,\pi)$ and $(G^*,\pi^*)$ be Poisson Lie groups integrating $\g$, resp.\ $\g^*$, along with morphisms $\phi_G\colon G\to D$ and  $\phi_{G^*}\colon G^*\to D$ integrating the inclusions $\g\to\mathfrak{d}$ and $\g^*\to\mathfrak{d}$.  We call the triple $(D,G,G^*)$ a \emph{Manin triple of Lie groups} integrating $(\mathfrak{d},\g,\g^*)$.  We say the triple $(D,G,G^*)$ is \emph{complete} if the map
\begin{equation}\label{eq:complete}
\psi\colon G^*\times G\longto D,\qquad\psi(z,g)=\phi_{G^*}(z)\phi_G(g)^{-1}
\end{equation}
is a diffeomorphism.  If the triple is complete we obtain \emph{dressing actions} of $G^*$ on $G$ and of $G$ on~$G^*$, 
\[
G\times G^*\longto G,\quad(g,z)\longmapsto g^z,\qquad G\times G^*\longto G^*,\quad(g,z)\longmapsto{}^gz,
\]
determined by the equation
\[
gz={}^gz g^z\qquad\text{for all $(g,z)\in G \times G^*$}.
\]
It was shown in \cite{Lu-we1} (see also \cite[Thm.~4.2]{lu:the}) that the product $G\ltimes G^*$ carries a symplectic form $\Omega$ which makes the dressing action groupoid
\[
\begin{tikzcd}
G\ltimes G^*\ar[r,shift left]\ar[r,shift right]&G^*
\end{tikzcd}
\]
a symplectic groupoid that integrates the Poisson structure $\pi^*$ on $G^*$.  (See \cite[Thm.~3]{am:sympl} or \cite[Ex.~1.2]{libland-severa1} for a formula for $\Omega$.)  As explained in \cite[Ex.~3.10]{moxu}, a hamiltonian $G\ltimes G^*$-space is the same thing as a symplectic manifold $M$ equipped with a Poisson $G$-action $G\times M\to M$ that admits a $G^*$-valued moment map $\mu\colon M\to G^*$ in the sense of \cite{lu:mom}.  Let us consider reduction of $M$ at the identity element $1\in G^*$.  
The normal bundle of $(G\ltimes G^*)|_{\{1\}}=G\ltimes\{1\}$ is $G\times\g^*$, the coadjoint action groupoid of $G$.  But the exponential map of $G^*$ is not equivariant with respect to the dressing action of $G$ and therefore does not give us a weak linearization of the dressing action groupoid.  A replacement in some cases was found in \cite{am:linpl} and is as follows.

Let us say that $\g$ is a \emph{coboundary Lie bialgebra} if the Manin triple admits a $G$-equivariant splitting $j\colon\g^*\to\mathfrak{d}$ (which is usually different from the inclusion map $\g^*\to\mathfrak{d}$).  If $\g$ is a coboundary Lie bialgebra we define the map $\Exp\colon \g^*\to G^*$ to be the following composition
\[
\begin{tikzcd}
\Exp\colon\quad\g^*\ar[r,"j"]&\mathfrak{d}\ar[r,"\exp"]&D\ar[r,"\psi^{-1}"]&G^*\times G\ar[r,"\pr_{G^*}"]&G^*.
\end{tikzcd}
\]
This map is $G$-equivariant (\cite[\S 3.3]{am:linpl}) and satisfies $T_0\Exp=\id_{\g^*}$, and therefore provides a linearization map $\Phi=\id\times\Exp\colon G\ltimes\g^*\to G\ltimes G^*$ at $G\ltimes\{1\}$.  Let
\[
\Omega_\lin=\omega_\can\in\Omega^2(G\ltimes \g^*)\quad\text{and}\quad \eta=j^*\varpi_D-j^*\exp^*(\psi^{-1})^*\frac{1}{2}\langle \theta^l_{G^*},\theta^l_G\rangle\in\Omega^2(\g^*),
\]
where $\varpi_D\in\Omega^2(\mathfrak{d})$ is the form given in \eqref{eq:varpi} for the group $D$.   Prop.~3.6 in \cite{am:linpl} says that the dressing action groupoid $G\ltimes G^*$ is weakly symplectically linearizable at the orbit $G\ltimes\{1\}$ with linearization data $(\Phi,\Omega_\lin,\eta)$.  Taking $U_\bu=G\ltimes\g$ we deduce from Thm.~\ref{theorem;reduction} the following result.

\begin{theorem}\label{theorem;lu}
Let $\g$ be a coboundary Lie bialgebra and suppose that the Manin triple $(\mathfrak{d},\g,\g^*)$ is integrable to a complete Manin triple of Lie groups $(D,G,G^*)$.  Let $(M,\omega)$ be a symplectic manifold equipped with a Poisson $G$-action and a moment map $\mu\colon M\to G^*$.  The symplectic quotient $M\qu[1]G$ at the unit element $1\in G^*$ is the quasi-smooth action groupoid $G\ltimes\huaZ(\widehat{M},\widehat{M}\times\g,s)$, where $\widehat{M}$ is the fibred product 
\[
M\times_{G^*}\g^*=\{\,(x,\alpha)\in M\times\g^*\mid\mu(x)=\Exp(\alpha)\,\},
\]
$G$ acts on $\widehat{M}$ by the diagonal action, and $s$ is the section $s(x,\alpha)=(x,\alpha,\alpha)$.  The reduced symplectic form is $\omega^\red_\bu$ with 
\[
\omega^\red_{2,-1,1}=-\omega_\can\in\Omega^{2,-1,1}(M\qu[1]G),\qquad\omega^\red_{2,0,0}=\pr_1^*\omega-\pr_2^*\eta\in\Omega^{2,0,0}(M\qu[1]G),
\]
and the symplectic pairing at a point $x\in\mu^{-1}(1)$ is
\[
\begin{tikzcd}[column sep=large]
\bT_x(M\qu[1]G)\ar[d,"(\omega^\red)_\bu^\sharp"']\colon&\g\ar[d,"\id"']\ar[r,"a"]&T_xM\ar[r,"-T\mu"]\ar[d,"\omega^\sharp"']&\g^*\ar[d,"-\id"]\\
\bT^*_x(M\qu[1]G)\colon&\g\ar[r,"-T^*\mu"]&T_x^*M\ar[r,"a^*"]&\g^*.
\end{tikzcd}
\]
\end{theorem}

\begin{remark}
If the map $\psi$ in~\eqref{eq:complete} is not a global diffeomorphism, the map $\Exp\colon\g^*\to G^*$ is defined only in a neighbourhood of the origin and therefore the fiber replacement is not as good.  We suspect in this case it may still be possible to make sense of the $0$-shifted symplectic derived Lie groupoid modelling the reduced space.  
\end{remark}

\subsection{Epilogue}\label{subsection;epilogue}

We finish with some commentary and supplementary results, mostly stated without proof.

\subsubsection{The symplectic reduction cube}

Let $G$ be a Lie group and $(M,\mu,\omega)$ a hamiltonian $G$-manifold.  As in \S\,\ref{sec:hamred} let 
\[\huaZ=\huaZ(M,M\times\g^*,\mu\times\id)\]
be the derived zero fibre of $\mu$ and $M\qu G=G\ltimes\huaZ$ the symplectic quotient of $M$.  The Marsden-Weinstein reduction theorem can be interpreted as saying that in the case of regular reduction the natural map $\huaZ\to M^-\times M\qu G$ is a lagrangian embedding, in other words that the zero fibre defines a lagrangian correspondence $M\dashrightarrow M\qu G$.  This statement generalizes to the singular case as follows.  The \emph{symplectic reduction cube} is the commutative diagram of derived $1$-groupoids
\begin{equation}\label{equation;cube}
\begin{tikzcd}[column sep=small]
&\huaZ\arrow[dl]\arrow[rr]\arrow[dd]&&0\arrow[dl]\arrow[dd]\\
M\qu G\arrow[rr,crossing over]\arrow[dd]&&G\ltimes0\\
&M\arrow[dl]\arrow[rr,"\mu"near start]&&\g^*\arrow[dl]\\
G\ltimes M\arrow[rr]&&G\ltimes\g^*\arrow[from=uu,crossing over]\\
\end{tikzcd}
\end{equation}
The three arrows pointing to $G\ltimes\g^*$ are lagrangian morphisms (\S\,\ref{subsubsection;moment}) and the front face of the cube expresses the fact that $M\qu G$ is the intersection of the lagrangians $0$ and $G\ltimes M$ (Thm.~\ref{theorem;marsden-weinstein}).  The next proposition says that the three faces adjacent to $\huaZ$ are lagrangian correspondences. 

\begin{proposition}\label{proposition;cube}
The morphisms
\[
\ca{Z}\longto M^-\times M\qu G,\qquad\ca{Z}\longto0\times M,\qquad\ca{Z}\longto 0\times M\qu G
\]
define $0$-shifted lagrangian correspondences
\[
\ca{Z}\colon M\dashrightarrow M\qu G,\qquad\ca{Z}\colon0\dashrightarrow M,\qquad\ca{Z}\colon0\dashrightarrow M\qu G.
\]
Each of the faces of the cube~\eqref{equation;cube} is a cartesian square.
\end{proposition}

\begin{proof}[Outline of proof]
This is a groupoid version of \cite[\S\,2.2]{cal:crit}.  We verify here that the morphism $f\colon\ca{Z}\longto M^-\times M\qu G$, equipped with the zero isotropic structure, is lagrangian.  Let $x\in\mu^{-1}(0)$ and let $\N_x^\bu$ be the normal cone of $f$ at $x$.  Let $\Omega=-\omega\oplus\omega_\bu^\red$ be the symplectic form on $M^-\times M\qu G$.  We must show that the chain map $\lambda=\lambda^{0,f,\Omega}\colon\N_x^\bu\to\bT^{*,\bu}_x\ca{Z}[-1]$ is a quasi-isomorphism.  A computation shows that $\N_x^\bu$ is the three-term complex
\[
\begin{tikzcd}
T_xM\oplus\g\ar[r,"d_0"]&\g^*\oplus T_xM\oplus T_xM\ar[r,"d_1"]&\g^*
\end{tikzcd}
\]
with differential given by 
\[
d_0=\begin{pmatrix}
    T_x\mu&0\\-\id&0\\-\id&a_x
\end{pmatrix},\qquad
d_1=\begin{pmatrix}
    \id&0&T_x\mu
\end{pmatrix},
\]
where $T_x\mu\colon T_xM\to\g^*$ is the derivative of $\mu$ and $a_x\colon\g\to T_xM$ is the infinitesimal action at $x$.  The cohomology of the complex is $H^0=\ker(a_x)=\g_x$ (the stabilizer subalgebra of $x$), $H^1=T_xM/T_x(Gx)$ (the slice to the $G$-action at $x$), and $H^2=0$.
The complex $\bT^{*,\bu}_x\ca{Z}[-1]$ is the two-term complex
\[
\begin{tikzcd}
\g\ar[r,"d_0"]&T^*_xM
\end{tikzcd}
\]
with differential given by $d_0=T^*_x\mu$, the transpose of $T_x\mu$.  The cohomology of the complex is $H^0=\ker(T^*_x\mu)=\g_x$ and $H^1=T^*_xM/\omega^\sharp(T_x(Gx))$.  The chain map $\lambda$ has two nonzero components,
\[
\lambda_0=\begin{pmatrix}
    0&\id   
\end{pmatrix},\qquad
\lambda_1=\begin{pmatrix}
   0&-\omega^\sharp&\omega^\sharp
\end{pmatrix}.
\]
One verifies that $\lambda_0$ induces $\id\colon\g_x\to\g_x$ and  that $\lambda_1$ induces the isomorphism $\omega^\sharp\colon T_xM/T_x(Gx)\to T^*_xM/\omega^\sharp(T_x(Gx))$.
\end{proof}

\begin{remark}\label{remark;principal}
Let $\huaM$ be a derived manifold and $G\times\huaM\to\huaM$ a $G$-action on $\huaM$.  Then the action groupoid $G\ltimes\huaM$ is a derived Lie groupoid and the identity bisection $\huaM\to G\ltimes\huaM$ is a principal $G$-bundle in the $2$-category of derived Lie groupoids as in \cite{bnz}.  In particular the arrows from the back face to the front face of the cube~\eqref{equation;cube} are $2$-categorical principal $G$-bundles.
\end{remark}

\subsubsection{The shifting trick}

About Marsden-Weinstein reduction at nonzero orbits we can say the following.

\begin{lemma}
Let $G$ be a Lie group.  
\begin{enumerate}
\item\label{item;fixed-semisimple}
Let $\alpha\in\g^*$ and let $\orbit$ be the coadjoint orbit of $\alpha$.  The coadjoint action is weakly symplectically linearizable at $\orbit$ if $\alpha$ is fixed under the coadjoint action, or if $\g$ is reductive and $\alpha$ is semisimple.
\item\label{item;abelian-compact}
If $G$ is abelian or compact, the coadjoint action is weakly symplectically linearizable at every orbit.
\item\label{item;sl2}
Let $G=\mathbf{SL}(2,\R)$ and let $\alpha\ne0$ be nilpotent.  The coadjoint action of $G$ is not weakly linearizable at the orbit $\orbit$ of $\alpha$.
\end{enumerate}
\end{lemma}

\begin{proof}[Sketch of proof]
\eqref{item;fixed-semisimple}~Suppose $\alpha$ is $G$-fixed.  Then the normal bundle of $G\ltimes\orbit$ is $G\ltimes\g^*$.  The translation map $\Phi\colon G\ltimes\g^*\to G\ltimes\g^*$ defined by $\Phi(g,\beta)=(g,\beta+\alpha)$ is an automorphism of the symplectic groupoid $G\ltimes\g^*$ and defines a global symplectic linearization of $G\ltimes\g^*$ at the orbit $\orbit$.  If $\g$ is reductive and $\alpha$ semisimple, then the orbit $\orbit$ is closed in $\g^*$ and the statement follows from \cite[Prop.~1.2, Thm.~2.7]{luna;operations-differentiables}.  

\eqref{item;abelian-compact} follows immediately from~\eqref{item;fixed-semisimple}.

\eqref{item;sl2}~Without loss of generality we may assume $\alpha=\bigl(\begin{smallmatrix}0&0\\\pm1&0\end{smallmatrix}\bigr)$.  The centralizer $Z(\alpha)$ is the unipotent group consisting of all matrices of the form $\bigl(\begin{smallmatrix}\pm1&0\\t&\pm1\end{smallmatrix}\bigr)$ with $t\in\R$.  The centralizer acts trivially on the fibre $\g^*/T_\alpha\orbit$ of the normal bundle $\nu_0$ of the orbit.  It follows that the normal bundle $\nu_0$ is $G$-equivariantly trivial and that the stabilizer of every element of $\nu_0$ is conjugate in $G$ to $Z(\alpha)$.  However, there are elements $\bigl(\begin{smallmatrix}\epsilon&0\\\pm1&-\epsilon\end{smallmatrix}\bigr)$ in $\g^*$ arbitrarily close to $\alpha$ whose centralizer is not isomorphic to $Z(\alpha)$, which contradicts the weak linearizability of the action.  
\end{proof}

Let $G$ be a Lie group and $(M,\mu,\omega)$ a hamiltonian $G$-manifold.  The following result is familiar in the regular case.

\begin{theorem}[shifting trick]\label{theorem;shifting-trick}
Let $\orbit$ be a coadjoint orbit of $G$.  Suppose the coadjoint action is weakly symplectically linearizable at $\orbit$.  There exists a symplectic Morita equivalence $M\qu[\orbit]G\to(M\times\orbit^-)\qu[0]G$, where $\orbit^-$ denotes $\orbit$ equipped with the opposite symplectic structure.    
\end{theorem}

\begin{proof}[Sketch of proof]
Let $(\Phi_\bu\colon\huaU_\bu\to\Gamma_\bu,\Omega_\lin,\eta)$ be linearization data at $\orbit$.  Here $\huaU_\bu$ is an open subgroupoid of $G\ltimes\ca{Z}(\nu_0,\pi_0^*\nu_0,\epsilon_{\nu_0})$ ($\nu_0$ being the normal bundle of $\orbit$) and $\Gamma_\bu=G\ltimes\g^*$ is the coadjoint action groupoid.  The base map of the morphism $\Phi_\bu$ is a tubular neighbourhood embedding $\Phi_0\colon U_0\to\g^*$ and $\Phi_1\colon U_1\to G\times\g^*$ is of the form 
\[\Phi_1(g,u)=(\gamma(g,u),\Phi_0(u))\]
for all $(g,u)\in U_1\subseteq G\times U_0$, where $\gamma\colon U_1\to G$ is a map satisfying $\Phi_0(gu)=\gamma(g,u)\Phi_0(u)$ for all $(g,u)\in U_1$.  Let $m_t$ be multiplication by $t$ on the vector bundle $\nu_0$ and for $0<t\le1$ put
\[\Phi_0^{(t)}(u)=\pi_0(u)+\frac1t\bigl(\Phi_0(m_tu)-\pi_0(u)\bigr).\]
Then $\Phi_0^{(0)}=\lim_{t\to0}\Phi_0^{(t)}$ is the normal derivative of $\Phi_0$ and (after shrinking $\huaU_\bu$ if necessary) the family of maps $\Phi_0^{(t)}\colon U_0\to\g^*$ defines an isotopy from $\Phi_0^{(0)}$ to $\Phi_0^{(1)}=\Phi_0$.  For $0\le t\le1$ we put
\[
\Phi_1^{(t)}(g,u)=\bigl(\gamma(g,u),\Phi_0^{(t)}(u)\bigr),\quad\Omega_\lin^{(t)}=\bigl(\Phi_1^{(t)}\bigr)^*(\Phi_1^*)^{-1}\Omega_\lin,\quad\eta^{(t)}=\bigl(\Phi_0^{(t)}\bigr)^*(\Phi_0^*)^{-1}\eta,
\]
so obtaining a family of weak symplectic linearization data at $\orbit$.  Thus we may assume without loss of generality that $\Phi_0$ is (the restriction to $U_0$ of) a fibrewise affine map $\nu_0\to\g^*$.  The symplectic quotient $M\qu[\orbit]G$ is the derived Lie groupoid $\huaG_\bu=M\times_{\g^*}\huaU_\bu$.  We have
\[\huaG_0=\ca{Z}(M\times_{\g^*}U_0,f_0^*\nu_0,f_0^*\epsilon_{\nu_0}),\]
where $M\times_{\g^*}U_0=\mu^{-1}(\Phi_0(U_0))$ and $f_0$ is the map
\[f_0=\pi_0\circ\pr_2\colon M\times_{\g^*}U_0\longto\orbit.\]
Similarly,
\[\huaG_1=\ca{Z}(M\times_{\g^*}U_1,f_1^*\nu_1,f_1^*\epsilon_{\nu_1}),\]
where $f_1$ is the map
\[
f_1=\pi_1\circ\pr_2\colon M\times_{\g^*}U_1\longto G\times\orbit.
\]
The symplectic quotient $(M\times\orbit^-)\qu[0]G$ is the derived Lie groupoid $G\ltimes\huaZ(\tilde\mu)$, where $\tilde\mu(x,\alpha)=\mu(x)-\alpha$ is the moment map for the $G$-action on $M\times\orbit^-$.  The symplectic Morita equivalence $\Psi_\bu\colon\huaG_\bu\to G\ltimes\huaZ(\tilde\mu)$ comes about as follows.  We define the derived morphism $\Psi_0\colon\huaG_0\to\huaZ(\tilde\mu)$ by the diagram
\[
\begin{tikzcd}
f_0^*\nu_0\ar[r,"\Psi_0"]\ar[d]&M\times\orbit\times\g^*\ar[d]
\\
M\times_{\g^*}U_0\ar[r,"\pr_1\times f_0"]\ar[u,bend right,"f_0^*\epsilon_{\nu_0}"']&M\times\orbit\ar[u,bend right,"\tilde\mu"'].
\end{tikzcd}
\]
Here $\Psi_0$ is the composition of the maps
\[
f_0^*\nu_0=(M\times_{\g^*}U_0)\times_\orbit\nu_0\longto M\times U_0\times\nu_0\longto M\times\orbit\times\g^*,
\]
where the first arrow is the natural inclusion and the second arrow sends $(x,u,v)\in M\times U_0\times\nu_0$ to $(x,\pi(u),\Phi_0(v)-\pi_0(v))$.  The derived morphism $\Psi_1\colon\huaG_1\to G\times\huaZ(\tilde\mu)$ is defined by a similar diagram,
\[
\begin{tikzcd}
f_1^*\nu_1\ar[r,"\Psi_1"]\ar[d]&G\times M\times\orbit\times\g^*\ar[d]
\\
M\times_{\g^*}U_1\ar[r,"\pr_1\times f_1"]\ar[u,bend right,"f_1^*\epsilon_{\nu_1}"']&G\times M\times\orbit\ar[u,bend right,"\tilde\mu"'],
\end{tikzcd}
\]
where $\Psi_1$ is the composition of the maps
\[
f_1^*\nu_1=(M\times_{\g^*}U_1)\times_\orbit\nu_1\longto M\times U_1\times\nu_1\longto G\times M\times\orbit\times\g^*,
\]
the first arrow being the natural inclusion and the second arrow sending $(x,(g,u),v)\in M\times U_1\times\nu_1$ to $(g,x,\pi_0(u),\Phi_0(v)-\pi_0(v))$.
\end{proof}

\subsubsection{The stratification}

Let $G$ be a Lie group and let $(M,\mu,\omega)$ be a hamiltonian $G$-space.  Suppose the $G$-action on $M$ is proper.  Then for each closed subgroup $H$ of $G$ the orbit type stratum
\[M_{(H)}=\{\,x\in M\mid\text{$G_x$ is conjugate to $H$}\,\}\]
is a submanifold of $M$.  The main result of~\cite{SjLe91} states that the intersection
\[Z_{(H)}=\{\,x\in M_{(H)}\mid\mu(x)=0\,\}\]
is likewise a submanifold of $M$, that the quotient 
\[S_{(H)}=Z_{(H)}/G\]
is a manifold, and that the form $\omega|_{Z_{(H)}}$ descends to a symplectic form $\omega_{(H)}$ on $S_{(H)}$.  Thus the coarse quotient space $\mu^{-1}(0)/G$ of the derived symplectic quotient $M\qu G$
has a natural stratification into symplectic manifolds.  The following theorem explains how the symplectic strata $S_{(H)}$ relate to the derived symplectic quotient.  Recall (Thm.~\ref{theorem;marsden-weinstein}) that $M\qu G$ is the derived Lie groupoid $G\ltimes\huaZ$, where $\huaZ$ is the derived manifold $\huaZ(M,M\times\g^*,\id\times\mu)$.  Let $i\colon G\ltimes Z_{(H)}\to G\ltimes\huaZ$ be the inclusion and let $p\colon G\ltimes Z_{(H)}\to S_{(H)}$ be the projection to (the identity groupoid of) $S_{(H)}$.

\begin{theorem}\label{theorem;strata}
Let $G$ be a Lie group and let $(M,\mu,\omega)$ be a proper hamiltonian $G$-space.  For every closed subgroup $H$ of $G$ the morphism $p\times i\colon G\ltimes Z_{(H)}\to S_{(H)}^-\times M\qu G$, equipped with the zero isotropic structure, defines a $0$-shifted lagrangian correspondence $S_{(H)}\dashrightarrow M\qu G$.
\end{theorem}

\begin{proof}[Sketch of proof]
It follows from
\[
-i^*\omega_{(H)}+p^*\omega^\red_\bu=-\omega|_{Z_{(H)}}+\omega|_{Z_{(H)}}=0
\]
that the morphism $f=p\times i$ is isotropic.  It remains to show that the cochain map 
\[
\lambda=\lambda^{0,f,-\omega_{(H)}\oplus\omega_\bu^\red}_x\colon
\N^{f,\bu}_x\longto\bT^{*,\bu}_x(G\ltimes Z_{(H)})[-1]
\]
is a quasi-isomorphism for all $x$ with $\mu(x)=0$ and $G_x=H$.  To calculate this map we use the local model for hamiltonian $G$-spaces of \cite{guillemin-sternberg;normal}, according to which $x$ has a $G$-invariant open neighbourhood in $M$ that is isomorphic to a neighbourhood of the zero section of the homogeneous vector bundle $E=G\times^H(\m^*\times V)$.  Here $\m=T_{\bar1}(G/H)=\g/\h$ and $V$ is the symplectic slice
\[
V=\bigl(T_x(G\cdot x)\bigl)^\omega\big/\bigl((T_x(G\cdot x))^\omega\cap T_x(G\cdot x)\bigr).
\]
Replacing $M$ with the bundle $E$, we find $Z_{(H)}=G/H\times V^H$ and $S_{(H)}=V^H$, so $T_xZ_{(H)}=\m\oplus V^H$ and $T_{\bar{x}}S_{(H)}=V^H$, where $\bar{x}=p(x)$ and $V^H$ is the subspace of $H$-fixed vectors.  Using Thm.~\ref{theorem;marsden-weinstein}\eqref{item;tangent} we obtain the following diagram, where the top row is the normal complex (of amplitude $[-1,2]$), the bottom row is the cotangent complex $\bT^{*,\bu}_x(G\ltimes Z_{(H)})[-1]$ (of amplitude $[1,2]$), the vertical map is $\lambda$, and $\pr_\m\colon\g\to\m$ is the projection.
\[
\begin{tikzcd}[ampersand replacement=\&,column sep=huge,row sep=huge]
\g\ar[r,"{\Bigl(\begin{smallmatrix}\pr_\m\\0\\\id\end{smallmatrix}\Bigr)}"]
\ar[d]\&
\m\oplus V^H\oplus\g\ar[r,"{\biggl(\begin{smallmatrix}0&\id&0\\\id&0&-\pr_\m\\0&0&0\\0&\id&0\end{smallmatrix}\biggr)}"]
\ar[d]\&
V^H\oplus\m\oplus\m^*\oplus V^H\ar[r,"{(\begin{smallmatrix}0&0&-\pr_\m^*&0\end{smallmatrix})}"]
\ar[d,"{\Bigl(\begin{smallmatrix}0&0&\id&0\\\omega_{(H)}^\sharp&0&0&-\omega_{(H)}^\sharp\end{smallmatrix}\Bigr)}"']\&
\g^*\ar[d,"\id"]
\\
0\ar[r]\&
0\ar[r]\&
\m^*\oplus(V^H)^*\ar[r,"{(\begin{smallmatrix}-\pr_\m^*&0\end{smallmatrix})}"]\&
\g^*
\end{tikzcd}
\]
A straightforward verification shows that $\lambda$ is a quasi-isomorphism.
\end{proof}

\appendix

\section{Stalkwise weak equivalence}\label{app:comb}

In this appendix we provide proofs of Prop.~\ref{prop:w-eq-comb} and Cor.~\ref{cor:w-eq-comb}.

Given a morphism $f_\bu\colon X_\bu\rightarrow Y_\bu$ in $\gpd{\Cat, \covers}$, we can form the pullback diagram
\[
\begin{tikzcd}
X_\bu\times_{f, Y_\bu, d_0}Y_\bu^I\ar[d,"\pr_1"']\ar[r,"\pr_2"]&Y_\bu^I\ar[d,"d_0"]\\
X_\bu\ar[r,"f"]&Y_\bu.
\end{tikzcd}    
\]
It follows from \cite[Lemma 2.4, Prop.~2.5, Thm.~7.1]{Rogers-Zhu:2016} that $f_\bu\colon  X_\bu\to Y_\bu$ is a stalkwise weak equivalence if and only if the composition
\begin{equation} \label{eq:pf}
p(f)\colon X_\bu\times_{f, Y_\bu, d_0 }Y_\bu^I \xto{\pr_2} Y_\bu^I \xto{d_1} Y_\bu 
\end{equation} 
is a hypercover. 

It is easy to see that each level of the left hand side of \eqref{eq:pf} is 
\[
(X\times_{Y} Y^I)_k=X_k \times_{Y_k} (Y^I)_k = \Hom(\Delta[k]\xto{\delta^0} \Delta[1]\times\Delta[k], X\xto{f}Y),
\]
where $\delta^0$ is the coface map $\Delta[0]\to \Delta[1]$. Similarly, we see that 
\begin{align*}
\Hom(\partial\Delta[k] \to \Delta[k], p(f))&=\partial_k (X) \times_{\partial_k(Y)} \Hom(\partial\Delta[k]\times \Delta[1], Y) \times_{\partial_k(Y)} Y_k\\
&=\partial_k (X)\times_{\partial_k(Y)}  \Hom( \Delta[k] \sqcup_{0\times \partial \Delta[k]} \Delta[1]\times \partial\Delta[k], Y).
%
%
\end{align*}
    Thus, by Cor.~\ref{cor:hypercover-n} the map $f_\bu$ is a stalkwise weak equivalence if and only if the natural map
    \begin{equation} \label{eq:w-eq-first-step}
        \Hom\bigl(\Delta[k]\xto{d^0} \Delta[1]\times\Delta[k], X\xto{f}Y\bigr) \longto \partial_k (X)\times_{\partial_k(Y)}  \Hom\bigl( \Delta[k] \sqcup_{0\times \partial \Delta[k]} \Delta[1]\times \partial\Delta[k], Y\bigr)    
    \end{equation}
    is a cover for all $k\le n-1$ and an isomorphism for $k= n$.  (For $n=\infty$ we understand this to mean that \eqref{eq:w-eq-first-step} is a cover for all $k$.) The right hand side of \eqref{eq:w-eq-first-step} is representable by Rem.\ \ref{rm:h-cover-rep} when $f_\bu$ is a stalkwise weak equivalence.  
    
    In the following we show that this is equivalent to the fact that the maps  \eqref{eq:w-eq-infty} 
   \begin{equation*} 
    r(f)_j:X_j\times_{Y_j, d_{j+1}} Y_{j+1} \xrightarrow{} \partial_j(X)\times_{\partial_j(Y)} \Lambda^{j+1}_{j+1}(Y)
\end{equation*}
described in Prop.~\ref{prop:w-eq-comb} are well-defined covers for $j\geq 0$ and an isomorphism for $j=n$, which proves the first statement of Cor.~\ref{cor:w-eq-comb}. Notice that the right hand side $R_j$ of \eqref{eq:w-eq-infty}, is not automatically representable because $\partial_j(X) $ might not be representable for $\infty$-groupoid objects.  But $R_0=Y_0$ is representable, thus $R_1 = L_0 \times_{r_1, R_0, r_1} L_0$ is also representable, where $L_j$ denotes the left hand side of \eqref{eq:w-eq-infty}. We should see $L_j$ as a bi-colored $(j, 0)$-simplex (with $j+1$ white points and $1$ black point) and $R_j$ as a bi-colored $(j, 0)$-horn  (with $j+1$ white points and $1$ black point and the white face missing). Then the representability follows inductively just as in the case of usual Kan condition \cite[Cor.2.5]{henriques}. This is explained in \cite[Prop.4.4]{blohmann-krishna-zhu} and \cite[Lemma 6.28]{li:thesis} for Lie 2-groupoids. 

\begin{figure}[ht]
\includegraphics{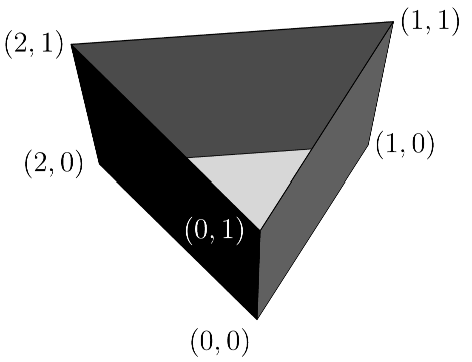}
\caption{The bucket
  $\Delta[2]\sqcup_{0\times\partial\Delta[2]}(\Delta[1]\times\partial\Delta[2])$.}
\label{figure;bucket-cylinder}
\end{figure}

The natural inclusion of the bucket shape into the cylinder
\begin{equation}\label{eq:bucket-inclusion}
        \Delta[k] \sqcup_{0\times \partial \Delta[k]} \Delta[1]\times \partial\Delta[k] \to \Delta[1] \times \Delta[k],
\end{equation}
is a collapsible extension (see \cite[Def.~3.5]{Rogers-Zhu:2016}) by \cite[Lemma 3.3.3]{hovey}. In fact, the statement therein says that \eqref{eq:bucket-inclusion} is an anodyne extension. However the construction done in the proof, as explained in details in  \cite[Lemma 7.14]{Rogers-Zhu:2016}, gives exactly a collapsible extension\footnote{This also is exactly the filtration given in the proof of \cite[Thm.~5.1]{Behrend-Getzler:2015}}
    \[ T_0 =  \Delta[k] \sqcup_{0\times \partial \Delta[k]} \Delta[1]\times \partial\Delta[k] \to T_1 \to \dots \to T_k \to T_{k+1} = \Delta[1] \times \Delta[k], \]
    where each step $T_{j+1} = T_j \sqcup_{\Horn{k+1}{j+1}} \Delta[k+1]$ is obtained by a push-out of a horn-extension for $j=0, \dots, k$. 
    
Thus the sequence of maps $\Hom(T_0, Y) \to \Hom(T_1, Y) \to \dots \to \Hom(T_{k+1}, Y)$ consists of covers for all $k$ and isomorphisms for $k\ge n$.  Since being a cover or isomorphism is stable under pullbacks and compositions\footnote{
The argument is identical to the proof of \cite[Thm.~5.1]{Behrend-Getzler:2015}, except for the step in \eqref{diag:j-j+1}, where the following modification is required:
The right vertical map is a horn projection instead of a face map. With this,  we are able to obtain acyclic fibrations on higher levels instead of having only covers. }, and the following diagrams are pullback diagrams
\begin{equation}\label{diag:j-j+1}
\begin{tikzcd}
\partial_k (X)\times_{\partial_k(Y)}\Hom(T_{j+1},Y)\ar[d]\ar[r]&Y_{k+1}\ar[d,"p^{k+1}_{j+1}"]\\
\partial_k (X)\times_{\partial_k(Y)}\Hom(T_{j},Y)\ar[r]&\Horn{k+1}{j+1}(Y),  
\end{tikzcd}
\end{equation}
we have that
\begin{equation}\label{eq:missing-a-horn}
\partial_k (X)\times_{\partial_k(Y)} \Hom(T_k, Y)  \longto    \partial_k (X)\times_{\partial_k(Y)}  \Hom( \Delta[k] \sqcup_{0\times \partial \Delta[k]} \Delta[1]\times \partial\Delta[k], Y) 
\end{equation}
is a cover for all $k$ and an isomorphism for $k\ge n$. Especially the left hand side of \eqref{eq:missing-a-horn} is representable. 
Then filling the last horn as in $T_k \to T_{k+1} =\Delta[1] \times \Delta[k]$, we obtain the following pullback diagram:
\[
\begin{tikzcd}
X_k\times_{Y_k} (Y^I)_k  \ar[r] \ar[d,"r'_k"']& X_k \times_{Y_k, d_{k+1}} Y_{k+1} \ar[d,"r_k=(\partial_k\text{,}p^{k+1}_{k+1} )"]\\
\partial_k (X)\times_{\partial_k(Y)}  \Hom(T_k , Y) \ar[r] & \partial_k (X)\times_{\partial_k(Y)} \Horn{k+1}{k+1} (Y).
\end{tikzcd}
\]
If $r_k=(\partial_k, p^{k+1}_{k+1})$ is a cover for all $0\le k\le n-1$ and an isomorphism for $k= n$, the pullback map $r'_k$ is also a cover for all $0\le k\le n-1$ and an isomorphism for $k= n$. Since the map \eqref{eq:w-eq-first-step} is the composition of \eqref{eq:missing-a-horn} and $r'_k$, we conclude that $f_\bu$ is a weak equivalence if $r_k$ is a cover for all $0\le k\le n-1$ and an isomorphism for $k= n$.  The reverse implication now follows from the following pullback diagram:
\[
\begin{tikzcd}
X_k\times_{Y_k} Y_{k+1} \ar[r] \ar[d,"r_k"'] & X_k \times_{Y_k} (Y^I)_k \ar[d,"\eqref{eq:w-eq-first-step}"]\\
       \partial_k (X)\times_{\partial_k(Y)} \Horn{k+1}{k+1}(Y) \ar[r] & \partial_k (X)\times_{\partial_k(Y)}   \Hom( \Delta[k] \sqcup_{0\times \partial \Delta[k]} \Delta[1]\times \partial\Delta[k], Y).
\end{tikzcd}
\]

It remains to prove the second statement of Cor.~\ref{cor:w-eq-comb}. Let $f_\bu:X_\bu\to Y\bu$ be a weak equivalence in $\mathsf{Gpd}_n[\Cat, \covers]$. Notice that for $k\geq n$ the map 
$$\tau:=d_{k+1}\circ (p^{k+1}_{k+1})^{-1}:\Lambda^{k+1}_{k+1}(Y)\to Y_k$$ is a cover since $d_{k+1}$ is a cover.
Chasing through the following diagram
\[
\begin{tikzcd}[
  column sep=large,
  row sep=large,
  cells={nodes={inner sep=2pt}}
]
X_k\times_{Y_k}Y_{k+1}
  \arrow[r, two heads, "\widetilde p"]
&
X_k\times_{Y_k}\Lambda^{k+1}_{k+1}(Y)
  \arrow[r, "\widetilde r_k"]
  \arrow[d, two heads, "\mathrm{pr}_1"']
  \arrow[dr, phantom, "\lrcorner", very near start]
&
\partial_k(X)\times_{\partial_k(Y)}\Lambda^{k+1}_{k+1}(Y)
  \arrow[r, "\mathrm{pr}_2"]
  \arrow[d, two heads, "{\mathrm{id}\times\tau}"']
  \arrow[dr, phantom, "\lrcorner", very near start]
&
\Lambda^{k+1}_{k+1}(Y)
  \arrow[d, two heads, "\tau"]
\\
{}
&
X_k
  \arrow[r, "{q(f)_k=(\partial_k^X,f_k)}"']
&
\partial_k(X)\times_{\partial_k(Y)}Y_k
  \arrow[r, "\mathrm{pr}_2"']
  \arrow[d,  "\mathrm{pr}_1"]
  \arrow[dr, phantom, "\lrcorner", very near start]
&
Y_k
  \arrow[d, "\partial_k^Y"]
\\
{}
&%
{}
&
\partial_k(X)
  \arrow[r, "\partial_k(f)"']
&
\partial_k(Y),
\end{tikzcd}
\]
\noindent it follows that $r_k= \widetilde r_k \circ \widetilde p$ is a cover (resp. an isomorphism) iff $\widetilde r_k$ is a cover (an isomorphism) iff $q(f)_k=(\partial_k^X, f_k)$ is a cover (an isomorphism). Thus in these cases, our condition that $r_k$ is a cover (resp. an isomorphism) can be simplified to $q(f)_k$ being a cover (resp. an isomorphism). This is exactly the condition \eqref{eq:hyper-n} in hypercover.

\section{Other pretopologies}\label{section;alternative}

A disadvantage of the pretopology $\covers_{\lsf}$ introduced in \S\,\ref{subsection;topology} is that in the absence of an implicit function theorem for derived manifolds it is hard to recognize when a fibration is locally split.  This does not matter for the purposes of this paper, since the fibrations of interest to us (namely the face maps of certain simplicial derived manifolds) are automatically locally split (by the degeneracy maps).  Nevertheless we wish to show in this appendix that the results of \S\,\ref{subsection;topology} go wrong if one relaxes the requirement that covers be locally split.  In particular we show that the pretopology of surjective fibrations, although it is subcanonical (Lemma~\ref{lem:eff-epi}), does \emph{not} meet our purposes.  On a positive note we point out that all the main results of \S\,\ref{subsection;topology} hold for the pretopology $\covers_\etale$ consisting of all surjective and essentially surjective \'etale fibrations (Thm.~\ref{theorem;etale-pretopology}).

\begin{lemma}\label{lemma;split}
Let $\covers$ be a subcanonical pretopology on $\DMfd$ with the property that the stalk functors $\ppt_{\huaX,x}\colon\Sh(\DMfd,\covers)\to\Set$ defined in \eqref{eq:dm-point} are points.  Then every cover in $\covers$ is locally split.
\end{lemma}

\begin{proof}
Let $c\colon\huaU\to\huaM$ be a cover in $\covers$.  Since $\covers$ is subcanonical, the presheaves $\yon(\huaU)$ and $\yon(\huaM)$ are sheaves.  We claim that for every point $\ppt$ of the category of sheaves $\Sh(\DMfd,\covers)$ the map
\begin{equation}\label{equation;surjective}
\ppt(c_*)\colon\ppt(\yon(\huaU))\longto\ppt(\yon(\huaM))
\end{equation}
induced by the sheaf morphism $c_*\colon\yon(\huaU)\to\yon(\huaM)$ is surjective.  To see this, let $f\in\yon(\huaM)(\huaX)$, i.e.\ $f\colon\huaX\to\huaM$ is a morphism, and form the pullback diagram
\[
\begin{tikzcd}
\huaX\times_\huaM\huaU\arrow[r,"\bar{f}"]\arrow[d,"\bar{c}"]&
\huaU\arrow[d,"c"]
\\
\huaX\arrow[r,"f"]&\huaM.
\end{tikzcd}
\]
Then $\bar{f}\in\yon(\huaM)(\huaX\times_\huaM\huaU)$ and
\[c_*(\bar{f})=c\circ\bar{f}=f\circ\bar{c}=\bar{c}^*(f).\]
This shows that $c_*\colon\yon(\huaU)\to\yon(\huaM)$ is a local surjection in the sense of \cite[Def.~2.7]{Rogers-Zhu:2016}, and hence an epimorphism of sheaves by \cite[Lemma~4.8]{Rogers-Zhu:2016}.  Because $\ppt$ is a point, it preserves colimits, and in particular it sends epimorphisms to epimorphisms, so~\eqref{equation;surjective} is surjective.  Since the functors $\ppt_{\huaX,x}$ are points it follows that $\ppt_{\huaX,x}(c_*)$ is surjective for all $(\huaX,x)$.  By Lemma~\ref{lem:stalkwise-surj-prop} this implies that $c$ is locally split.
\end{proof}

Here are some obvious necessary conditions for a morphism to be
locally split.

\begin{lemma}\label{lemma;locally-split}
Let $f\colon\ca{M}\to\ca{N}$ be a morphism of derived manifolds.  If $f$ is locally split, then
\begin{enumerate}
\item\label{item;top-split}
the map $\pi_0(\ca{M})\to\pi_0(\ca{N})$ induced by $f$ is locally
split as a map of topological spaces;
\item\label{item;submersion}
for every $y\in\ca{N}$ there exists $x\in f^{-1}(y)$ such that the
tangent map $T_x^kf\colon T_x^k\ca{M}\to T_y^k\ca{N}$ is surjective
for all $k\ge0$;
\item\label{item;coh-surjective}
for every $y\in\pi_0(\ca{N})$ there exists $x\in
f^{-1}(y)\cap\pi_0(\ca{M})$ such that the induced map in cohomology
$H^k(T_x^\bu f)\colon H^k(T_x^\bu\ca{M})\to H^k(T_y^\bu\ca{N})$ is
surjective for all $k\ge0$.
\end{enumerate}
\end{lemma}

\begin{proof}
Let $y\in\huaN$, let $s\colon\huaW\to\huaM$ be a local splitting of $f$ at $y$, and let $x=s(y)$.  The derivative $T^k_ys$ is a splitting of $T^k_xf$, so $T^k_xf$ is surjective.  Suppose $y\in\pi_0(\huaN)$.  Then $\pi_0(s)\colon\pi_0(\huaW)\to\pi_0(\huaM)$ defines a local splitting of the continuous map $\pi_0(f)$.  Moreover, $H^k(T^\bu_ys)$ is a splitting of $H^k(T^\bu_xf)$, so $H^k(T^\bu_xf)$ is surjective.
\end{proof}

In general the conditions of Lemma~\ref{lemma;locally-split} are not sufficient for $f$ to be locally split, but the following lemma shows that at non-classical points condition \ref{lemma;locally-split}\eqref{item;submersion} alone is sufficient for $f$ to be locally split.

\begin{lemma}\label{lemma;non-classical-split}
Let $\ca{M}=(M,\ca{S}_M^\bu,Q_M)$ and $\ca{N}=(N,\ca{S}_N^\bu,Q_N)$ be derived manifolds.
\begin{enumerate}
\item\label{item;constant}
Let $x\in\ca{M}$ be a non-classical point.  There exists a graded
chart at $x$ in which $Q_M$ is a constant vector field,
$Q_M=\pardif{}{\xi}$, where $\xi$ is a coordinate of degree $-1$.
\item\label{item;split}
Let $f\colon\ca{M}\to\ca{N}$ be a derived morphism and let
$x\in\ca{M}$ be a point such that $y=f(x)$ is not a classical point of
$\ca{N}$ and the tangent map $T_x^\bu f\colon T_x^\bu\ca{M}\to
T_y^\bu\ca{N}$ is surjective in every degree.  There exists a local
splitting $s$ of $f$ at $y$ with $s(y)=x$.
\end{enumerate}
\end{lemma}

\begin{proof}
\eqref{item;constant}~Let $\R[-1]=\ca{Z}(*,\R,\pardif{}{\tau})$ be the
$-1$-shifted real line.  Let $F\colon\R[-1]\times\ca{M}\to\ca{M}$ be
the flow of the cohomological vector field $Q_M$, i.e.\ the morphism
of graded manifolds characterized by the following two properties: the
restriction of $F$ to ${\{0\}\times\ca{M}}\cong\ca{M}$ is the
identity, and the vector fields $\pardif{}{\tau}$ and $Q_M$ are
$F$-related.  (See
e.g.\ \cite[Rem.~3.8]{cattaneo-schatz;supergeometry}.)
Choose a graded chart $\ca{U}$ centred at $x$ modelled on a graded
vector space $U_\bu$.  Choose a hyperplane $V_1$ in the vector space
$U_1$ complementary to the nonzero vector $Q_{M,x}$, and put $V_i=U_i$
for $i\ne1$.  The graded subspace $V_\bu$ of $U_\bu$ defines a graded
submanifold $\ca{V}$ of $\ca{U}$ of codimension~$1$ and the flow
restricts to a graded morphism
$f\colon\R[-1]\times\ca{V}\to\ca{U}\subseteq\ca{M}$ whose tangent map
at $x$ is bijective.  By the inverse function theorem for graded
manifolds (\cite[Thm. 4.30]{vysoky;global-graded}), after shrinking $\ca{V}$ if necessary, $f$ defines a graded
chart at $x$.  In this chart $\tau$ is a coordinate of degree $-1$ and
we have $Q_M=\pardif{}{\tau}$.

\eqref{item;split}~We have $Q_{N,y}\ne0$ and hence $Q_{M,x}\ne0$
because $Q_M\sim_fQ_N$.  By part~\eqref{item;constant} we can find a
graded chart $\ca{U}$ centred at $x$ modelled on a graded vector space
$U_\bu$ in which $Q_M=\pardif{}{\xi}$ is constant.  Choose a graded
subspace $V_\bu$ of $U_\bu$ such that $Q_{M,x}\in V_1$ and $V_\bu$ is
complementary to the kernel of $T_x^\bu f$.  Then $V_\bu$ defines a
derived submanifold $\ca{V}$ of $\ca{U}$ which contains $x$ and is
transverse to the fibres of $f$.  The restriction of $f$ to $\ca{V}$
is a derived morphism $g\colon\ca{V}\to\ca{N}$ whose tangent map at
$x$ is bijective.  By the inverse function theorem $g$ has a local
inverse $s\colon\ca{W}\to\ca{V}$, which is a derived morphism defined
in an open neighbourhood $\ca{W}$ of $y$ and is a local splitting
of~$f$.
\end{proof}

\subsection{Surjective fibrations}\label{sec:sf}

Surjective fibrations of derived manifolds satisfy axioms \ref{P1}-\ref{P4} of \S\,\ref{sec:pre-points} and therefore define a pretopology on the category $\DMfd$, which we denote by $\covers_\sfib$.  In Lemma~\ref{lem:eff-epi} we show that the pretopology $\covers_\sfib$ is subcanonical.  However, not all surjective fibrations are locally split 
(Ex.~\ref{ep:counter-ep}), so Lemma~\ref{lemma;split} tells us that the functors $\ppt_{\huaX,x}$ fail to be points of the category of sheaves $\Sh(\DMfd,\covers_\sfib)$ (see example below)!  This difficulty renders surjective fibrations useless for the purpose of building an iCFO of derived groupoids.  (The necessity of covers having local splittings appears to be overlooked in \cite{Zeng}.)

\begin{example}\label{example;not-point}
Let $s\colon\R\to\R$ be the identity map, viewed as a section of the trivial line bundle $\underline{\R}$ over $\R$.  Let $\ca{N}$ be the quasismooth derived manifold $\ca{Z}(\R,\underline{\R},s)$.  Let $f\colon\huaM\to\huaN$ be a surjective fibration, where $\huaM$ is a derived manifold with empty classical locus.  (For instance, we can take $\huaM=\huaP\times\huaN$ and $f=\pr_2$, where $\huaP$ is any derived manifold with empty classical locus, as in Ex.~\ref{ep:counter-ep}.)  The classical locus of $\huaN$ is $\pi_0(\ca{N})=\{0\}$.  Let $\huaX=\huaN$ and $x=0$.  Let $\huaU$ be an open neighbourhood of $x$ in $\huaX$.  There are no morphisms $\huaU\to\huaM$, so the stalk $\ppt_{\huaX,x}(\yon(\huaM))$ is the empty set.  Contrariwise, the stalk $\ppt_{\huaX,x}(\yon(\huaN))$ contains the germ of the identity morphism $\huaX\to\huaN$ and therefore is nonempty.  (In fact, $\ppt_{\huaX,x}(\yon(\huaN))$ is isomorphic to $\huaC^\infty_{\R,0}$, the algebra of germs of smooth functions on $\R$ at the origin.)  Thus the map $\ppt_{\huaX,x}(\yon(f))\colon\ppt_{\huaX,x}(\yon(\huaM))\to\ppt_{\huaX,x}(\yon(\huaN))$ fails to be surjective and therefore $\ppt_{\huaX,x}$ is not a point of $(\DMfd,\covers_\sfib)$.
\end{example}

\begin{lemma}\label{lem:eff-epi}
The pretopology $\covers_\sfib$ on the category of derived manifolds is subcanonical.
\end{lemma}

\begin{proof}
We prove that every surjective fibration of derived manifolds $f\colon\huaM\to\huaN$ is an effective epimorphism, i.e. 
\[
\begin{tikzcd}
\huaM\times_\huaN\huaM\ar[r,shift left=0.8ex,"\pr_1"]\ar[r,shift right=0.8ex,"\pr_2"']&\huaM\ar[r,"f"]&\huaN
\end{tikzcd}
\]
is a coequalizer diagram.  Let $g\colon\huaM\to\huaP$ be a morphism of derived manifolds such that $g\circ\pr_1=g\circ\pr_2$.  We must show that there exists a unique morphism of derived manifolds $\tilde{g}\colon\huaN\to\huaP$ such that $g=\tilde{g}\circ f$.  It follows from the submersion theorem for graded manifolds (\cite[Thm.~4.36]{vysoky;global-graded}) that there exists a unique morphism of \emph{graded} manifolds $\tilde{g}\colon N\to P$ such that $g=\tilde{g}\circ f$.  It remains to show that $\tilde{g}$ is a morphism of derived manifolds.  This is a local question, so by once again using the graded submersion theorem, we can reduce it to the case where $\huaM=\huaU\times\huaV$ is a product of two graded charts $\huaU$ and $\huaV$, $\huaN=\huaU$, $f\colon\huaU\times\huaV\to\huaU$ is the projection onto the first factor, and $\huaP$ is a graded chart $\huaW$.  Denoting the coordinates on $\huaU$ by $x_i$, $\xi_j$, on $\huaV$ by $y_k$, $\eta_l$, and on $\huaW$ by $z_p$, $\zeta_q$, we have local expressions for the cohomological vector fields $Q_N$ and $Q_P$,
\[Q_N=\sum_ja_j(x,\xi)\pardif{}{\xi_j},\qquad Q_P=\sum_qc_q(z,\zeta)\pardif{}{\zeta_q}.\]
Since $Q_M$ and $Q_N$ are $f$-related, we have $Q_M\circ f^\sharp(\xi_j)=f^\sharp\circ Q_N(\xi_j)$ for all $j$, which tells us that $Q_M$ is of the form
\[Q_M=\sum_ja_j(x,\xi)\pardif{}{\xi_j}+\sum_lb_l(x,\xi,y,\eta)\pardif{}{\eta_l}.\]
Similarly, the fact that $Q_M$ and $Q_P$ are $g$-related amounts to the identities
\begin{equation}\label{equation;g}
\sum_ja_j(x,\xi)\pardif{g^\sharp(\zeta_q)}{\xi_j}+\sum_lb_l(x,\xi,y,\eta)\pardif{g^\sharp(\zeta_q)}{\eta_l}=
g^\sharp(c_q)
\end{equation}
for all $q$.  The identity $g^\sharp=f^\sharp\circ\tilde{g}^\sharp$ yields
\[
g^\sharp(\zeta_q)=\tilde{g}^\sharp(\zeta_q),\qquad g^\sharp(c_q)=\tilde{g}^\sharp(c_q).
\]
In particular, the functions $g^\sharp(\zeta_q)$ do not depend on $y_k$ and $\eta_l$.  Substituting this into~\eqref{equation;g} we obtain
\[
\sum_ja_j(x,\xi)\pardif{\tilde{g}^\sharp(\zeta_q)}{\xi_j}=\tilde{g}^\sharp(c_q)
\]
for all $q$, which means that $Q_N$ is $\tilde{g}$-related to $Q_P$.  Thus $\tilde{g}$ is a morphism of derived manifolds.
\end{proof}

\subsection{\'Etale and coconnective fibrations}\label{sec:comp}

Let $\ca{M}=(M,\ca{S}_M^\bu,Q_M)$ and $\ca{N}=(N,\ca{S}_N^\bu,Q_N)$ be derived manifolds.  A morphism $f\colon\huaM\to\huaN$ is \emph{\'etale} if for every classical point $x\in\pi_0(\huaM)$ the tangent map $T^\bu_xf\colon T^\bu_x\huaM\to T^\bu_{f(x)}\huaN$ is a quasi-isomorphism of complexes, i.e.\ induces isomorphisms in cohomology
\begin{equation}\label{equation;quasi}
\begin{tikzcd}
H^k(T^\bu_xf)\colon H^k(T^\bu_x\huaM)\ar[r,"\cong"]&H^k(T^\bu_{f(x)}\huaN)
\end{tikzcd}
\end{equation}
for all $k\ge0$.  Let $f\colon\huaM\to\huaN$ be a surjective and essentially surjective \'etale fibration.  Lemma~\ref{lemma;locally-split} says that $f$ is locally split at all non-classical points of $\huaN$.  The (much harder) inverse function theorem for derived manifolds \cite[Cor.~3.2]{blx:etale-linfty} says that $f$ is also locally split at all classical points of $\huaN$.  Using these facts one can establish the following result.  As we will not use this result, we omit the proof, which is analogous to that of Thm.~\ref{pro:LSW-pret-dmfd}.

\begin{theorem}\label{theorem;etale-pretopology}
Surjective and essentially surjective \'etale fibrations serve as covers for a subcanonical pretopology $\covers_\et$ on the category $\DMfd$.  A presheaf of sets on $\DMfd$ is a sheaf with respect to $\covers_\et$ if and only if it is a sheaf with respect to $\covers_\lsf$ or $\covers_\open$. The stalk functors $\ppt_{\ca{X},x}$ defined in \eqref{eq:dm-point}  are points, the collection $\{\ppt_{\ca{X},x}\}$ is jointly conservative, and the pretopology $\covers_\et$ is locally stalkwise with respect to this collection.
\end{theorem}

A morphism $f\colon\huaM\to\huaN$ is \emph{coconnective} (\cite{getzler;private}) if it is essentially surjective and the condition~\eqref{equation;quasi} holds for all $k\ge1$.  Coconnectivity is not enough to guarantee that $f$ is locally split, not even if $f$ is also a surjective fibration.  Indeed, Ex.~\ref{example;coconnective} below gives an instance of a surjective coconnective fibration $f$ together with a classical point $y\in\pi_0(\huaN)$ such that $H^0(T^\bu_xf)\colon H^0(T^\bu_x\huaM)\to H^0(T^\bu_y\huaN)$ is not surjective for any $x\in\pi_0(\huaM)$ with $f(x)=y$.  It follows from Lemma~\ref{lemma;locally-split}\eqref{item;coh-surjective} that such an $f$ is not locally split.

\begin{example}\label{example;coconnective}
Let $a\colon\R\to\R$ and $b\colon\R^2\to\R$ be smooth functions.  Let $M=\R^2$ and define $s_M\colon M\to\R^2$ by $s_M(x,y)=(a(x),b(x,y))$, regarded as a section of the trivial bundle $E_M=\underline{\R}^2$.  Let $\huaM$ be the quasismooth derived manifold $\huaZ(M,E_M,s_M)$.  Let $N=\R$ and define $s_N\colon N\to\R$ by $s_N(x)=a(x)$, regarded as a section of the trivial bundle $E_N=\underline{\R}$.  Let $\huaN$ be the quasismooth derived manifold $\huaZ(N,E_N,s_N)$.  The projection onto the first factor $M\to N$, $E_M\to E_N$ defines a surjective fibration of derived manifolds $f\colon\huaM\to\huaN$.  Choose $a$ and $b$ such that $a(0)=b(0,0)=0$ and $a(x)\ne0$ for $x\ne0$.  Then $\pi_0(\huaM)=\{(0,0)\}$ and $\pi_0(\huaN)=\{0\}$ and $f$ is essentially surjective.  The tangent map of $f$ at the classical point $(0,0)$ is
\[
\begin{tikzcd}[ampersand replacement=\&,column sep=7em,row sep=large]
T^\bu_{(0,0)}\huaM\colon\ar[d,"T^\bu_{(0,0)}f"']\&\R^2\ar[r," {\Bigl(\begin{smallmatrix}a'(0)&0\\\pardif{b}{x}(0,0)&\pardif{b}{y}(0,0)\end{smallmatrix}\Bigr)}"]\ar[d," {(\begin{smallmatrix}1&0\end{smallmatrix})}"']\&\R^2\ar[d,"{(\begin{smallmatrix}1&0\end{smallmatrix})}"]\\
T^\bu_0\huaN\colon\&\R\ar[r,"a'(0)"]\&\R
\end{tikzcd}
\]
Choose $a$ such that $a'(0)=0$ and $b$ such that $\pardif{b}{x}(0,0)=1$ and $\pardif{b}{y}(0,0)=0$.  For instance, take $a(x)=x^2$, $b(x,y)=x+y^2$.  Then $H^1(T_x^\bu f)$ is surjective, so $f$ is a surjective coconnective fibration, but $H^0(T_x^\bu f)=0$ is not surjective, so $f$ is not locally split.
\end{example}

\bibliographystyle{amsalpha}
\bibliography{bibz,morerefs}

\end{document}